\documentclass[11pt]{amsart}
\usepackage{amsmath}
\usepackage{amssymb}
\usepackage{amsthm}
\usepackage[all]{xy}
\usepackage{graphicx}
\usepackage[margin=1in]{geometry}
\usepackage{color}
\usepackage{hyperref}
\usepackage{verbatim}
\usepackage[mathscr]{eucal}
\usepackage{pb-diagram,pb-xy}

\newtheorem*{thmA}{Theorem A}
\newtheorem*{thmB}{Theorem B}

\numberwithin{equation}{section}

\newcommand{\R}{\ensuremath{\mathbb{R}}}
\newcommand{\N}{\ensuremath{\mathbb{N}}}

\newcommand{\Z}{\ensuremath{\mathbb{Z}}}

\newtheorem{theorem}{Theorem}[section]

\newtheorem{corollary}[theorem]{Corollary}

\newtheorem{proposition}[theorem]{Proposition}
\newtheorem{lemma}[theorem]{Lemma}
\newtheorem{claim}[theorem]{Claim}

\theoremstyle{definition} 
\newtheorem{defn}[theorem]{Definition}
\newtheorem{remark}[theorem]{Remark} 
\newtheorem{Notation}[theorem]{Notational Convention}

\newtheorem{Construction}{Construction}[section]

\newcommand{\pr}{{\mathrm{pr}}}
\newcommand{\Th}{{\mathsf T}\!{\mathsf h}}
\newcommand{\MT}{{\mathsf M}{\mathsf T}}

\newcommand{\mb}[1]{\mathbf{#1}}

\newcommand{\Cob}{{\mathbf{Cob}}}

\DeclareMathOperator*{\hofibre}{hofibre}
\DeclareMathOperator*{\colim}{colim\ }

\DeclareMathOperator*{\dist}{dist}
\DeclareMathOperator*{\inn}{in}
\DeclareMathOperator*{\Int}{Int}

\DeclareMathOperator*{\Diff}{Diff}

\DeclareMathOperator*{\Emb}{Emb}
\DeclareMathOperator*{\BDiff}{BDiff}

\DeclareMathOperator*{\Sub}{Sub}

\DeclareMathOperator*{\Maps}{Maps}

\DeclareMathOperator*{\op}{\text{op}}

\DeclareMathOperator*{\Id}{\text{Id}}

\DeclareMathOperator*{\Bun}{\mathrm{Bun}}

\DeclareMathOperator*{\Ob}{\text{Ob}}
\DeclareMathOperator*{\Mor}{\mathrm{Mor}}

\DeclareMathOperator*{\Ker}{Ker}

\DeclareMathOperator*{\Image}{Im}

\DeclareMathOperator*{\loc}{loc}

\DeclareMathOperator*{\hocolim}{hocolim}
\DeclareMathOperator*{\std}{std}

\DeclareMathOperator*{\locc}{loc}

\DeclareMathOperator*{\stb}{stb}
\DeclareMathOperator*{\cyl}{cyl}

\DeclareMathOperator*{\Morse}{Morse}
\DeclareMathOperator*{\Surg}{Surg}
\DeclareMathOperator*{\sdl}{sdl}
\DeclareMathOperator*{\Trc}{Trc}
\DeclareMathOperator*{\rg}{rg}
\DeclareMathOperator*{\tr}{tr}
\DeclareMathOperator*{\prm}{pr}

\DeclareMathOperator*{\mf}{mf}

\DeclareMathOperator*{\indexx}{index}
\DeclareMathOperator*{\Mfds}{Mfds}

\author{Nathan Perlmutter}
\address{Stanford University Department of Mathematics, Building 380, Stanford, California,  94305, USA}
\email{nperlmut@stanford.edu}

\title[Cobordism Categories and Parametrized Morse Theory]{Cobordism Categories and Parametrized Morse Theory}

\begin{document}
\maketitle

\begin{abstract}
Fix a tangential structure $\theta: B \longrightarrow BO(d+1)$ and an integer $k < d/2$.
In this paper we determine the homotopy type of a cobordism category $\Cob^{\mf, k}_{\theta}$, where morphisms are given by $\theta$-cobordisms $W: P \rightsquigarrow Q$ equipped with a choice a Morse function $h_{W}: W \longrightarrow [0, 1]$, with the property that all critical points $c \in W$ of $h_{W}$ satisfy: $k < \text{index}(c) < d-k+1$.
In particular, we prove that there is a weak homotopy equivalence $B\Cob^{\mf, k}_{\theta} \simeq\Omega^{\infty}\mb{hW}^{k}_{\theta}$, where $\mb{hW}^{k}_{\theta}$ is a certain Thom spectrum associated to the space of Morse jets on $\R^{d+1}$.

In the special case that $k  = -1$, the equivalence $B\Cob^{\mf, -1}_{\theta} \simeq\Omega^{\infty}\mb{hW}^{-1}_{\theta}$ follows from the work of 
Madsen and Weiss in \cite{MW 07}, used in their celebrated proof of the \textit{Mumford conjecture}. 
Following the methods of Madsen and Weiss we use the weak equivalence  $B\Cob^{\mf, k}_{\theta} \simeq\Omega^{\infty}\mb{hW}^{k}_{\theta}$ to give an alternative proof the ``high-dimensional Madsen-Weiss theorem'' of Galatius and Randal-Williams from \cite{GRW 14}, which identifies the homology of the moduli spaces, $\BDiff((S^{n}\times S^{n})^{\# g}, D^{2n})$, in the limit $g \to \infty$.
\end{abstract}

\setcounter{tocdepth}{1}
\tableofcontents
\vspace*{-5mm}

\section{Introduction} \label{section: introduction}
To state our results, we begin by defining a topological category whose objects are given by closed manifolds embedded in high-dimensional Euclidean space, and whose morphisms are given by embedded cobordisms equipped with a Morse function. 
Fix an integer $d \geq 0$ and a tangential structure $\theta: B \longrightarrow BO(d+1)$. 
\begin{defn} \label{defn: first morse cob cat}
Objects of the (non-unital) topological category $\Cob^{\mf}_{\theta}$ are given by closed $d$-dimensional submanifolds $M \subset \R^{\infty}$, equipped with a $\theta$-structure $\hat{\ell}_{M}: TM\oplus\epsilon^{1} \longrightarrow \theta^{*}\gamma^{d+1}$. 
The morphisms $M_{0} \rightsquigarrow M_{1}$ are given by pairs $(t, W)$ where $t \in (0, \infty)$ and $W \subset [0, t]\times\R^{\infty}$ is a compact submanifold equipped with a $\theta$-structure $\hat{\ell}_{W}: TW \longrightarrow \theta^{*}\gamma^{d+1}$, subject to the following conditions:
\begin{enumerate} \itemsep.2cm
\item[(i)] the height function, $W \hookrightarrow [0, t]\times\R^{\infty} \stackrel{\text{proj}} \longrightarrow [0, t]$, is a Morse function; 
\item[(ii)] $W\cap(\{0, t\}\times\R^{\infty}) = (M_{0}\times\{0\})\sqcup (M_{1}\times\{t\})$ as a $\theta$-manifold,
and furthermore the intersection is orthogonal.
\end{enumerate}
We will always denote the height function from (i) by $h_{W}: W \longrightarrow [0, t].$
By condition (ii) $W$ is a cobordism between $M_{0}$ and $M_{1}$. 
Composition in this category is defined in the usual way by concatenation of cobordisms. 
\end{defn}
The category $\Cob^{\mf}_{\theta}$ is topologized so that for any two objects $M_{0}, M_{1}$, there is a homotopy equivalence 
\begin{equation} \label{equation: homotopy type of the morphism space}
\Cob^{\mf}_{\theta}(M_{0}, M_{1}) \; \simeq \; \coprod_{[W]}\left(\Bun(TW, \theta^{*}\gamma^{d+1}; \hat{\ell}_{M_{0}}\sqcup\hat{\ell}_{M_{1}})\times\Morse(W; M_{0}, M_{1})\right)//\Diff(W, \partial W),
\end{equation}
where the disjoint union is taken over all diffeomorphism classes of compact manifolds $W$, whose boundary is equipped with an identification, $\partial W = M_{0}\sqcup M_{1}$. 
The space $\Morse(W; M_{0}, M_{1})$ is the space of all Morse functions $f: W \longrightarrow [0, 1]$ having $0$ and $1$ as regular values, with $f^{-1}(0) = M_{0}$ and $f^{-1}(1) = M_{1}$. 
The space $\Bun(TW, \theta^{*}\gamma^{d+1}; \hat{\ell}_{M_{0}}\sqcup\hat{\ell}_{M_{1}})$ consists of all $\theta$-structures on $W$ that agree with $\hat{\ell}_{M_{0}}\sqcup\hat{\ell}_{M_{1}}$ when restricted to the boundary.

Using the work of Madsen and Weiss from \cite{MW 07}, the homotopy type of the classifying space $B\Cob^{\mf}_{\theta}$ can be identified with the infinite loopspace of a familiar Thom spectrum. 
Let $G^{\mf}_{\theta}(\R^{\infty})$ denote the space of tuples $(V, \hat{\ell}, l, \sigma)$ where:
\begin{itemize} \itemsep.1cm
\item $V \subset \R^{\infty}$ is a $(d+1)$-dimensional vector subspace,
\item $\hat{\ell}$ is a $\theta$-orientation on $V$,
\item $l: V \longrightarrow \R$ is a linear functional,
\item $\sigma: V\otimes V \longrightarrow \R$ is a symmetric bilinear form,
\end{itemize}
with the property that $(l, \sigma)$ satisfies the \textit{Morse condition}: if $l = 0$, then $\sigma$ is non-degenerate.
Let $U_{d+1, \infty} \longrightarrow G^{\mf}_{\theta}(\R^{\infty})$ denote the canonical $(d+1)$-dimensional vector bundle and let $\mb{hW}_{\theta}$ denote the Thom spectrum associated to the $-(d+1)$-dimensional virtual bundle, $-U_{d+1, \infty} \longrightarrow G^{\mf}_{\theta}(\R^{\infty})$.
The theorem stated below follows from the work of Madsen and Weiss in \cite{MW 07}.
\begin{thmA}[Madsen-Weiss, 2002] 
For any tangential structure $\theta: B \longrightarrow BO(d)$, there is a weak homotopy equivalence, 
$B\Cob^{\mf}_{\theta} \simeq \Omega^{\infty-1}\mb{hW}_{\theta}.$
\end{thmA}

\begin{remark}
In \cite{MW 07}, Madsen and Weiss prove a weak homotopy equivalence $|\mathscr{W}_{\theta}| \simeq \Omega^{\infty-1}\mb{hW}_{\theta}$, where $|\mathscr{W}_{\theta}|$ is the representing space of a sheaf $\mathscr{W}_{\theta}$.
The cobordism category $\Cob^{\mf}_{\theta}$ is never officially uttered in their paper, but by importing an extra argument from \cite{GMTW 08} one can prove the weak equivalence $B\Cob^{\mf}_{\theta} \simeq |\mathscr{W}_{\theta}|$; this is explained in Sections \ref{subsection: localization sequence} and \ref{subsection: families of morse functions}.
\end{remark}

In this paper we prove a generalization of Theorem A for subcategories of $\Cob^{\mf}_{\theta}$ consisting of cobordisms $(t, W): M \rightsquigarrow N$ whose associated Morse function $h_{W}: W \longrightarrow [0, t]$ is subject to certain constraints. 
\begin{defn} \label{defn: subcategory k}
Let $k \geq -1$ be an integer.
The objects of the subcategory $\Cob^{\mf, k}_{\theta} \subset \Cob^{\mf}_{\theta}$ are given by those $M$ for which the 
map $\ell_{M}: M \longrightarrow B$ (which underlies the $\theta$-structure $\hat{\ell}_{M}$) is $k$-connected. 
The morphisms of $\Cob^{\mf, k}_{\theta}$ are given by those $(t, W): P \rightsquigarrow Q$ such that all critical points $c \in W$ of the associated Morse function $h_{W}: W \longrightarrow [0, t]$ satisfy: $k < \text{index}(c) < d-k+1.$
It is immediate from the definition that $\Cob^{\mf, -1}_{\theta} \cong \Cob^{\mf}_{\theta}$.
\end{defn}
Our main theorem identifies the homotopy type of the classifying space $B\Cob_{\theta}^{\mf, k}$ in the case that $k < d/2$.
We need to construct a generalization of the spectrum $\mb{hW}_{\theta}$. 
For $k \in \Z_{\geq -1}$, we define $G^{\mf}_{\theta}(\R^{\infty})^{k} \subset G^{\mf}_{\theta}(\R^{\infty})$ to be the subspace consisting of those $(V, \hat{\ell}, l, \sigma)$ subject to the following condition: if $l = 0$, then $\sigma$ is non-degenerate with $k < \text{index}(\sigma) < d-k+1$.
As before, $\mb{hW}^{k}_{\theta}$ is defined to be the Thom spectrum associated to the virtual bundle, $-U_{d+1, \infty} \longrightarrow G^{\mf}_{\theta}(\R^{\infty})^{k}$.
The main theorem of this paper is stated below.
\begin{thmB} 
Let $-1 \leq k < d/2$.
Suppose that the tangential structure $\theta: B \longrightarrow BO(d+1)$ is chosen so that $B$ satisfies Wall's finiteness condition $F(k)$ (see \cite{W 65}).
Then there is a weak homotopy equivalence, 
$
B\Cob^{\mf, k}_{\theta} \simeq \Omega^{\infty-1}\mb{hW}^{k}_{\theta}.
$
\end{thmB}
We note that the case of Theorem B for $k = -1$ is covered by Theorem A.

\begin{remark} In \cite{MW 07}, the weak homotopy equivalence 
$B\Cob^{\mf}_{\theta} \simeq \Omega^{\infty-1}\mb{hW}_{\theta}$ from Theorem A constitutes a major step in Madsen and Weiss' proof of the \textit{Mumford Conjecture}, which identifies the homology of the moduli space $\BDiff((S^{1}\times S^{1})^{\#g}, D^{2})$ in the limit as $g \to \infty$.
Proceeding in a way similar to Madsen and Weiss, we will use the weak equivalences $B\Cob^{\mf, k}_{\theta} \simeq \Omega^{\infty-1}\mb{hW}^{k}_{\theta}$ from Theorem B to recover the ``High-dimensional Madsen-Weiss theorem'' of Galatius and Randal-Williams from \cite{GRW 14}, which identifies the homology of $\BDiff((S^{n}\times S^{n})^{\#g}, D^{2n})$ in the limit as $g \to \infty$ for all $n \geq 2$. 
We give an outline of how this is done later in the introduction in Subsection \ref{subsection: homology fibrations and stable moduli spaces} after discussing another application of Theorem B.
\end{remark}


\subsection{Outline of the proof of Theorem B} \label{subsection: localization sequence}
The main technical ingredient in Madsen and Weiss' proof of Theorem A is Vassiliev's h-principle \cite{Va 94} applied to certain spaces of proper Morse functions.
We refer the reader to \cite[Section 4]{MW 07} for the precise formulation of how the $h$-principle is used and for the precise definition of the function spaces that they apply it to.
Now, there does not exist any version of Vassiliev's h-principle for such spaces of proper Morse functions $f: W \longrightarrow \R$ 
whose critical points $c \in W$ satisfy the condition $k < \text{index}(c) < d-k+1$. 
For this reason, our proof of Theorem B must follow a different strategy than the proof of Theorem A in \cite{MW 07}.

Our first step is to replace the classifying space $B\Cob^{\mf, k}_{\theta}$ with a weakly equivalent space of manifolds that is more flexible and easier to work with. 
\begin{defn} \label{defn: space of long manifolds}
For $k \in \Z_{\geq -1}$, 
the space $\mathcal{D}^{\mf, k}_{\theta}$ consists of $(d+1)$-dimensional submanifolds $W \subset \R\times\R^{\infty}$, equipped with a $\theta$-structure $\hat{\ell}_{W}: TW \longrightarrow \theta^{*}\gamma^{d+1}$, subject to the following conditions:
\begin{enumerate} \itemsep.2cm
\item[(i)] the map $W \hookrightarrow \R\times\R^{\infty} \stackrel{\text{proj}} \longrightarrow \R$ is a proper Morse function 
for which all critical points $c \in W$ satisfy:  $k < \text{index}(c) < d-k+1$;
\item[(ii)] the map $\ell_{W}: W \longrightarrow B$ is $k$-connected.
\end{enumerate}
The space of manifolds $\mathcal{D}^{\mf, k}_{\theta}$ is topologized by the same procedure used in \cite[Section 2]{GRW 09}.
In the case that $k = -1$, the space $\mathcal{D}^{\mf, -1}_{\theta}$ 
is a model for the representing space of the sheaf $\mathscr{W}_{\theta}$ considered by Madsen and Weiss in \cite{MW 07}. 
\end{defn}
Using an argument from \cite[Section 4]{GMTW 08}, it follows that for all $k$ there is a weak homotopy equivalence,
$B\Cob^{\mf, k}_{\theta} \simeq \mathcal{D}^{\mf, k}_{\theta}.$
Using a version of the Pontryagin-Thom construction from \cite{MW 07} we construct a weak map 
\begin{equation} \label{equation: pontryagin thom construction}
\mathcal{P}^{k}_{\theta}: \mathcal{D}^{\mf, k}_{\theta} \longrightarrow \Omega^{\infty-1}\mb{hW}^{k}_{\theta}.
\end{equation}
In the special case that $k = -1$, Madsen and Weiss prove that this map gives the weak homotopy equivalence $\mathcal{D}^{\mf, -1}_{\theta} \simeq \Omega^{\infty-1}\mb{hW}^{-1}_{\theta}$. 
Our strategy is to prove Theorem B by induction on the integer $k$, using Madsen and Weiss' result for $k = -1$ as the base case of the induction.
In order to get the induction argument running, we need to introduce one more tool.

We let $G^{\mf}_{\theta}(\R^{\infty})_{\loc} \subset G^{\mf}_{\theta}(\R^{\infty})$ denote the subspace consisting of those $(V, \hat{\ell}, l, \sigma)$ for which $l = 0$. 
For $k \in \Z_{\geq -1}$, 
we define
$G^{\mf}_{\theta}(\R^{\infty})_{\loc}^{\{k, d-k+1\}} \subset G^{\mf}_{\theta}(\R^{\infty})_{\loc}$
to be the subspace 
consisting of those $(V, \hat{\ell}, \sigma)$ for which $\text{index}(\sigma) \in \{k, d-k+1\}$.
\begin{defn} \label{defn: W-loc space}
For $k \in \Z_{\geq -1}$, the space $\mathcal{D}^{\mf, \{k, d-k+1\}}_{\theta, \loc}$ is defined to consist of pairs $(\bar{x}, \phi)$ where: 
\begin{enumerate} \itemsep.2cm
\item[(i)] $\bar{x} \subset \R\times\R^{\infty}$ is a $0$-dimensional submanifold with the property that $\bar{x}\cap\left((-\delta, \delta)\times\R^{\infty}\right)$ is finite for all $\delta \in (0, \infty)$;
\item[(ii)] $\phi: \bar{x} \longrightarrow G^{\mf}_{\theta}(\R^{\infty})_{\loc}^{\{k, d-k+1\}}$ is a map.
\end{enumerate}
There is a map 
\begin{equation} \label{equation: localization map}
\mb{L}_{k}: \mathcal{D}^{\mf, k-1}_{\theta} \; \longrightarrow \; \mathcal{D}^{\mf, \{k, d-k+1\}}_{\theta, \loc}
\end{equation}
defined by sending $W \in \mathcal{D}^{\mf, k}_{\theta}$ to the pair $(\bar{x}, \phi)$ where $\bar{x}$ is the set of critical points of degrees $k$ and $d-k+1$ for the height function $W \hookrightarrow \R\times\R^{\infty} \stackrel{\text{proj}} \longrightarrow \R$, and $\phi$ is the map which sends each $x \in \bar{x}$ to the tuple 
$\left(T_{x}W, \; \hat{\ell}_{W}|_{x}, \; d^{2}h_{W}|_{x}\right),$
where $d^{2}h_{W}|_{x}$ is the Hessian of $h_{W}$ at the critical point $x$. 
\end{defn}

Notice that for any $k$, the subspace $\mathcal{D}^{\mf, k}_{\theta}  \subset \mathcal{D}^{\mf, k-1}_{\theta}$ is contained in the fibre $\mb{L}_{k}^{-1}(\emptyset) \subset \mathcal{D}^{\mf, k-1}_{\theta}$ of the map $\mb{L}_{k}$ over the empty element $\emptyset \in \mathcal{D}^{\mf, \{k, d-k+1\}}_{\theta, \loc}$.
As a result, the inclusion $\mathcal{D}^{\mf, k}_{\theta}  \hookrightarrow \mathcal{D}^{\mf, k-1}_{\theta}$ induces a map 
\begin{equation} \label{equation: inclusion into homotopy fibre}
\mathcal{D}^{\mf, k}_{\theta} \longrightarrow \hofibre\left(\mathcal{D}^{\mf, k-1}_{\theta} \stackrel{\mb{L}_{k}} \longrightarrow \mathcal{D}^{\mf, \{k, d-k+1\}}_{\theta, \loc}\right).
\end{equation}
The main technical theorem that we prove in this paper is stated below.
\begin{theorem} \label{theorem: localization fibre sequence}
Suppose that $k < d/2$. 
Suppose that $B$ satisfies Wall's finiteness condition $F(k)$. 
Then the map from (\ref{equation: inclusion into homotopy fibre}) is a homotopy equivalence, and thus the following sequence of maps
$$
\xymatrix{
\mathcal{D}^{\mf, k}_{\theta} \ar@{^{(}->}[r] & \mathcal{D}^{\mf, k-1}_{\theta} \ar[rr]^{\mb{L}_{k}} && \mathcal{D}^{\mf, \{k, d-k+1\}}_{\theta, \loc}
}
$$
is a homotopy fibre sequence.
\end{theorem}
Let us now sketch how to use the fibre sequence from Theorem \ref{theorem: localization fibre sequence} to prove Theorem B. 
Let $\mb{hW}^{\{k, d-k+1\}}_{\theta, \loc}$ denote the suspension spectrum $\Sigma^{\infty}\left(G^{\mf}_{\theta}(\R^{\infty})^{\{k, d-k+1\}}_{\loc, \; +}\right)$.
For all $k$ there is a \textit{scanning map}
\begin{equation} \label{equation: scanning map}
\mathcal{P}^{\{k, d-k+1\}}_{\loc}: \mathcal{D}^{\mf, \{k, d-k+1\}}_{\theta, \loc} \; \longrightarrow \; \Omega^{\infty-1}\mb{hW}^{\{k, d-k+1\}}_{\theta, \loc},
\end{equation}
and it follows from the \textit{Barrat-Priddy-Quillen} theorem (or \cite[Theorem 3.12]{GRW 09}) that this scanning map is a weak homotopy equivalence for all $k$. 
Following \cite[Section 3]{MW 07} there is a cofibre sequence of spectra, 
$
\xymatrix{
\mb{hW}^{k}_{\theta} \ar[r] & \mb{hW}^{k-1}_{\theta} \ar[r] & \mb{hW}^{\{k, d-k+1\}}_{\theta, \loc}.
}
$
Furthermore, the maps $\mathcal{P}^{\{k, d-k+1\}}_{\loc}$ and $\mathcal{P}^{k}$ can be modeled in such a way so that the following diagram is homotopy commutative,
\begin{equation} \label{equation: homotopy commutative scanning diagram}
\xymatrix{
\mathcal{D}^{\mf, k}_{\theta} \ar@{^{(}->}[r] \ar[d] & \mathcal{D}^{\mf, k-1}_{\theta} \ar[rr]^{\mb{L}_{k}} \ar[d] && \mathcal{D}^{\mf, \{k, d-k+1\}}_{\loc} \ar[d]_{\simeq} \\
\Omega^{\infty-1}\mb{hW}^{k}_{\theta} \ar[r] & \Omega^{\infty-1}\mb{hW}^{k-1}_{\theta} \ar[rr] && \Omega^{\infty-1}\mb{hW}^{\{k, d-k+1\}}_{\theta, \loc}.
}
\end{equation}
With this homotopy commutative diagram established, we may use Theorem \ref{theorem: localization fibre sequence} to prove Theorem B by induction on $k$. 
Indeed, in the case that $k = 0$, the vertical map $\mathcal{D}^{\mf, -1}_{\theta} \longrightarrow \Omega^{\infty-1}\mb{hW}^{-1}_{\theta}$ is a weak homotopy equivalence as proven by Madsen and Weiss. 
Since the right-vertical map is always a weak homotopy equivalence for all $k$ and since both rows are fibre sequences, it follows that the left-vertical map is a weak homotopy equivalence as well. 
Continuing by induction for all $k < d/2$ yields Theorem B. 

\subsection{Homology fibrations and stabilization} \label{subsection: homology fibrations and stable moduli spaces}
We now describe how to use Theorem B to recover the high-dimensional analogue of the Madsen-Weiss theorem proven by Galatius and Randal-Williams in \cite{GRW 14}.
We turn our attention to the case where $d = 2n$ and $k = n$. 
In this case the height function $h_{W}: W \longrightarrow \R$ associated to an element $W \in \mathcal{D}^{\mf, n}_{\theta}$ is a submersion. 
With $d = 2n$ and $k = n$, the sequence from the statement of Theorem \ref{theorem: localization fibre sequence} fails to be a fibre sequence. 
However, it turns out that we can recover a homological fibre sequence upon applying a certain `stabilization process' to the spaces $\mathcal{D}^{\mf, k}_{\theta}$.
In order to describe this stabilization process, we will need to specify a certain subspace of $\mathcal{D}^{\mf, k}_{\theta}$.

For what follows, fix a $d$-dimensional compact submanifold 
$L \subset [0, \infty)\times\R^{\infty-1}$ 
that agrees with $[0, \infty)\times\partial L$ near $\{0\}\times\R^{\infty-1}$. 
\begin{defn} \label{defn: long manifolds with standard strip}
Fix a $\theta$-structure $\hat{\ell}_{L}: TL\oplus\epsilon^{1} \longrightarrow \theta^{*}\gamma^{d+1}$.
For all $k \in \Z_{\geq -1}$, 
the subspace $\mathcal{D}^{\mf, k}_{\theta, L} \subset \mathcal{D}^{\mf, k}_{\theta}$ consists of those $W$ for which, 
$$W\cap\left(\R\times[0, \infty)\times\R^{\infty-1}\right) = \R\times L,$$ 
as a $\theta$-manifold.
\end{defn}
It can be shown whenever $k < d/2$, for all such manifolds $L$, the inclusion induces a weak homotopy equivalence,
$
\mathcal{D}^{\mf, k}_{\theta, L}  \simeq \mathcal{D}^{\mf, k}_{\theta}. 
$
Let $d = 2n$ and suppose that $L \neq \emptyset$. 
Let $P$ denote the boundary $\partial L$ and let $\hat{\ell}_{P}$ denote the restriction of $\hat{\ell}_{L}$ to $P = \partial L$.
\begin{defn} \label{defn: self cobordism stable}
Consider the manifold $S^{n}\times S^{n}$. 
Fix a $\theta$-structure $\hat{\ell}_{0}: T(S^{n}\times S^{n})\oplus\epsilon^{1} \longrightarrow \theta^{*}\gamma^{d+1}$ 
with the property that the underlying map $\ell_{0}: S^{n}\times S^{n} \longrightarrow B$ is null-homotopic. 
Choose an embedding $\phi: S^{n}\times S^{n} \longrightarrow (0, 1)\times\R^{\infty-1}$ with image disjoint from the manifold $(0, 1)\times P$. 
Let 
$$H(P) \subset [0, 1]\times\R^{\infty-1}$$ 
be the submanifold obtained by forming the connected sum of $[0, 1]\times P$ with $\phi(S^{n}\times S^{n})$ along an embedded arc connecting $[0, 1]\times P$ with $\phi(S^{n}\times S^{n})$.
Fix a $\theta$-structure $\hat{\ell}_{H(P)}$ on $H(P)$ that agrees with $\hat{\ell}_{P}$ on $\{0\}\times P$ and $\{1\}\times P$, and that agrees with $\hat{\ell}_{0}$ on the connected sum factor $S^{n}\times S^{n}$. 
We consider $H(P)$ to be a self-cobordism of the $\theta$-manifold $P$. 
\end{defn}
Let $\widehat{H}(P) \subset \R\times [0, \infty)\times\R^{\infty-1}$ be the submanifold defined by the union 
$$
\widehat{H}(P) \; = \; \R\times H(P)\bigcup_{\R\times(\{1\}\times P)}\R\times(L + e_{1})
$$
where $(L + e_{1}) \subset [1, \infty)\times\R^{\infty-1}$ denotes the manifold obtained by translating $L \subset [0, \infty)\times\R^{\infty-1}$ by one unit in the first coordinate.
Using $\widehat{H}(P)$, we define a map 
\begin{equation} \label{equation: stabilization by S-n S-n}
\--\cup\widehat{H}(P): \mathcal{D}^{\mf}_{\theta, L} \; \longrightarrow \; \mathcal{D}^{\mf}_{\theta, L}
\end{equation}
by sending $W \in \mathcal{D}^{\mf}_{\theta, L}$ to the union, 
$
\left(W\setminus(\R\times L)\right)\bigcup\widehat{H}(P),
$
and then rescaling the second coordinate of the ambient space.
Since $S^{n}\times S^{n}$ is $(n-1)$-connected, it follows that (\ref{equation: stabilization by S-n S-n}) restricts to a map 
$\mathcal{D}^{\mf, k}_{\theta, L} \; \longrightarrow \; \mathcal{D}^{\mf, k}_{\theta, L}$ for all $k \leq n = d/2$. 
The map $\--\cup \widehat{H}(P)$ may be iterated, and for each $k \leq d/2$ we define
$$
\mathcal{D}^{\mf, k, \stb}_{\theta, L} := \hocolim\left(\mathcal{D}^{\mf, k}_{\theta, L} \stackrel{\--\cup\widehat{H}(P)} \longrightarrow \mathcal{D}^{\mf, k}_{\theta, L}\stackrel{\--\cup\widehat{H}(P)} \longrightarrow \cdots \right).
$$
Now, for $k < d/2$ we prove that the map $\--\cup\widehat{H}(P): \mathcal{D}^{\mf, k}_{\theta, L} \longrightarrow \mathcal{D}^{\mf, k}_{\theta, L}$ is a weak homotopy equivalence, and thus 
we obtain the weak homotopy equivalence 
$
\mathcal{D}^{\mf, k, \stb}_{\theta, L} \simeq \mathcal{D}^{\mf, k}_{\theta, L}
$
when $k < d/2$. 
For $k = n = d/2$ however, this weak homotopy equivalence does not hold. 

Our main theorem relates $\mathcal{D}^{\mf, n, \stb}_{\theta, L}$ to the homotopy fibre of $\mb{L}_{n}: \mathcal{D}^{\mf, n-1, \stb}_{\theta, L} \longrightarrow \mathcal{D}^{\mf, \{n, n+1\}}_{\theta, \loc}$.
As before, the inclusion $\mathcal{D}^{\mf, n, \stb}_{\theta, L} \hookrightarrow \mathcal{D}^{\mf, n-1, \stb}_{\theta, L}$ factors through the fibre $\mb{L}_{n}^{-1}(\emptyset)$, and thus induces a map 
\begin{equation} \label{equation: map to fibre n}
\mathcal{D}^{\mf, n, \stb}_{\theta, L} \longrightarrow \hofibre\left(\mathcal{D}^{\mf, n-1, \stb}_{\theta, L} \stackrel{\mb{L}_{n}} \longrightarrow \mathcal{D}^{\mf, \{n, n+1\}}_{\theta, \loc}\right).
\end{equation}
Using the homological stability theorem from \cite{GRW 16} we prove the following theorem. 
\begin{theorem} \label{theorem: homology fibration after stabilization}
Let $d = 2n$ and suppose that $B$ satisfies Wall's finiteness condition $F(n)$. 
Then the map (\ref{equation: map to fibre n}) is an acyclic map.
\end{theorem}

By invoking the same argument given Section \ref{subsection: localization sequence} to derive Theorem B from Theorem \ref{theorem: localization fibre sequence}, we obtain the following corollary which is the analogue of Theorem B in the case that $d = 2n$ and $k = n$. 
\begin{corollary} \label{corollary: homological corollary for k = n}
Let $d = 2n$ and $k = n$. 
Suppose that $B$ satisfies Wall's finiteness condition $F(n)$. 
Then there is an acyclic map, 
$
\mathcal{D}^{\mf, n, \stb}_{\theta, L} \; \longrightarrow \; \Omega^{\infty-1}\mb{hW}^{n}_{\theta}.
$
\end{corollary}

The high-dimensional version of the Madsen-Weiss theorem, proven by Galatius and Randal-Williams in \cite{GRW 14}, can be recovered from Corollary \ref{corollary: homological corollary for k = n} by simply reinterpreting the spaces involved in the statement of the above corollary.
In particular, the spectrum $\mb{hW}^{n}_{\theta}$ is homotopy equivalent to the \textit{Madsen-Tillmann spectrum} $\Sigma^{-1}\MT\theta_{d} \simeq \mb{hW}^{n}_{\theta}$ and the space $\mathcal{D}^{\mf, n, \stb}_{\theta, L}$ is homotopy equivalent the \textit{Stable moduli space of $\theta$-manifolds} studied in \cite{GRW 14} and \cite{GRW 16}. 
We will show how to recover their result from the above corollary in Section \ref{section: stable moduli spaces}. 
We remark that our method of proof specializes to Madsen and Weiss' proof of the \textit{Mumford Conjecture} \cite{MW 07} in the $n = 1$ case. 

\subsection{Organization} \label{subsection: organization}
In Section \ref{section: sheaves and families of morse functions} we cover the basic definitions and state the main theorems for us to prove. 
In Section \ref{section: homotopy colimit decompositions} we define an alternative model for the spaces $\mathcal{D}^{\mf, k}_{\theta}$ which will facilitate the proofs of the main technical theorems of the paper. 
The proof that this newly defined space is weakly equivalent to $\mathcal{D}^{\mf, k}_{\theta}$ follows an argument from \cite{MW 07} almost verbatim, and thus we relegate its proof to the appendix.  
Sections \ref{section: parametrized surgery} through \ref{section: higher index handles} involve all of the technical work of the paper. 
The propositions proven in those sections are all stated in terms of the definitions given in Section \ref{section: homotopy colimit decompositions} and thus the agenda for those sections is set and described in Section \ref{section: homotopy colimit decompositions}. 
In Section \ref{section: stable moduli spaces} we show how to use our results to recover the theorem of Galatius and Randal-Williams from \cite{GRW 14}.
Appendix \ref{section: proof of theorem homotopy colimit theorem} is devoted to giving the proof of a technical result stated in Section \ref{section: homotopy colimit decompositions}.

\subsection{Acknowledgments} \label{subsection: Acknowledgments}
The author thanks Johannes Ebert for a careful reading of an earlier version of this paper and for providing a number of helpful critical remarks. 
The author also thanks Boris Botvinnik, Soren Galatius, and Fabian Hebestreit for may helpful conversations. 
The author was supported by an NSF Post-Doctoral Fellowship, DMS-1502657, at Stanford University.

\section{Spaces of Manifolds and the Pontryagin-Thom Construction} \label{section: sheaves and families of morse functions}
\subsection{Spaces of manifolds} \label{subsection: spaces of manifolds}
In this section we define some spaces to be used throughout the paper. 
We first recall the definition of a \textit{tangential structure}.
\begin{defn} \label{defn: tangential structure}
A tangential structure is a map $\theta: B \longrightarrow BO(m)$. 
A $\theta$-structure on a $l$-dimensional manifold $M$ (with $l \leq m$) is defined to be a bundle map $\hat{\ell}_{M}: TM\oplus\epsilon^{m-l} \longrightarrow \theta^{*}\gamma^{m}$, i.e.\ a fibrewise linear isomorphism.
The pair $(M, \hat{\ell}_{M})$ is called an $l$-dimensional $\theta$-manifold. 
\end{defn}

For what follows, fix an integer $d \in \Z_{\geq 0}$ and a tangential structure $\theta: B \longrightarrow BO(d+1)$.
Below we define a space of $(d+1)$-dimensional $\theta$-manifolds. 
\begin{defn} \label{defn: space of manifolds with boundary}
Let $N \in \Z_{\geq 0}$. 
Fix  a $(d-1)$-dimensional, closed, submanifold $P \subset \R^{\infty-1}$ and a $\theta$-structure $\hat{\ell}_{P}$ on $P$.
The space $\mathcal{D}_{\theta, P}$ consists of pairs $(W, \hat{\ell}_{W})$ where 
$W \subset \R\times(-\infty, 0]\times\R^{\infty-1}$
is a $(d+1)$-dimensional submanifold (not necessarily compact), and $\hat{\ell}_{W}$ is a $\theta$-structure on $W$, subject to the following conditions:
\begin{enumerate} \itemsep.2cm
\item[(i)] $\partial W = \R\times P = W\cap\left(\R\times\{0\}\times\R^{\infty-1}\right)$; 
\item[(ii)] $W$ agrees with $\R\times(-\infty, 0]\times P$ near $\R\times\{0\}\times\R^{\infty-1}$ as a $\theta$-manifold, where the space $\R\times(-\infty, 0]\times P$ is equipped with the $\theta$-structure induced by $\hat{\ell}_{P}$.
\item[(iii)] the map, 
$
\xymatrix{W \ar@{^{(}->}[r] & \R\times(-\infty, 0]\times\R^{\infty-1} \ar[rr]^{\ \ \ \ \ \ \text{proj.}} && \R,}
$
is a proper map.
We note that this condition implies that for any $\delta \in (0, \infty)$, the space $W\cap\left([-\delta, \delta]\times(-\infty, 0]\times\R^{\infty-1}\right)$ is compact.
\end{enumerate}
The space $\mathcal{D}_{\theta, P}$ is topologized by the same process that was employed in \cite[Section 2]{GRW 09}; we do not repeat their construction here. 
\end{defn}

We will also have to work with spaces of compact $d$-dimensional manifolds with prescribed boundary.
\begin{defn} \label{defn: space of manifolds (absolute)}
Let $N \in \Z_{\geq 0}$. 
Fix a $(d-1)$-dimensional closed submanifold $P \subset \R^{\infty-1}$, and a $\theta$-structure $\hat{\ell}_{P}$ on $P$. 
The space $\mathcal{N}_{\theta, P}$ consists of pairs $(M, \hat{\ell}_{M})$ where $M \subset (-\infty, 0]\times\R^{\infty-1}$ is a $d$-dimensional compact submanifold, and $\hat{\ell}: TM\oplus\epsilon^{1} \longrightarrow \theta^{*}\gamma^{d+1}$ is a $\theta$-structure, subject to the following conditions:
\begin{enumerate} \itemsep.2cm
\item[(i)] $\partial M = M\cap\left(\{0\}\times\R^{\infty-1}\right) = P$, 
\item[(ii)] $M$ agrees with $(-\infty, 0]\times P$ near $\{0\}\times\R^{\infty-1}$ as a $\theta$-manifold, where $(-\infty, 0]\times P$ is equipped with the $\theta$-structure induced by $\hat{\ell}_{P}$.
\end{enumerate}
The space $\mathcal{N}_{\theta, P}$ is topologized in the usual way by identifying,
\begin{equation} \label{equation: quotient construction for moduli}
\mathcal{N}_{\theta, P} \; \cong \; \coprod_{[M]}\left(\Emb(M, (-\infty, 0]\times\R^{\infty-1}; i_{P})\times\Bun(TM\oplus\epsilon^{1}, \theta^{*}\gamma^{d+1}; \hat{\ell}_{P})\right)/\Diff(M, P),
\end{equation}
where the union ranges over all diffeomorphism classes of compact manifolds $M$, equipped with an identification $\partial M = P$.
The space $\Emb(M, (-\infty, 0]\times\R^{\infty-1}; i_{P})$ consists of all neat embeddings $M \longrightarrow (-\infty, 0]\times\R^{\infty-1}$ that agree with the inclusion $i_{P}: P \hookrightarrow \R^{\infty-1}$ on $\partial M = P$. 
The space
$\Bun(TM\oplus\epsilon^{1}, \theta^{*}\gamma^{d+1}; \hat{\ell}_{P})$ consists of all $\theta$-structures on $M$ that agree with $\hat{\ell}_{P}$ on $\partial M$. 
The action of $\Diff(M, P)$ is given by the formula $((\phi, \hat{\ell}), f) \mapsto (\phi\circ f, f^{*}\hat{\ell})$. 
\end{defn}

Both of the spaces $\mathcal{D}_{\theta, P}$ and $\mathcal{N}_{\theta, P}$ depend on a choice of closed $(d-1)$-dimensional $\theta$-manifold $(P, \hat{\ell}_{P})$.
Latter on it will be important to let these spaces vary with the manifold $(P, \hat{\ell}_{P})$. 
It will be useful to specify the space of manifolds that $(P, \hat{\ell}_{P})$ is drawn from. 
\begin{defn} \label{defn: space of closed theta manifolds}
We define $\mathcal{M}_{\theta}$ to be the space consisting of pairs $(P, \hat{\ell}_{P})$ where $P \subset \R^{\infty-1}$ is a closed $(d-1)$-dimensional submanifold, and $\hat{\ell}_{P}$ is a $\theta$-structure on $P$. 
The space $\mathcal{M}_{\theta}$ is topologized as a quotient space similar to (\ref{equation: quotient construction for moduli}). 
\end{defn}

\begin{Notation} \label{Notation: brief notation for theta manifolds}
From here on out, for elements of the spaces $\mathcal{M}_{\theta}$, $\mathcal{N}_{\theta, P}$, and $\mathcal{D}_{\theta, P}$ we will use the notation $W := (W, \hat{\ell}_{W})$. 
Since the $\theta$-structure has $W$ as a subscript there is no real loss of information with this notational convention. 
For such an element, 
we will denote by $\ell_{W}: W \longrightarrow B$ the underlying map associated to the $\theta$-structure $\hat{\ell}_{W}$.

At the beginning of this section, we fixed an integer $d$ and tangential structure $\theta: B \longrightarrow BO(d+1)$. 
We will continue to use this convention throughout the whole paper.
The space $\mathcal{D}_{\theta, P}$ will always consist of $(d+1)$-dimensional manifolds, $\mathcal{N}_{\theta, P}$ will always consist of $d$-dimensional manifolds, and $\mathcal{M}_{\theta}$ will always consist of $(d-1)$-dimensional manifolds.
\end{Notation}

\subsection{Spaces of manifolds equipped with morse functions} \label{subsection: families of morse functions}
In this section we define the main spaces of interest.
Fix an integer $d \in \Z_{\geq 0}$ and a tangential structure $\theta: B \longrightarrow BO(d+1)$. 
We must fix some more notation first. 
\begin{Notation} \label{Notation: height function}
Let $P \in \mathcal{M}_{\theta}$.
For any element $W \in \mathcal{D}_{\theta, P}$, we will let 
$
h_{W}: W \longrightarrow \R
$
denote the function,
$\xymatrix{
W \ar@{^{(}->}[r] & \R\times(-\infty, 0]\times\R^{\infty-1} \ar[r]^{ \ \ \ \ \ \ \ \ \ \ \ \ \text{proj}} & \R.
}$
We refer to this function as the \textit{height function}.
For any subset $K \subseteq \R$, we write 
$
W|_{K} = W\cap h_{W}^{-1}(K)
$
if it is a manifold. 
If $\hat{\ell}$ is a $\theta$-structure on $W$ and $W|_{K}$ a submanifold of $W,$ then we 
 write $\hat{\ell}|_{K}$ for the restriction of $\hat{\ell}$ to $W|_{K}$.
\end{Notation}

\begin{defn} \label{defn: sheaf of fibrewise morse functions}
Let $P \in \mathcal{M}_{\theta}$. 
We define $\mathcal{D}^{\mf}_{\theta, P} \subset \mathcal{D}_{\theta, P}$ to be the subspace consisting of those $W$ for which the height function $h_{W}: W \longrightarrow \R$ is a Morse function. 
For each integer $k$, the subspace $\mathcal{D}^{\mf, k}_{\theta, P} \subset \mathcal{D}^{\mf}_{\theta, P}$ consists of those $W$ subject to the following further conditions: 
\begin{enumerate} \itemsep.2cm
\item[(i)] For all critical points $c \in W$ of $h_{W}: W \longrightarrow \R$, 
the index of $c$ satisfies the inequality $k < \text{index}(c) < d-k+1$. 
\item[(ii)] The map $\ell_{W}: W \longrightarrow B$ is $k$-connected. 
\end{enumerate}
\end{defn}

Using the above definition we define cobordism categories. 
\begin{defn} \label{defn: morse cobordism category}
Let $P \in \mathcal{M}_{\theta}$.
The (non-unital) topological category $\Cob^{\mf}_{\theta, P}$ has $\mathcal{N}_{\theta, P}$ for its space of objects. 
The morphism space is the following subspace of $\R\times\mathcal{D}^{\mf}_{\theta, P}$: A pair $(t, W)$ is a morphism if there exists an $\varepsilon > 0$ such that 
$$W|_{(-\infty, \varepsilon)}  =  (-\infty, \varepsilon)\times W|_{0} \quad \text{and} \quad W|_{(t - \varepsilon, \infty)}   =  (t-\varepsilon, \infty)\times W|_{t}$$
as $\theta$-manifolds.
The source of such a morphism is $W|_0$ and the target is $W|_t$, equipped with their respective restrictions of the $\theta$-structure $\hat{\ell}_{W}$ on $W$.
Composition is defined in the usual way by concatenation of cobordisms. 

Let $k$ be an integer.
The subcategory $\Cob^{\mf, k}_{\theta, P} \subset \Cob^{\mf}_{\theta, P}$ has for its objects those $M \in \Ob\Cob^{\mf}_{\theta, P}$ for which the map $\ell_{M}: M \longrightarrow B$ is $k$-connected. 
It has for its morphisms those $(t, W) \in \Mor\Cob^{\mf}_{\theta, P}$ such that $W$ is contained in the subspace $\mathcal{D}^{\mf, k}_{\theta, P} \subset \mathcal{D}^{\mf}_{\theta, P}$.
\end{defn}

This paper is concerned with the homotopy type of the classifying space $B\Cob^{\mf, k}_{\theta, P}$ for different values of $k$. 
The following proposition is proven in the same way as \cite[Theorems 3.9 and 3.10]{GRW 09}.
We omit the proof. 
\begin{proposition} \label{proposition: cobcat to long manifolds}
For all integers $k$ and $P \in \mathcal{M}_{\theta}$, there is a weak homotopy equivalence, 
$$
B\Cob^{\mf, k}_{\theta, P} \simeq \mathcal{D}^{\mf, k}_{\theta, P}.
$$
\end{proposition}

\subsection{The Pontryagin-Thom construction} \label{subsection: pontryagin thom}
Fix an integer $d \in \Z_{\geq 0}$ and a tangential structure $\theta: B \longrightarrow BO(d+1)$.
We will keep these choices fixed for the rest of the paper. 
Our main theorem concerns the homotopy type of the space $\mathcal{D}^{\mf, k}_{\theta, P}$ for $k < d/2$.

\begin{defn} \label{defn: morse Grassmannian and spectrum}
For integers $k$ and $m$, let $G^{\mf}_{\theta}(\R^{m})^{k}$ denote the space of tuples $(V, l, \sigma)$ where:
\begin{enumerate} \itemsep.2cm
\item[(i)] $V \subset \R^{m}$ is a $(d+1)$-dimensional linear subspace equipped with a $\theta$-orientation $\hat{\ell}_{V}$; 
\item[(ii)] $l: V \longrightarrow \R$ is a linear functional and $\sigma: V\otimes V \longrightarrow \R$ is a symmetric bilinear form subject to the following condition: if $l = 0$, then $\sigma$ is non-degenerate with $k < \text{index}(\sigma) < d-k+1$. 
\end{enumerate}
We use the spaces $G^{\mf}_{\theta}(\R^{m})^{k}$ to define a Thom spectrum.
For each $N \in \N$ we let 
$$U_{d+1, N} \longrightarrow G^{\mf}_{\theta}(\R^{d+1+N})^{k}$$ 
denote the canonical $(d+1)$-dimensional vector bundle, and we let $U^{\perp}_{d+1, N} \longrightarrow G^{\mf}_{\theta}(\R^{d+1+N})^{k}$ denote the orthogonal complement bundle, which has $N$-dimensional fibres. 
For each $N$, the bundle $U^{\perp}_{d+1, N+1}$ pulls back to the bundle $\epsilon^{1}\oplus U^{\perp}_{d+1, N}$ by the inclusion $G^{\mf}_{\theta}(\R^{d+1+N})^{k} \hookrightarrow G^{\mf}_{\theta}(\R^{d+1+N+1})^{k}$, and thus this induces a map of Thom spaces 
\begin{equation} \label{equation: spectrum structure maps}
S^{1}\wedge\Th(U^{\perp}_{d+1, N}) \cong \Th(\epsilon^{1}\oplus U^{\perp}_{d+1, N}) \; \longrightarrow \; \Th(U^{\perp}_{d+1, N+1}).
\end{equation}
We define $\mb{hW}^{k}_{\theta}$ to be the spectrum whose $(d+1+N)$th space is given by the Thom space $\Th(U^{\perp}_{d+1, N})$, and with structure maps given by (\ref{equation: spectrum structure maps}).

We let $G^{\mf}_{\theta}(\R^{d+1+\infty})^{k}$ denote the colimit of the spaces $G^{\mf}_{\theta}(\R^{d+1+N})^{k}$ taken as $N \to \infty$ and let $U_{d+1, \infty} \longrightarrow G^{\mf}_{\theta}(\R^{d+1+\infty})^{k}$ denote the canonical $(d+1)$-dimensional vector bundle. 
It follows easily that $\mb{hW}^{k}_{\theta}$ is homotopy equivalent to the Thom spectrum associated to the virtual bundle $-U_{d+1, \infty} \longrightarrow G^{\mf}_{\theta}(\R^{d+1+\infty})^{k}$, and thus the above definition agrees with the definition of $\mb{hW}^{k}_{\theta}$ given in the introduction. 
\end{defn}

We now construct a zig-zag of maps between the spaces $\mathcal{D}^{\mf, k}_{\theta, P}$ and $\Omega^{\infty-1}\mb{hW}^{k}_{\theta}$.
We will need to use a slight modification of the space $\mathcal{D}^{\mf, k}_{\theta, P}$.
\begin{defn} \label{defn: vert tube nbh sheaves}
Fix an integer $N \in \N$. 
We define $\widehat{\mathcal{D}}^{\mf, k, N}_{\theta, P} \subset \mathcal{D}^{\mf, k}_{\theta, P}$ to be the 
subspace consisting of those $W$ that satisfy the following conditions:
\begin{enumerate} \itemsep.2cm
\item[(i)] $W$ is contained in the subspace $\R\times(-\infty, 0]\times\R^{d+ N-1}$;
\item[(ii)] the exponential map $\exp: TW \longrightarrow \R\times\R^{d+N}$ has injectivity radius greater than or equal to $1$ with respect to the Euclidean metric on $\R\times\R^{d+N}$.
\end{enumerate}
\end{defn}
It follows by a standard argument using the \textit{tubular neighborhood theorem} that the natural map induced by the inclusions,
$
\displaystyle{\colim_{N\to\infty}}\widehat{\mathcal{D}}^{\mf, k, N}_{\theta, P} \stackrel{\simeq} \longrightarrow \mathcal{D}^{\mf, k}_{\theta, P},
$
is a weak homotopy equivalence. 
\begin{Construction} \label{Construction: pontryagin thom construction}
For each $N \in \N$ we construct a map 
\begin{equation} \label{equation: preliminary pontryagin thom map}
\widehat{\mathcal{P}}^{k, N}_{\theta}: \widehat{\mathcal{D}}^{\mf, k, N}_{\theta, P} \longrightarrow \Omega^{d+N}\Th(U^{\perp}_{d+1, N}).
\end{equation}
Let $W \in \widehat{\mathcal{D}}^{\mf, k, N}_{\theta, P}$. 
The $2$-jet of the morse function $h_{W}: W \longrightarrow \R$ equips the tangent space $TW$ with the data $(l, \sigma)$, where $l$ is a linear functional (the differential of $h_{W})$ and $\sigma$ is a symmetric bilinear form (the quadratic differential of $h_{W}$). 
Since $h_{W}$ is a Morse function it follows that $(l, \sigma)$ satisfies the Morse condition. 
Furthermore, we have $k < \text{index}(\sigma) < d-k+1$. 
It follows that the Gauss map for the vertical tangent bundle $TW$ yields a map 
$
\tau: W \longrightarrow G_{\theta}^{\mf}(\R^{d+1+N})^{k}.
$
The inclusion $W \subset \R\times(-\infty, 0]\times\R^{d+N-1}$ induces the bundle trivialization
$$
TW\oplus\nu_{W} \cong \epsilon^{d+1+N}
$$
where $\nu_{W}$ is the normal bundle of $W$. 
From this bundle trivialization it follows that the map $\tau$ is covered by a bundle map 
$
\hat{\tau}^{\perp}: \nu_{W} \longrightarrow U^{\perp}_{d+1, N}.
$
Let $U \subset \R\times\R^{d+N}$ be the geodesic tubular neighborhood of $W$ of radius $1$ (this exists by Definition \ref{defn: vert tube nbh sheaves}). 
This tubular neighborhood $U$ determines a collapse map 
$c_{U}: \R\times S^{d+N} \longrightarrow \Th(\nu_{W}),$
and composing with $\Th(\hat{\tau}^{\perp}_{\pi})$ we obtain,  
$
\R\times S^{d+N} \longrightarrow \Th(U^{\perp}_{d+1, N}).
$
Precomposing this map with the inclusion $\{0\}\times S^{d+N} \hookrightarrow \R\times S^{d+N}$ and then forming the adjoint yields an element,  
$
\widehat{\mathcal{P}}^{k, N}_{\theta}(W) \in \Omega^{d+N}\Th(U^{\perp}_{d+1, N}).
$
This is our definition of the map $\widehat{\mathcal{P}}^{k, N}_{\theta}$ from (\ref{equation: preliminary pontryagin thom map}).
We denote by 
\begin{equation} \label{equation: limiting pontryagin thom}
\widehat{\mathcal{P}}^{k}_{\theta}: \displaystyle{\colim_{N\to\infty}}\widehat{\mathcal{D}}^{\mf, k, N}_{\theta, P}  \longrightarrow \Omega^{\infty-1}\mb{hW}^{k}_{\theta}
\end{equation}
the map induced by the maps $\widehat{\mathcal{P}}^{k, N}_{\theta}$ in the limit $N \to \infty$.
\end{Construction}

The following theorem was proven by Madsen and Weiss in \cite{MW 07}.
\begin{theorem}[Madsen-Weiss 2007] \label{theorem: madsen-weiss + vassiliev}
Let $P \in \mathcal{M}_{\theta}$.
Suppose that $P$ is null-bordant as a $\theta$-manifold.
Then the map, $\widehat{\mathcal{P}}^{-1}_{\theta}: \displaystyle{\colim_{N\to\infty}}\widehat{\mathcal{D}}^{\mf, -1, N}_{\theta, P}  \longrightarrow \Omega^{\infty-1}\mb{hW}^{-1}_{\theta},$
is a weak homotopy equivalence.
\end{theorem}

We remark that Theorem A stated in the introduction is obtained by combining the above theorem with the weak homotopy equivalence $\mathcal{D}_{\theta, P}^{\mf, -1} \simeq \Cob^{\mf, -1}_{\theta, P}$ from Proposition \ref{proposition: cobcat to long manifolds}.
The next theorem below generalizes the above result to the case where $k < d/2$. 
\begin{theorem} \label{theorem: main theorem}
Let $k < d/2$ and suppose that $\theta: B \longrightarrow BO(d+1)$ is such that $B$ satisfies Wall's finiteness condition $F(k)$ (see \cite{W 65}).
Let $P \in \mathcal{M}_{\theta}$. 
Suppose that $P$ is null-bordant as a $\theta$-manifold (this includes the case $P = \emptyset$).
Then the map,  
$\widehat{\mathcal{P}}^{k}_{\theta}: \displaystyle{\colim_{N\to\infty}}\widehat{\mathcal{D}} ^{\mf, k, N}_{\theta, P}\longrightarrow \Omega^{\infty-1}\mb{hW}^{k}_{\theta},$
is a weak homotopy equivalence. 
\end{theorem}
Combining Theorem \ref{theorem: main theorem} with Proposition \ref{proposition: cobcat to long manifolds} implies Theorem B from the introduction.
Theorem \ref{theorem: main theorem} is the main result that we are after and its proof is the subject of the rest of the paper. 

\subsection{The Localization Sequence} \label{section: the localization sequence}
We proceed to construct the localization sequence that was discussed in Section \ref{subsection: localization sequence}. 
We begin with a preliminary definition.
\begin{defn}
For each integer $m$, we let $G_{\theta}^{\mf}(\R^{m})_{\loc} \subset G_{\theta}^{\mf}(\R^{m})$ denote the subspace consisting of those tuples $(V, l, \sigma)$ with the property that $l = 0$.
Such elements of $G_{\theta}^{\mf}(\R^{m})_{\loc}$ will be denoted by the pair $(V, \sigma)$ since the $l$ term is redundant. 
The space $G_{\theta}^{\mf}(\R^{\infty})_{\loc}$ is defined to be the direct limit of the spaces $G_{\theta}^{\mf}(\R^{m})_{\loc}$ taken as $m \to \infty$.
Let $k$ be an integer. 
The subspace $G_{\theta}^{\mf}(\R^{m})_{\loc}^{\{k, d-k+1\}} \subset G_{\theta}^{\mf}(\R^{m})_{\loc}$ consists of those $(V, \sigma)$ such that $\text{index}(\sigma) \in \{k, d-k+1\}$. 
As before,
$G_{\theta}^{\mf}(\R^{\infty})_{\loc}^{\{k, d-k+1\}}$ is defined to be the limiting space. 
\end{defn}

\begin{defn} \label{defn: W-loc}
We define $\mathcal{D}^{\mf}_{\theta, \loc}$ to be the space of tuples $(\bar{x}; (V, \sigma))$ where $\bar{x} \subset \R\times\R^{\infty}$ is a zero-dimensional submanifold (not necessarily compact), and $(V, \sigma): \bar{x} \longrightarrow G_{\theta}^{\mf}(\R\times\R^{\infty})_{\loc}$ is a map, subject to the following condition: for all $\delta > 0$, the set $\bar{x}\cap\left((-\delta, \delta)\times\R^{\infty})\right)$ is a finite set (or in other words is a compact $0$-manifold). 
We will need to define certain subspaces of $\mathcal{D}^{\mf}_{\theta, \loc}$.
For an integer $k$, we define $\mathcal{D}^{\mf, \{k, d-k+1\}}_{\theta, \loc}$ to be the space consisting of those $(\bar{x}; (V, \sigma))$ for which, 
$$(V, \sigma)(x) = (V(x), \sigma(x)) \in G_{\theta}^{\mf}(\R\times\R^{\infty})^{\{k, d-k+1\}}_{\loc} \quad \text{for all $x \in \bar{x}$.}$$
The spaces $\mathcal{D}^{\mf, \{k, d-k+1\}}_{\theta, \loc} \subset \mathcal{D}^{\mf}_{\theta, \loc}$ are topologized using the same process carried out in \cite[Section 2.1]{GRW 09}.
\end{defn} 
We have a map 
$\mb{L}: \mathcal{D}^{\mf}_{\theta, P} \longrightarrow \mathcal{D}^{\mf}_{\theta, \loc}$
defined by sending $W \in \mathcal{D}^{\mf}_{\theta, P}$ to the pair $\left(\bar{x}; (V, \sigma) \right)$ where:
\begin{itemize} \itemsep.2cm
\item $\bar{x} \subset W$ is the set of critical points of the height function $h_{W}: W \longrightarrow \R$; 
\item for each $x \in \bar{x}$, $V(x)$ is the tangent space $T_{x}W$ and $\sigma(x): T_{x}W\otimes T_{x}W \longrightarrow \R$ is the 
Hessian associated to the height function $h_{W}$. 
\end{itemize}
For each $k$ we have the map 
\begin{equation} \label{equation: product map}
\mb{L}_{k}: \mathcal{D}_{\theta, \locc}^{\mf, k-1} \longrightarrow \mathcal{D}^{\mf, \{k, d-k+1\}}_{\theta, \locc}
\end{equation}
defined by composing the map $\mb{L}$ with the projection $\mathcal{D}^{\mf, k-1}_{\theta, \loc} \longrightarrow \mathcal{D}^{\mf, \{k, d-k+1\}}_{\theta, \loc}$ that sends $(\bar{x}; (V, \sigma))$ to the element
$$\left(\bar{x}|_{\{k, d-k+1\}}; \; (V, \sigma)|_{\bar{x}|_{\{k, d-k+1\}}}\right) \in \mathcal{D}^{\mf, \{k, d-k+1\}}_{\theta, \locc},$$ 
where $\bar{x}|_{\{k, d-k+1\}} \subset \bar{x}$ is the subset consisting of all $x \in \bar{x}$ with $\text{index}(\sigma(x)) \in \{k, d-k+1\}$.
The theorem below is a restatement of Theorem \ref{theorem: localization fibre sequence} from the introduction. 
The proof of this theorem is carried out over the course of the whole paper.
\begin{theorem} \label{theorem: localization sequence}
Let $0 < k < d/2$ and let $P \in \mathcal{M}_{\theta, d-1}$ be null-bordant as a $\theta$-manifold (including the case that $P = \emptyset)$.
Suppose that $B$ satisfies Wall's finiteness condition $F(k)$. 
Then the sequence 
\begin{equation} \label{equation: localization sequence}
\xymatrix{
\mathcal{D}^{\mf, k}_{\theta, P} \ar[r] & \mathcal{D}^{\mf, k-1}_{\theta, P} \ar[rr]^{\mb{L}_{k} \ \ \ \ } && \mathcal{W}^{\mf, \{k, d-k+1\}}_{\theta, \locc}
}
\end{equation}
is a homotopy fibre sequence.
\end{theorem}
We will refer to the homotopy fibre sequence in (\ref{equation: localization sequence}) as the \textit{localization sequence in degree $k$}.
Below we show how to derive Theorem \ref{theorem: main theorem} (and hence Theorem B) using the localization sequence. 

\subsection{Theorem B from Theorem \ref{theorem: localization sequence}} \label{subsection: proof of theorem B}
As discussed in Section \ref{subsection: localization sequence}, we will prove Theorem \ref{theorem: main theorem} (and Theorem B) using Theorem \ref{theorem: localization sequence}. 
An outline for the proof of Theorem B was provided in the introduction. 
In this section we fill in the details.

We need to construct an infinite loopspace to map $\mathcal{D}_{\theta, \loc}^{\mf, \{k, d-k+1\}}$ to.
For each $N \in \N$, let $U^{\loc}_{d+1, N} \longrightarrow G^{\mf}_{\theta}(\R^{d+1+N})^{\{k, d-k+1\}}_{\loc}$ denote the restriction of the canonical bundle 
$$U_{d+1, N} \longrightarrow G^{\mf}_{\theta}(\R^{d+1+N})$$ 
to the subspace $G^{\mf}_{\theta}(\R^{d+1+N})^{\{k, d-k+1\}}_{\loc}$.
We let $U^{\loc, \perp}_{d+1, N} \longrightarrow G^{\mf}_{\theta}(\R^{d+1+N})^{\{k, d-k+1\}}_{\loc}$ denote the restriction of the orthogonal complement $U^{\perp}_{d+1, N}$. 
We define $\mb{hW}^{\{k, d-k+1\}}_{\theta, \loc}$ to be the spectrum whose $(d+1+N)$th space is given by the Thom space, 
$
\Th(U^{\loc}_{d+1, N}\oplus U^{\loc, \perp}_{d+1, N}). 
$
For each $N$, the $N$-th structure map is the map between Thom spaces induced by the bundle map
$$
U^{\loc}_{d+1, N}\oplus U^{\loc, \perp}_{d+1, N}\oplus\epsilon^{1} \longrightarrow U^{\loc}_{d+1, N+1}\oplus U^{\loc, \perp}_{d+1, N+1}
$$
that covers the inclusion, 
$$G^{\mf}_{\theta}(\R^{d+1+N})^{\{k, d-k+1\}}_{\loc} \hookrightarrow G^{\mf}_{\theta}(\R^{d+1+N+1})_{\loc}.$$
It is easy to see that $\mb{hW}^{\{k, d-k+1\}}_{\theta, \loc}$ is homotopy equivalent to the suspension spectrum, 
$$\Sigma^{\infty}\left(G^{\mf}_{\theta}(\R^{\infty})^{\{k, d-k+1\}}_{\loc, +}\right),$$
and thus it coincides with the spectrum discussed in the introduction.
For any $k$, we have a sequence of spectrum maps 
\begin{equation} \label{equation: cofibre sequence of thom spectra}
\xymatrix{
\mb{hW}^{k}_{\theta} \ar[r] & \mb{hW}^{k-1}_{\theta} \ar[r] & \mb{hW}^{\{k, d-k+1\}}_{\theta, \loc},
}
\end{equation}
where the first map is induced by the inclusion $G^{\mf}_{\theta}(\R^{d+1+N})^{k-1} \hookrightarrow G^{\mf}_{\theta}(\R^{d+1+N})^{k}$ (or rather the bundle map that covers it) and the second map is induced by the inclusion of Thom spaces 
$$
\Th\left(U^{\loc, \perp}_{d+1, N}\right) \; \hookrightarrow \; \Th\left(U^{\loc}_{d+1, N}\oplus U^{\loc, \perp}_{d+1, N}\right).
$$
By replicating the same arguments from \cite[Section 2]{MW 07} we obtain the following proposition:
\begin{proposition} \label{proposition: cofibre of the natural inclusion map}
The sequence of maps, 
$\mb{hW}^{k}_{\theta} \longrightarrow \mb{hW}^{k-1}_{\theta} \longrightarrow \mb{hW}^{\{k, d-k+1\}}_{\theta, \locc},$
is a cofibre sequence of spectra.
\end{proposition}
\begin{proof}[Proof sketch]
The complement of $G^{\mf}_{\theta}(\R^{d+1+N})^{\{k, d-k+1\}}_{\loc}$ in $G^{\mf}_{\theta}(\R^{d+1+N})^{k-1}$ is given by the subspace $G^{\mf}_{\theta}(\R^{d+1+N})^{k}$.
The normal bundle of $G^{\mf}_{\theta}(\R^{d+1+N})^{\{k, d-k+1\}}_{\loc}$ in $G^{\mf}_{\theta}(\R^{d+1+N})^{k-1}$
is isomorphic to the dual of the canonical bundle,
$U^{\loc}_{d+1, N} \longrightarrow G^{\mf}_{\theta}(\R^{d+1+N})^{\{k, d-k+1\}}_{\loc}.$
From this observation we obtain a cofibre sequence of spaces,
$G^{\mf}_{\theta}(\R^{d+1+N})^{k-1} \longrightarrow G^{\mf}_{\theta}(\R^{d+1+N})^{k} \longrightarrow \Th(U^{\loc}_{d+1, N}),$
which in turn leads to the cofibre sequence
$$
\Th(U^{\perp}_{d+1, N}|_{G^{\mf}_{\theta}(\R^{d+1+N})^{k-1}}) \; \longrightarrow \; \Th(U^{\perp}_{d+1, N}|_{G^{\mf}_{\theta}(\R^{d+1+N})^{k}}) \; \longrightarrow \; \Th\left(U^{\loc}_{d+1, N}\oplus U^{\loc, \perp}_{d+1, N}\right).
$$
The proposition follows by observing that these spaces are the $(d+1+N)$th spaces of the spectra $\mb{hW}^{k}_{\theta}$, $\mb{hW}^{k-1}_{\theta}$, and $\mb{hW}^{\{k, d-k+1\}}_{\theta, \loc}$ respectively.
\end{proof}

We want to map the space $\mathcal{D}^{\mf, \{k, d-k+1\}}_{\theta, \loc}$ to the infinite loopspace $\Omega^{\infty-1}\mb{hW}^{\{k, d-k+1\}}_{\theta, \locc}$.
As in the previous section we won't be able to define the map ``on the nose''. 
We will need to work with a slightly different model of $\mathcal{D}^{\mf, \{k, d-k+1\}}_{\theta, \loc}$.
\begin{defn} \label{defn: sheaf version of W-loc}
Fix $N \in \N$. 
We define $\widehat{\mathcal{D}}^{\mf, \{k, d-k+1\}, N}_{\theta, \loc} \subset \mathcal{D}^{\mf, \{k, d-k+1\}}_{\theta, \loc}$ to be the subspace consisting of those $(\bar{x}; (V, \sigma))$ for which:
\begin{enumerate} \itemsep.2cm
\item[(i)] $\bar{x}$ is contained in the subspace $\R\times\R^{d+N}$; 
\item[(ii)] $\dist(x, y) > 2$ for all pairs of distinct points $x, y \in \bar{x}$. 
\end{enumerate}
\end{defn}
It follows by a standard argument that the natural map induced by inclusion 
$$
\colim_{N\to\infty}\widehat{\mathcal{D}}^{\mf, \{k, d-k+1\}, N}_{\theta, \loc} \stackrel{\simeq} \longrightarrow  \mathcal{D}^{\mf, \{k, d-k+1\}}_{\theta, \loc}
$$
is a weak homotopy equivalence. 
We proceed to define a map 
\begin{equation} \label{equation: pontryagin thom local}
\widehat{\mathcal{P}}^{k}_{\theta, \loc}: \colim_{N\to\infty}\widehat{\mathcal{D}}^{\mf, \{k, d-k+1\}, N}_{\theta, \loc} \longrightarrow \Omega^{\infty-1}\mb{hW}^{\{k, d-k+1\}}_{\theta, \loc}.
\end{equation}
We do this in the construction below. 
\begin{Construction} \label{Construction: pontryagin thom construction loc}
To define (\ref{equation: pontryagin thom local}) we will construct for each $N \in \N$ a map 
\begin{equation} \label{equation: transformation to mapping space}
\widehat{\mathcal{D}}^{\mf, \{k, d-k+1\}, N}_{\theta, \loc} \; \longrightarrow \; \Omega^{d+N}\Th(U^{\loc}_{d+1, N}\oplus U^{\loc, \perp}_{d+1, N})
\end{equation}
Let $(\bar{x}; (V, \sigma)) \in \widehat{\mathcal{D}}^{\mf, \{k, d-k+1\}, N}_{\theta, \loc}$. 
By condition (ii) of Definition \ref{defn: sheaf version of W-loc} there is a geodesic tubular neighborhood $U \subset \R\times\R^{d+N}$ of $\bar{x}$ of radius $1$. 
Since the normal bundle $\nu_{\bar{x}}$ of $\bar{x}$ is naturally isomorphic to the trivial bundle $\epsilon^{d+N+1}$ ($\bar{x}$ is a discrete set of points),
the tubular neighborhood $U$ determines a collapse map 
$c: \R\times S^{d+N} \longrightarrow S^{d+N}\times\bar{x}.$
Composing this collapse map with $\Id_{S^{d+N}}\times(V, \sigma)$ we obtain
$$
\xymatrix{
\R\times S^{d+N} \ar[r]^{ c} & S^{d+N}\times\bar{x} \ar[rrr]^{\Id_{S^{d+N}}\times(V, \sigma) \ \ \ \ \ \ \ \ \ \ } &&& S^{d+N}\times G_{\theta}^{\mf}(\R\times\R^{d+N})^{\{k, d-k+1\}}_{\loc, \; +}.
}
$$
Precomposing this map with the inclusion $\{0\}\times S^{d+N} \hookrightarrow \R\times S^{d+N}$ and forming the adjunction yields the desired element of 
$\Omega^{d+N}\Th(U^{\loc}_{d+1, N}\oplus U^{\loc, \perp}_{d+1, N})$.
This construction determines the map (\ref{equation: transformation to mapping space}). 
Letting $N \to \infty$, yields the map 
$\widehat{\mathcal{P}}^{k}_{\theta, \loc}$ of (\ref{equation: pontryagin thom local}). 
\end{Construction}

The following proposition is a special case of the \textit{Barrat-Priddy-Quillen theorem} for configurations spaces consisting of points with labels in the space $G_{\theta}^{\mf}(\R\times\R^{\infty})_{\loc}^{\{k, d-k+1\}}$.
This proposition can also be viewed as a special case of \cite[Theorem 3.12]{GRW 09} applied to spaces of zero-dimensional manifolds with maps to a background space.
\begin{proposition}  \label{proposition: local pontryagin thom equivalence}
For all $k$, the map (\ref{equation: pontryagin thom local}) induces the weak homotopy equivalence,
$$
\xymatrix{
  \displaystyle{\colim_{N\to\infty}}\widehat{\mathcal{D}}^{\mf, \{k, d-k+1\}, N}_{\theta, \loc}  \ar[rr]^{\simeq} && \Omega^{\infty-1}\mb{hW}^{\{k, d-k+1\}}_{\theta, \loc}.
 }
$$ 
\end{proposition}
We emphasize that unlike Theorem \ref{theorem: main theorem}, the map (\ref{equation: pontryagin thom local}) is a weak homotopy equivalence for \textbf{all} choices of $k$.
With the above constructions in place, we are now in a position to put everything together and prove Theorem \ref{theorem: main theorem} using Theorem \ref{theorem: localization sequence}. 
\begin{proof}[Proof of Theorem \ref{theorem: main theorem}, assuming Theorem \ref{theorem: localization sequence}]
Let $k < d/2$.
The maps defined in Constructions \ref{Construction: pontryagin thom construction} and \ref{Construction: pontryagin thom construction loc} fit together into the commutative diagram
\begin{equation} \label{equation: diagram of pontryagin thom}
\xymatrix{
\mathcal{D}^{\mf, k}_{\theta, P} \ar[rr] && \mathcal{D}^{\mf, k-1}_{\theta, P}  \ar[rr] && 
\mathcal{D}^{\mf, \{k, d-k+1\}}_{\theta, \loc} \\
\displaystyle{\colim_{N\to\infty}}\widehat{\mathcal{D}}^{\mf, k, N}_{\theta, P} \ar[u]_{\simeq} \ar[rr] \ar[d] && \displaystyle{\colim_{N\to\infty}}\widehat{\mathcal{D}}^{\mf, k-1, N}_{\theta, P} \ar[u]_{\simeq} \ar[rr] \ar[d] && \displaystyle{\colim_{N\to\infty}}\widehat{\mathcal{D}}^{\mf, \{k, d-k+1\}, N}_{\theta, \loc} \ar[u]_{\simeq} \ar[d]^{\simeq} \\
\Omega^{\infty-1}\mb{hW}^{k}_{\theta} \ar[rr] && \Omega^{\infty-1}\mb{hW}^{k-1}_{\theta} \ar[rr] && \Omega^{\infty-1}\mb{hW}^{\{k, d-k+1\}}_{\theta, \locc}.
}
\end{equation}
The maps of the bottom row are induced by the sequence (\ref{proposition: cofibre of the natural inclusion map}), the maps of the top row are induced by the localization sequence in the statement of Theorem \ref{theorem: localization sequence}, and the maps of the middle row are defined analogously to the maps on the top row.
Commutativity is easily checked by tracing through the constructions. 
By Proposition \ref{proposition: cofibre of the natural inclusion map} the bottom row is a homotopy fibre sequence. 
With this commutative diagram established, the proof of Theorem \ref{theorem: main theorem} follows by the same induction argument given in Section \ref{subsection: localization sequence} (page 9) used in our outline of the proof of Theorem B, where the base of the induction is given by Theorem \ref{theorem: madsen-weiss + vassiliev} (or Theorem A).  
\end{proof}
With the above argument out of the way, we can now devote all of our resources to proving Theorem \ref{theorem: localization sequence}. 

\section{Homotopy Colimit Decompositions} \label{section: homotopy colimit decompositions}
\subsection{A homotopy colimit decomposition} \label{subsection: homotopy commit of sheaves}
We now proceed as in \cite[Section 5]{MW 07} to express the homotopy type of $\mathcal{D}^{\mf, k}_{\theta, P}$ as a homotopy colimit of spaces of compact $d$-dimensional manifolds equipped with surgery data.
We begin by recalling a definition from \cite{MW 07}.
For what follows, fix once and for all an infinite set $\Omega$. 
\begin{defn} \label{defn: surgery category}
An object of the category $\mathcal{K}$ is a finite subset $\mb{s} \subset \Omega$ equipped with a map 
$$\delta: \mb{t} \longrightarrow \{0, 1, 2, \dots, d+1\}.$$ 
A morphism from $\mb{s}$ to $\mb{t}$ is a pair $(j, \varepsilon)$ where $j$ is an injective map over $\{0, \dots, d+1\}$ from $\mb{s}$ to $\mb{t}$, and $\varepsilon$ is a function $\mb{t}\setminus j(\mb{s}) \longrightarrow \{-1, +1\}$. 
The composition of two morphisms $(j_{1}, \varepsilon_{1}): \mb{s} \longrightarrow \mb{t}$ and $(j_{2}, \varepsilon_{2}): \mb{t} \longrightarrow \mb{u}$ is $(j_{2}j_{1}, \varepsilon_{3})$, where $\varepsilon_{3}$ agrees with $\varepsilon_{2}$ outside $j_{2}(\mb{t})$ and $\varepsilon_{1}\circ j^{-1}_{2}$ on $j_{2}(\mb{t}\setminus j_{1}(\mb{s}))$. 
\end{defn}

We will define a space-valued contravariant functor on the category $\mathcal{K}$.
Before doing this we will need to fix some notational conventions.
\begin{Notation} \label{Notation: non-degenerate bilinear form}
Let $(V, \sigma) \in G^{\mf}_{\theta}(\R^{\infty})_{\loc}$. 
The standard Euclidean inner product $\langle \--, \--\rangle$ on the ambient space $\R^{\infty}$ induces an inner product $\langle \--, \--\rangle_{V}$ on the $(d+1)$-dimensional subspace $V$. 
Let $T_{V}: V \longrightarrow V^{*}$ denote the linear isomorphism, $v \mapsto \langle v, \--\rangle_{V}$.
Using $T_{V}$, the bilinear form $\sigma$ determines a self-adjoint linear operator 
$\bar{\sigma}: V \longrightarrow V$
defined by, 
$$\bar{\sigma}(v) \; = \; T^{-1}_{V}(\sigma(v, \-- )).$$
The form $\sigma$ is non-degenerate if and only if $\bar{\sigma}$ is an isomorphism. 
Since $\bar{\sigma}$ is self-adjoint, the vector space $V$ decomposes as $V = V^{+}\oplus V^{-}$ where $V^{\pm}$ are the negative and positive eigenspaces of $\bar{\sigma}$. 
Notice that $\sigma$ is negative definite on the subspace $V^{-}$ and 
$\indexx(\sigma) = \dim(V^{-}).$
\end{Notation}

\begin{Notation} \label{Notation: preliminaries}
Let $\mb{t} \in \mathcal{K}$. 
We will need to consider the mapping space, 
$$(G^{\mf}_{\theta}(\R^{\infty})_{\loc})^{\mb{t}} = \Maps(\mb{t}, \; G^{\mf}_{\theta}(\R^{\infty})_{\loc}).$$ 
A typical element of this mapping space will be denoted $(V, \sigma)$ with 
$$(V, \sigma)(i)  = (V(i), \sigma(i)) \in G^{\mf}_{\theta}(\R^{\infty})_{\loc}$$ 
for $i \in \mb{t}$.
Let 
$(V, \sigma) \in (G^{\mf}_{\theta}(\R^{\infty})_{\loc})^{\mb{t}}.$
For each $i \in \mb{t}$ we may form the unit spheres and disks $S(V^{\pm}(i))$ and $D(V^{\pm}(i))$ defined with respect to the inner product $\langle \--, \--\rangle_{V}$ induced from the inner product in the ambient space. 
We will denote,
$$
D(V^{+})\times_{\mb{t}}S(V^{-}) := \bigsqcup_{i \in \mb{t}}D(V^{+}(i))\times S(V^{-}(i)).
$$
The $\theta$-orientations on the vector spaces $V(i)$ induce a $\theta$-structure on $D(V^{+})\times_{\mb{t}}S(V^{-})$ which we will denote by 
$\hat{\ell}_{\mb{t}}: T(D(V^{+})\times_{\mb{t}}S(V^{-})) \longrightarrow \theta^{*}\gamma^{d+1}.$
We may think of $D(V^{+})\times_{\mb{t}}S(V^{-})$ as being a fibre bundle over $\mb{t}$ equipped with a fibrewise $\theta$-structure. 
\end{Notation}

Fix once and for all an integer $d \in \Z_{\geq 0}$ and a tangential structure $\theta: B \longrightarrow BO(d+1)$. 
We will keep this structure fixed for the rest of the paper. 
To prevent notational overcrowding, the spaces defined below will not contain $\theta$ in their notation.

\begin{defn} \label{defn: colimit decomp of W}
Fix $P \in \mathcal{M}_{\theta}$ and let $\mb{t} \in \mathcal{K}$. 
The set $\mathcal{W}_{P, \mb{t}}$ consists of tuples $(M, (V, \sigma), e)$  where:
\begin{enumerate} \itemsep.3cm
\item[(i)] $M$ is an element of $\mathcal{N}_{\theta, P}$ (where $\mathcal{N}_{\theta, P}$ was defined in Definition \ref{defn: space of manifolds (absolute)});
\item[(ii)] $(V, \sigma)$ is an element of $(G^{\mf}_{\theta}(\R^{\infty})_{\loc})^{\mb{t}}$ that satisfies,
$\delta(i) = \text{index}(\sigma(i))$ for all $i \in \mb{t}$, where recall that $\delta: \mb{t} \longrightarrow \{0, \dots, d+1\}$ is the labeling function. 
\item[(iii)] $e: D(V^{+})\times_{\mb{t}}D(V^{-}) \longrightarrow (-\infty, 0)\times\R^{\infty-1}$ is an embedding 
subject to the following further conditions:
\begin{enumerate} \itemsep.2cm
\item[(a)] $e^{-1}(M) = D(V^{+})\times_{\mb{t}}S(V^{-})$;
\item[(b)] the induced $\theta$-structure $\hat{\ell}_{\mb{t}}$ on $D(V^{+})\times_{\mb{t}}S(V^{-})$ agrees with $\hat{\ell}_{M}$.
\end{enumerate}
\end{enumerate}
\end{defn}

We also have a local version of the above definition. 
\begin{defn} \label{defn: W-loc T}
Let $\mb{t} \in \mathcal{K}$. 
The space $\mathcal{W}_{\locc, \mb{t}}$ consists of pairs $((V, \sigma), e)$ where:
\begin{enumerate} \itemsep.2cm
\item[(i)] $(V, \sigma)$ is an element of $(G^{\mf}_{\theta}(\R^{\infty})_{\loc})^{\mb{t}}$ that satisfies
$\delta(i) = \text{index}(\sigma(i))$
for all $i \in \mb{t}$.
\item[(ii)] $e: D(V^{+})\times_{\mb{t}}D(V^{-}) \longrightarrow (-\infty, 0)\times\R^{\infty-1}$ is a smooth embedding.
\end{enumerate}
\end{defn}

We need to describe how to make the correspondences 
$\mb{t} \mapsto \mathcal{W}_{P, \mb{t}}$ and $\mb{t} \mapsto \mathcal{W}_{\loc, \mb{t}}$
into contravariant functors on $\mathcal{K}$.
The case with $\mathcal{W}_{\loc, (\--)}$ is easy; if $(j, \varepsilon): \mb{s} \longrightarrow \mb{t}$ is a morphism, the map 
$
(j, \varepsilon)^{*}: \mathcal{W}_{\loc, \mb{t}} \longrightarrow \mathcal{W}_{\loc, \mb{s}}
$
is given by sending an element $(V, e) \in \mathcal{W}_{\loc, \mb{t}}$ to the element in $\mathcal{W}_{\loc, \mb{s}}$ obtained by precomposing $V$ with $j: \mb{s} \hookrightarrow \mb{t}$, and then restricting the embedding $e$. 
Notice that in this definition the function $\varepsilon$ played no role. 

Describing the functor $\mb{t} \mapsto \mathcal{W}_{P, \mb{t}}$ will take more work. 
Let $(j, \varepsilon): \mb{s} \longrightarrow \mb{t}$ be a morphism in $\mathcal{K}$. 
If $j$ is bijective, there is an obvious identification $\mathcal{W}_{P, \mb{t}} \cong \mathcal{W}_{P, \mb{s}}$ and this is the induced map. 
We may assume that $j$ is an inclusion $\mb{s} \hookrightarrow \mb{t}$. 
We may then reduce to the case where $\mb{t}\setminus \mb{s}$ has exactly one element, $a$. 
This case has two subcases: $\varepsilon(a) = +1$ and $\varepsilon(a) = -1$. 

\begin{defn}  \label{defn: pullback +1}
Let $(j, \varepsilon): \mb{s} \longrightarrow \mb{t}$ be a morphism in $\mathcal{K}$ where $j$ is an inclusion and $\mb{t}\setminus\mb{s} = \{a\}$ with $\varepsilon(a) = +1$. 
We describe the induced map 
$
(j, \varepsilon)^{*}: \mathcal{W}_{P, \mb{t}} \longrightarrow \mathcal{W}_{P, \mb{s}}.
$
Let $(M, (V, \sigma), e)$ be an element of $\mathcal{W}_{P, \mb{t}}$.
Map this to an element of $\mathcal{W}_{P, \mb{s}}$ by pulling $M$ and $V$ back along the inclusion $\mb{s}\times X \hookrightarrow \mb{t}\times X$, and restricting $e$ accordingly.
\end{defn}

\begin{defn} \label{defn: pullback -1} 
Let $(j, \varepsilon): \mb{s} \longrightarrow \mb{t}$ be a morphism in $\mathcal{K}$ where $j$ is an inclusion and $\mb{t}\setminus\mb{s} = \{a\}$ with $\varepsilon(a) = -1$. 
The induced map 
$
\mathcal{W}_{P, \mb{t}} \longrightarrow \mathcal{W}_{P, \mb{s}}.
$
is defined as follows. 
Let $(M, (V, \sigma), e)$ be an element of $\mathcal{W}_{P, \mb{t}}$.
Map this to the element $(\widetilde{M}, (\widetilde{V}, \widetilde{\sigma}), \widetilde{e})$ in $\mathcal{W}_{P, \mb{s}}$ where: 
\begin{enumerate} \itemsep.2cm
\item[(i)] $\widetilde{M}$ is the element of $\mathcal{N}_{\theta, P}$ given by, 
$$
\widetilde{M} \; = \; (M\setminus e(S(V^{-})\times_{a}D(V^{+}))\bigcup e(D(V^{-})\times_{a}S(V^{+})),
$$
where $S(V^{-})\times_{a}D(V^{+}) \subset S(V^{-})\times_{\mb{t}}D(V^{+})$ is the component of $S(V^{-})\times_{\mb{t}}D(V^{+})$ that corresponds to $a \in \mb{t}$.
\item[(ii)] $(\widetilde{V}, \widetilde{\sigma})$ is the restriction of $(V, \sigma)$ to $\mb{s}$; 
\item[(iii)] $\widetilde{e}$ is obtained from $e$ by restriction. 
\end{enumerate}
The definitions above make the assignment $\mb{t} \mapsto \mathcal{W}_{P, \mb{t}}$ into a contravariant functor on $\mathcal{K}$.
\end{defn}

We will need to work with a ``restricted index'' version of $\mathcal{W}_{P, \mb{t}}$.
For each integer $k$, let $\mathcal{K}^{k} \subset \mathcal{K}$ denote the full subcategory consisting of those $\mb{t} \in \mathcal{K}$ whose reference map $\delta: \mb{t} \longrightarrow \{0, \dots, d+1\}$ has its image in the subset $\{k+1, \dots, d-k\} \subset \{0, \dots, d+1\}$. 
Similarly, we let $\mathcal{K}^{\{k, d-k+1\}} \subset \mathcal{K}$ be the full subcategory consisting of those objects $\mb{t}$ with $\delta(\mb{t}) \subset \{k, d-k+1\}$.

\begin{defn} \label{defn: restricted index version of D}
Let $k$ be an integer.
For $\mb{t} \in \mathcal{K}^{k}$ we define
$\mathcal{W}^{k}_{P, \mb{t}} \subset \mathcal{W}_{P, \mb{t}}$
to be the subset consisting of those $(M, (V, \sigma), e)$ for which the map 
$\ell_{M}: M \longrightarrow B$
is $k$-connected. 
In this way $\mb{t} \mapsto \mathcal{W}^{k}_{P, \mb{t}}$ defines a functor on $\mathcal{K}^{k}$.
\end{defn}

Let 
$p_{\{k, d-k+1\}}: \mathcal{K}^{k-1} \longrightarrow \mathcal{K}^{\{k, d-k+1\}}$
denote the projection functor, which is defined by sending an object $\mb{t} \in \mathcal{K}^{k-1}$ to the subset $\mb{t}_{\{k, d-k+1\}} \subset \mb{t}$ consisting of all points with label contained in $\{k, d-k+1\}$.
For $\mb{t} \in \mathcal{K}^{k-1}$, we define 
\begin{equation}
\mathcal{W}^{\{k, d-k+1\}}_{\locc, \mb{t}} := \mathcal{W}_{\locc, \; p_{\{k, d-k+1\}}(\mb{t})}.
\end{equation}
In this way we may view $\mathcal{W}^{\{k, d-k+1\}}_{\locc, (\--)}$ as a functor on $\mathcal{K}^{k-1}$.
For all $\mb{t} \in \mathcal{K}^{k-1}$ we have a map 
\begin{equation} \label{equation: localization map hocolim version}
\mathcal{W}^{k-1}_{P, \mb{t}} \longrightarrow \mathcal{W}^{\{k, d-k+1\}}_{\locc, \mb{t}}
\end{equation}
defined by sending $(M, (V, \sigma), e)$ to $((V, \sigma)_{\{k, d-k+1\}}, \;  e_{\{k, d-k+1\}})$, which is the element obtained by restricting $(V, \sigma)$ and $e$ to the components of $\mb{t}$ with labels in $\{k, d-k+1\}$.
This map yields a natural transformation of contravariant functors on $\mathcal{K}^{k-1}$.

\subsection{Homotopy colimits} \label{subsection: homotopy colimits}
We turn our attention to the homotopy colimit,
$\displaystyle{\hocolim_{\mb{t} \in \mathcal{K}^{k}}}\mathcal{W}^{k}_{P, \mb{t}}.$
Before stating the main theorem, let us recall a particular model for the homotopy colimit of a functor (or contravariant functor) from a small category to the category of topological spaces. 
\begin{defn} \label{defn: homotopy colimit}
Let $\mathcal{C}$ be a small category and let $\mathcal{F}$ be a contravariant functor from $\mathcal{C}$ to the category of topological spaces. 
The \textit{transport category} $\mathcal{C}\wr\mathcal{F}$ has as its objects, pairs $(C, x)$ where $C \in \Ob\mathcal{C}$ and $x \in \mathcal{F}(C)$. 
A morphism $(B, x) \longrightarrow (C, y)$ in $\mathcal{C}\wr\mathcal{F}$ is a morphism $f: C \longrightarrow B$ in $\mathcal{C}$ such that $F(f)(x) = y$.
The \textit{homotopy colimit} $\displaystyle{\hocolim_{C \in \mathcal{C}}}\mathcal{F}(C)$ is defined to be the classifying space $B(\mathcal{C}\wr\mathcal{F})$.
\end{defn}

The following theorem is a generalization of the main result from \cite[Section 5]{MW 07}.
We give its proof in Appendix \ref{section: proof of theorem homotopy colimit theorem}. 
The theorem is proven by repeating the same steps from \cite[Section 5]{MW 07} almost verbatim. 
In Appendix \ref{section: proof of theorem homotopy colimit theorem} we provide sketches of the main steps of the proof and refer the reader to the relevant construction in \cite{MW 07} for the details. 
\begin{theorem} \label{theorem: homotopy colimit decomposition}
For all $k$ and $P \in \mathcal{M}_{\theta}$, there exists spaces $\mathcal{L}^{k-1}_{\theta}$ and $\mathcal{L}^{\{k, d-k+1\}}_{\theta, \loc}$ together with a commutative diagram
$$
\xymatrix{
\displaystyle{\hocolim_{\mb{t} \in \mathcal{K}^{k-1}}}\mathcal{W}^{k-1}_{P, \mb{t}} \ar[d]  && \mathcal{L}^{k-1}_{\theta} \ar[ll]_{\simeq} \ar[rr]^{\simeq} \ar[d] && \mathcal{D}^{\mf, k-1}_{\theta, P}  \ar[d] \\
\displaystyle{\hocolim_{\mb{t} \in \mathcal{K}^{k-1}}}\mathcal{W}^{\{k, d-k+1\}}_{\locc, \mb{t}}  && \mathcal{L}^{\{k, d-k+1\}}_{\theta, \locc} \ar[ll]_{\simeq} \ar[rr]^{ \ \ \ \ \simeq} && \mathcal{D}_{\theta, \locc}^{\mf, \{k, d-k+1\}}, 
}
$$
such that all horizontal maps are weak homotopy equivalences. 
\end{theorem}

\subsection{Increased Connectivity} \label{subsection: increased connectivity}
In view of the above theorem, 
Theorem \ref{theorem: localization sequence} translates to the statement that the sequence, 
$
\displaystyle{\hocolim_{\mb{t} \in \mathcal{K}^{k}}}\mathcal{W}^{k}_{P, \mb{t}}\longrightarrow \displaystyle{\hocolim_{\mb{t} \in \mathcal{K}^{k-1}}}\mathcal{W}^{k-1}_{P, \mb{t}} \; \longrightarrow \displaystyle{\hocolim_{\mb{t} \in \mathcal{K}^{k-1}}}\mathcal{W}^{\{k, d-k+1\}}_{\locc, \mb{t}}, 
$
is a homotopy fibre sequence whenever $k < d/2$.
We will break the proof of this result down into two intermediate theorems. 
To state them we will need to define an intermediate space sitting in between $\mathcal{W}^{k}_{P, \mb{t}}$ and $\mathcal{W}^{k-1}_{P, \mb{t}}$.
\begin{defn} \label{defn: intermediate high connected sheaf}
For each $\mb{t} \in \mathcal{K}^{k}$, we define
$\mathcal{W}^{k, c}_{P, \mb{t}} \subset \mathcal{W}^{k}_{P, \mb{t}}$
to be the subspace consisting of those $(M, (V, \sigma), e)$ such that the $\theta$-structure $\hat{\ell}_{M}$ satisfies the following connectivity condition:
the restricted map, 
$\ell|_{M\setminus\Image(e)}: M\setminus\Image(e) \longrightarrow B$, 
is $(k+1)$-connected.
\end{defn}
Let $(j, \varepsilon): \mb{t} \longrightarrow \mb{s}$ be a morphism in $\mathcal{K}^{k}$.
If $(M, (V, \sigma), e) \in \mathcal{W}^{k, c}_{P, \mb{s}}$, it can be verified that the image 
$
(j, \varepsilon)^{*}(M, (V, \sigma), e) \in \mathcal{W}^{k}_{P, \mb{t}}
$
actually lies in the subset, 
$\mathcal{W}^{k, c}_{P, \mb{t}} \subset \mathcal{W}^{k}_{P, \mb{t}}.$
Indeed, the surgery performed on $(M, (V, \sigma), e)$ as a result of this morphism does not destroy the connectivity condition from Definition \ref{defn: intermediate high connected sheaf}. 
Thus, the correspondence 
$\mb{t} \mapsto \mathcal{W}^{k, c}_{P, \mb{t}}$
defines a functor on $\mathcal{K}^{k}$, and we can form the homotopy colimit, 
$\displaystyle{\hocolim_{\mb{t} \in \mathcal{K}^{k}}}\mathcal{W}^{k, c}_{P, \mb{t}}.$
With this space defined, Theorem \ref{theorem: localization sequence} follows by combining the two theorems stated below.
\begin{theorem} \label{theorem: fibrewise surgery}
Let $k < d/2$.
Suppose that $B$ satisfies Wall's finiteness condition $F(k)$. 
Then the inclusion,
$
\displaystyle{\hocolim_{\mb{t} \in \mathcal{K}^{k}}}\mathcal{W}^{k, c}_{P, \mb{t}} \; \hookrightarrow \; \displaystyle{\hocolim_{\mb{t} \in \mathcal{K}^{k}}}\mathcal{W}^{k}_{P, \mb{t}},
$
is a weak homotopy equivalence.
\end{theorem}

To state our next theorem we will need to fix a basepoint in the space 
$\displaystyle{\hocolim_{\mb{t} \in \mathcal{K}^{k-1}}}\mathcal{W}^{\{k, d-k-1\}}_{\locc, \mb{t}}$.
Notice that the space $\mathcal{W}^{\{k, d-k-1\}}_{\locc, \emptyset}$ has exactly one element, namely the one determined by the empty set. 
This element determines an element in the homotopy colimit which we denote by, 
$$\emptyset \; \in \; \displaystyle{\hocolim_{\mb{t} \in \mathcal{K}^{k-1}}}\mathcal{W}^{\{k, d-k-1\}}_{\locc, \mb{t}}.$$

\begin{theorem} \label{theorem: localization sequence (homotopy colimit)}
Let $k < d/2$.
The map, 
$\displaystyle{\hocolim_{\mb{t} \in \mathcal{K}^{k-1}}}\mathcal{W}^{k-1, c}_{P, \mb{t}} \; \longrightarrow \displaystyle{\hocolim_{\mb{t} \in \mathcal{K}^{k-1}}}\mathcal{W}^{\{k, d-k+1\}}_{\locc, \mb{t}}, 
$
is a quasi-fibration. 
The fibre over the point $\emptyset$ is given by the space
$\displaystyle{\hocolim_{\mb{t} \in \mathcal{K}^{k}}}\mathcal{W}^{k}_{P, \mb{t}}.$
\end{theorem}

The proof of Theorem \ref{theorem: fibrewise surgery} is carried out in Section \ref{section: parametrized surgery}. 
Theorem \ref{theorem: localization sequence (homotopy colimit)} is proven over the course of Sections \ref{section: A Fibre Sequence} through \ref{section: higher index handles}. 

\subsection{Fibre sequences of homotopy colimits} \label{subsection: fibre sequences of homotopy colimits}
In the following subsection we state a technique for establishing the fibre sequences in Theorem \ref{theorem: localization sequence (homotopy colimit)}.  
The following proposition is well known. 
Its proof can be found in \cite[Corollary B.3]{MW 07}.
\begin{proposition} \label{proposition: homotopy colimit fibre sequence}
Let $\mathcal{C}$ be a small category and let $u: \mathcal{G}_{1} \longrightarrow \mathcal{G}_{2}$ be a natural transformation between functors 
from $\mathcal{C}$ to the category of topological spaces. 
Suppose that for each morphism $f: a \longrightarrow b$ in $\mathcal{C}$, the map 
$$
f_{*}: \hofibre\left(u_{a}: \mathcal{G}_{1}(a) \rightarrow \mathcal{G}_{2}(a)\right) \; \longrightarrow \; \hofibre\left(u_{b}: \mathcal{G}_{1}(b) \rightarrow \mathcal{G}_{2}(b)\right)
$$
is a weak homotopy equivalence. 
Then for any object $c \in  \Ob\mathcal{C}$ the inclusion 
$$
\hofibre\left(u_{c}: \mathcal{G}_{1}(c) \rightarrow \mathcal{G}_{2}(c)\right) \; \hookrightarrow \; \hofibre\left(\hocolim\mathcal{G}_{1} \rightarrow \hocolim\mathcal{G}_{2}\right)
$$
is a weak homotopy equivalence. 
\end{proposition}

\section{Parametrized Surgery} \label{section: parametrized surgery}
This section is devoted to proving Theorem \ref{theorem: fibrewise surgery} which states that the inclusion 
$$
\displaystyle{\hocolim_{\mb{t} \in \mathcal{K}^{k}}}\mathcal{W}^{k, c}_{P, \mb{t}}\longrightarrow \displaystyle{\hocolim_{\mb{t} \in \mathcal{K}^{k}}}\mathcal{W}^{k}_{P, \mb{t}}
$$
is a weak homotopy equivalence whenever $k < d/2$ and the space $B$ is $F(k)$.

\subsection{The category of multiple surgeries} \label{subsection: the category of multiple surgeries}
To prove the above weak homotopy equivalence we will have to work with a category that is a generalization of \cite[Definition 6.6]{MW 07}.
For what follows, let $\Omega$ be the infinite set used in the definition of $\mathcal{K}$ (Definition \ref{defn: surgery category}). 
Let $M \subset \R^{\infty}$ be a
$d$-dimensional submanifold equipped with a $\theta$-structure $\hat{\ell}_{M}: TM\oplus\epsilon^{1} \longrightarrow \theta^{*}\gamma^{d+1}$. 
Suppose that the map 
$\ell_{M}: M \longrightarrow B$ is $(k-1)$-connected.
\begin{defn} \label{defn: top category of multiple surgeries}
With $M$ and $\hat{\ell}_{M}$ as above, let $\mathcal{S}^{k-1}_{M}$ be the category defined as follows.
An object is a tuple $(\mb{t}, \hat{\ell}, e)$ where $\mb{t} \subset \Omega$ is a finite subset, $\hat{\ell}$ is a $\theta$-structure on $\mb{t}\times D^{k}\times D^{d-k+1}$, and  
$e: \mb{t}\times D^{k}\times D^{d-k+1} \longrightarrow \R^{\infty}$ is a smooth embedding,
subject to the following conditions:
\begin{enumerate} \itemsep.2cm
\item[(i)] $e^{-1}(M) = \mb{t}\times S^{k-1}\times D^{d-k+1}$ and the restriction of $\hat{\ell}_{M}$ to the image of $e$ agrees with the structure $\hat{\ell}$ on $\mb{t}\times D^{k}\times D^{d-k+1}$.
\item[(ii)] Consider the surgered manifold, 
$\widetilde{M} = M\setminus e(\mb{t}\times S^{k-1}\times D^{d-k+1})\bigcup e(\mb{t}\times D^{k}\times S^{d-k}).$
Let $\hat{\ell}_{\widetilde{M}}$ be the $\theta$-structure on $\widetilde{M}$ induced by $\hat{\ell}_{M}$ and $\hat{\ell}$.
We require the map $\ell_{\widetilde{M}}: \widetilde{M} \longrightarrow B$ to be $k$-connected. 
\end{enumerate}
A morphism $(\mb{t}, e, \hat{\ell}) \longrightarrow (\mb{s}, e', \hat{\ell}')$ is defined to be an injective map $j: \mb{t} \hookrightarrow \mb{s}$ 
such that $j^{*}e' = e$ and $j^{*}\hat{\ell}' = \hat{\ell}$.
The object space of this category is topologized as a subspace of, 
$$
\coprod_{\mb{t} \subset \Omega}\left[\Emb(\mb{t}\times D^{k}\times D^{d-k+1}, \; \R^{\infty})\times\Bun(T(\mb{t}\times D^{k}\times D^{d-k+1}), \; \theta^{*}\gamma^{d+1})\right].
$$
The morphism space is topologized in a similar way.
\end{defn}

The main technical ingredient used in the proof of Theorem \ref{theorem: fibrewise surgery} is the following proposition, which we view as a generalization of \cite[Proposition 6.7 (Page 916)]{MW 07}.
\begin{proposition} \label{proposition: contractibility of the surgery category}
Let $k \leq d/2$ and 
suppose that $\theta: B \longrightarrow BO(d+1)$ is such that $B$ satisfies Wall's finiteness condition $F(k)$. 
Let $M \subset \R^{\infty}$ be a $d$-dimensional $\theta$-manifold with the property that $\ell_{M}: M \longrightarrow B$ is $(k-1)$-connected.
Then the classifying space $B\mathcal{S}^{k-1}_{M}$ is weakly contractible. 
\end{proposition}

The proof of the above proposition will require the use of techniques from the following subsection. 
The proof of Proposition \ref{proposition: contractibility of the surgery category} will be completed in Section \ref{subsection: proof of proposition contractible category}. 

\subsection{Bases for fibrewise surgery} \label{subsection: bases for fibrewise surgery} 
We begin by making a definition. 
\begin{defn} \label{notation: bundles}
Let $X$ be a smooth manifold.
Let $M \subset X\times\R^{\infty}$ be a smooth $(\dim(X) + d)$-dimensional submanifold for which the projection $\pi: M \longrightarrow X$ is a submersion (we do not assume that this submersion is proper).
Let $\hat{\ell}: T^{\pi}M\oplus\epsilon^{1} \longrightarrow \theta^{*}\gamma^{d+1}$ be a fibrewise $\theta$-structure with the property that for all $x \in X$, the map $\ell_{x}: M_{x} \longrightarrow B$ is $(k-1)$-connected, where $M_{x} = \pi^{-1}(x)$ denotes the fibre over $x$.  
We will denote by $\Sub_{\theta, k-1}(X)$ the set of all such pairs $(M, \hat{\ell})$. 
\end{defn}

\begin{defn} \label{defn: surgery basis fold map}
Fix once and for all an infinite set $\Omega$.
A \textit{fibrewise surgery basis} (of degree $k$) for an element $(M, \hat{\ell}) \in \Sub_{\theta, k-1}(X)$ is given by the following data:
\begin{itemize} \itemsep.2cm
\item a locally finite covering $\mathcal{U} = \{ U_{\alpha} \; | \; \alpha \in \Gamma\}$ of $X$;
\item a collection of finite subsets 
$\Lambda_{\alpha} \subset \Omega$ 
together with fibrewise embeddings 
\begin{equation} \label{equation: fibrewise surgery data}
\varphi_{\alpha}: U_{\alpha}\times\Lambda_{\alpha}\times D^{k}\times D^{d-k+1} \; \longrightarrow \; X\times\R^{\infty}, \quad \alpha \in \Gamma,
\end{equation}
for which
$\varphi_{\alpha}^{-1}(M) = U_{\alpha}\times\Lambda_{\alpha}\times S^{k-1}\times D^{d-k+1};$

\item for each $\alpha \in \Gamma$, a fibrewise $\theta$-structure,
$\phi_{\alpha}: U_{\alpha}\times\Lambda_{\alpha}\times T(D^{k}\times D^{d-k+1}) \longrightarrow \theta^{*}\gamma^{d+1},$
that extends the fibrewise $\theta$-structure, $\hat{\ell}\circ D\varphi_{\alpha}$.
\end{itemize}
The data described above is required to satisfy the following conditions:
\begin{enumerate} \itemsep.2cm
\item[(i)] For $\alpha \in \Gamma$ and $x \in U_{\alpha}$, let $(\widetilde{M}_{x}(\alpha), \tilde{\ell}_{x}(\alpha))$ denote the $\theta$-manifold obtained from $(M_{x}, \hat{\ell}_{x})$ by performing $\theta$-surgery with respect to the data $(\varphi_{\alpha}, \phi_{\alpha})$.
We require that for each $\alpha \in \Gamma$ and $x \in U_{\alpha}$, the map 
$\tilde{\ell}_{x}(\alpha): \widetilde{M}_{x}(\alpha) \longrightarrow B$ 
is $k$-connected.
\item[(ii)] For all $\alpha \neq \beta$, 
$\varphi_{\alpha}(U_{\alpha}\times\Lambda_{\alpha}\times D^{k}\times D^{d-k+1})\bigcap\varphi_{\beta}(U_{\beta}\times\Lambda_{\beta}\times D^{k}\times D^{d-k+1}) = \emptyset.$
\end{enumerate}
A fibrewise surgery basis defined as above will be denoted by the tuple $(\mathcal{U}, \varphi, \phi, \Gamma)$. 
\end{defn}

The following theorem is the main technical ingredient used in the proof of Proposition \ref{proposition: contractibility of the surgery category}. 
This can be viewed as a generalization of \cite[Lemma 6.8]{MW 07}
\begin{theorem} \label{theorem: fibrewise surgery lemma}
Let $k \leq d/2$ and let $(M, \hat{\ell}) \in \Sub_{\theta, k-1}(X)$. 
Suppose that $B$ and $M_{x}$ satisfy Wall's finiteness condition $F(k)$ for all $x \in X$. 
Then $(M, \hat{\ell})$ admits a fibrewise surgery basis. 
\end{theorem} 

For our next step in proving Theorem \ref{theorem: fibrewise surgery lemma}, we formulate and prove a slightly weaker version of the statement. 
We begin by formulating a weakened version of Definition \ref{defn: surgery basis fold map}. 
\begin{defn} \label{defn: fibrewise semi-surgery basis}
Let $(M, \hat{\ell}) \in \Sub_{\theta, k-1}(X)$.
A \textit{fibrewise surgery semi-basis} is a tuple $(\mathcal{U}, \varphi, \phi, \Gamma)$ exactly as in Definition \ref{defn: surgery basis fold map} except now we drop condition (iii). 
In particular, the images of the embeddings $\varphi_{i}$ are allowed to have non-empty intersections. 
\end{defn}

\begin{lemma} \label{lemma: relative existence of fibrewise surgery semi-basis}
Let $k \leq d/2$ and let $(M, \hat{\ell}) \in \Sub_{\theta, k-1}(X)$.
Suppose that $B$ and $M_{x}$ satisfy Wall's finiteness condition $F(k)$ for all $x \in X$. 
Then there exists a sequence of fibrewise surgery semi-bases
$(\mathcal{U}^{n}, \varphi^{n}, \phi^{n}, \Gamma^{n}), \; n \in \N,$
that satisfies the following condition:
If $i > j$, then  
$$
\varphi^{i}_{\alpha}(U^{i}_{\alpha}\times\Lambda^{i}_{\alpha}\times D^{k}\times D^{d-k+1})\bigcap\varphi^{j}_{\beta}(U^{j}_{\beta}\times\Lambda^{j}_{\beta}\times D^{k}\times D^{d-k+1}) = \emptyset$$
whenever $\alpha \in \Gamma^{i}$ and $\beta \in \Gamma^{j}$.
\end{lemma}
\begin{proof}
We will explicitly construct a family of fibrewise surgery semi-bases $(\mathcal{U}^{n}, \varphi^{n}, \phi^{n}, \Gamma^{n})$ satisfying the slightly weaker condition,
\begin{equation} \label{equation: disjoint core condition}
\varphi^{i}_{\alpha}(U^{i}_{\alpha}\times\Lambda^{i}_{\alpha}\times D^{k}\times\{0\})\bigcap\varphi^{j}_{\beta}(U^{j}_{\beta}\times\Lambda^{j}_{\beta}\times D^{k}\times\{0\}) = \emptyset,
\end{equation}
whenever $i > j$ and $\alpha \in \Gamma^{i}$ and $\beta \in \Gamma^{j}$. 
With such embeddings constructed satisfying this condition, the stronger condition from the statement of the lemma is achieved by precomposing the embeddings $\varphi^{i}_{\alpha}$ with a self-embedding $D^{k}\times D^{d-k+1} \longrightarrow D^{k}\times D^{d-k+1}$ with image sufficiently close to the core $D^{k}\times\{0\}$. 
We construct the family $(\mathcal{U}^{m}, \varphi^{m}, \phi^{m}, \Gamma^{m})$, $m \in \N$ satisfying (\ref{equation: disjoint core condition}) by induction on $m$.
Fix an integer $n$ and suppose that the fibrewise surgery semi-bases $(\mathcal{U}^{i}, \varphi^{i}, \phi^{i}, \Gamma^{i})$ have been constructed for all $i < n$. 
We show how to construct $(\mathcal{U}^{n}, \varphi^{n}, \phi^{n}, \Gamma^{n})$.
Let $x \in X$ and consider the fibre $M_{x}$. 
By assumption, the map $\ell_{x}: M_{x} \longrightarrow B$ is $(k-1)$-connected. 
Since $B$ and $M_{x}$ satisfy Wall's finiteness condition $F(k)$, it follows from \cite[Proposition 4]{Kr 00} that there exists the following data:
\begin{itemize} \itemsep.2cm
\item a finite subset $\Lambda_{x} \subset \Omega$;
\item an embedding 
$
\varphi_{x}: \Lambda_{x}\times D^{k}\times D^{d-k+1} \longrightarrow \{x\}\times\R^{\infty}
$
with 
$
\varphi^{-1}_{x}(M) = \Lambda_{x}\times S^{k-1}\times D^{d-k+1};
$
\item a $\theta$-structure 
$\phi_{x}: T(\Lambda_{x}\times D^{k}\times D^{d-k+1}) \longrightarrow \theta^{*}\gamma^{d+1}$
that extends
$$
\xymatrix{
T(\Lambda_{x}\times S^{k-1}\times D^{d-k+1})\oplus\epsilon^{1} \ar[rr]^{ \ \ \ \ \ D\varphi_{x}} && TM_{x}\oplus\epsilon^{1} \ar[rr]^{\hat{\ell}_{x}} && \theta^{*}\gamma^{d+1};
}
$$
\end{itemize}
with the property that the map  
$\tilde{\ell}_{x}: \widetilde{M}_{x} \longrightarrow B$ 
is $k$-connected, where 
$(\widetilde{M}_{x}, \tilde{\ell}_{x})$ is the $\theta$-manifold obtained from $(M_{x}, \ell_{x})$ by performing surgery with respect to the data $(\varphi, \phi)$. 
Since the collection of embeddings $\varphi^{i}_{\alpha}$ with $(i, \alpha) \in \{0, \dots, n-1\}\times(\coprod^{n-1}_{i=1}\Gamma^{i})$ is finite, 
we may perturb $\varphi_{x}$ so as to make 
$\varphi_{x}(\Lambda_{x}\times S^{k-1}\times\{0\}) \subset M_{x}$
transverse to the submanifolds 
$\varphi^{i}_{\alpha}(\Lambda^{i}_{\alpha}\times S^{k-1}\times\{0\})\cap M_{x}$ 
for all $(i, \alpha) \in \{0, \dots, n-1\}\times(\coprod^{n-1}_{i=1}\Gamma^{i})$.
Since $k \leq d/2$ and $\dim(M) = d$, this transversality condition implies that 
$$
\varphi_{x}(\Lambda_{x}\times S^{k-1}\times\{0\})\bigcap\varphi^{i}_{\alpha}(\Lambda^{i}_{\alpha}\times S^{k-1}\times\{0\}) \; = \; \emptyset, \\
$$
for all $(i, \alpha) \in \{0, \dots, n-1\}\times(\coprod^{n-1}_{i=1}\Gamma^{i})$. 
With the disjointness condition achieved in $M_{x}$, we may perturb 
$\varphi_{x}(\Lambda_{x}\times D^{k}\times\{0\}) \subset \R^{\infty},$
keeping $\varphi_{x}(\Lambda_{x}\times S^{k-1}\times\{0\}) \subset M$ fixed, so as to achieve
$$
\varphi_{x}(\Lambda_{x}\times D^{k}\times\{0\})\bigcap\varphi^{i}_{\alpha}(\Lambda^{i}_{\alpha}\times D^{k}\times\{0\}) \; = \; \emptyset, 
$$
for all $(i, \alpha) \in \{0, \dots, n-1\}\times(\coprod^{n-1}_{i=1}\Gamma^{i})$. 
Since $\pi: M \longrightarrow X$ is a locally trivial fibre bundle there exists: 
\begin{itemize} \itemsep.2cm
\item a neighborhood $U_{x} \subset X$ of $x$,
\item a fibrewise embedding 
$
\bar{\varphi}_{x}: U_{x}\times\Lambda\times D^{k}\times D^{d-k+1} \longrightarrow U_{x}\times\R^{\infty}
$
that extends $\varphi_{x}$, and 
\item a $\theta$-structure 
$
\bar{\phi}_{x}: U_{x}\times T(\Lambda\times D^{k}\times D^{d-k+1}) \longrightarrow \theta^{*}\gamma^{d+1}
$
that extends both $\phi_{x}$ and the fibrewise $\theta$-structure,
$
\xymatrix{
U_{x}\times T(\Lambda\times S^{k-1}\times D^{d-k+1})\oplus\epsilon^{1} \ar[rr]^{ \ \ \ \ \ \ \ \ \ \ \ \ \ D\varphi_{x}} && TM\oplus\epsilon^{1} \ar[r]^{\hat{\ell}_{x}} & \theta^{*}\gamma^{d+1}.
}
$
\end{itemize}
Furthermore, by choosing the neighborhood $U_{x}$ to be sufficiently small (containing the point $x \in X$) we may achieve the condition,
\begin{equation}
\bar{\varphi}_{x}(U_{x}\times\Lambda_{x}\times D^{k}\times\{0\})\bigcap\varphi^{i}_{\alpha}(U^{i}_{\alpha}\times\Lambda^{i}_{\alpha}\times D^{k}\times\{0\}) \; = \; \emptyset, 
\end{equation}
for all $(i, \alpha) \in \{0, \dots, n-1\}\times(\coprod^{n-1}_{i=1}\Gamma^{i})$.
By carrying out the same construction for every point $x \in X$ we obtain: 
\begin{itemize} \itemsep.2cm
\item an open cover $\{U_{x} \; | \; x \in X\}$; 
\item a collection of finite subsets $\Lambda_{x} \subset \Omega$, $x \in X$; 
\item a collection of fibrewise embeddings 
$\bar{\varphi}_{x}: U_{x}\times\Lambda_{x}\times D^{k}\times D^{d-k+1} \longrightarrow X\times\R^{\infty}, \quad x \in X;$ 
\item a collection of $\theta$-structures 
$\bar{\phi}_{x}: U_{x}\times T(\Lambda_{x}\times D^{k}\times D^{d-k+1}) \longrightarrow \theta^{*}\gamma^{d+1}, \quad x \in X;$ 
\end{itemize}
satisfying the same conditions stated above. 
Since $X$ is paracompact, $\{U_{x} \; | \; x \in X\}$ admits a locally-finite subcover. 
Let $\mathcal{U}^{n} = \{U^{n}_{\alpha} \; | \; \alpha \in \Gamma^{n}\}$ be such a locally-finite subcover of $\{U_{x} \; | \; x \in X\}$ and let $\varphi^{n}_{\alpha}$ and $\phi^{n}_{\alpha}$ denote the embedding and $\theta$-structure corresponding to $U_{\alpha}$ for each $\alpha \in \Gamma^{n}$. 
The tuple $(\mathcal{U}^{n}, \varphi^{n}, \phi^{n}, \Gamma^{n})$ then is the $n$-th fibrewise surgery semi-basis satisfying the condition in the statement of the lemma. 
\end{proof}

\subsection{Proof of Theorem \ref{theorem: fibrewise surgery lemma}}
We now show how to use Lemma \ref{lemma: relative existence of fibrewise surgery semi-basis} from the previous subsection to prove Theorem \ref{theorem: fibrewise surgery lemma}. 
Let $(M, \hat{\ell}) \in \Sub_{\theta, k-1}(X)$. 
By Lemma \ref{lemma: relative existence of fibrewise surgery semi-basis}, there exists a sequence $(\mathcal{U}^{n}, \varphi^{n}, \phi^{n}, \Gamma^{n})$, $n \in \N$, of fibrewise surgery semi-bases 
that satisfy the condition from the statement of Lemma \ref{lemma: relative existence of fibrewise surgery semi-basis}: 
If $i > j$, then  
$
\varphi^{i}_{\alpha}(U^{i}_{\alpha}\times\Lambda^{i}_{\alpha}\times D^{k}\times D^{d-k+1})\bigcap\varphi^{j}_{\beta}(U^{j}_{\beta}\times\Lambda^{j}_{\beta}\times D^{k}\times D^{d-k+1}) = \emptyset$
for all $\alpha \in \Gamma^{i}$ and $\beta \in \Gamma^{j}$.
With this sequence of fibrewise surgery semi-bases chosen, we set 
$$
\Gamma = \coprod_{n=1}^{\infty}\Gamma^{n}.
$$
Using the sequence $(\mathcal{U}^{n}, \varphi^{n}, \phi^{n}, \Gamma^{n})$, we define a space, 
\begin{equation} \label{equation: etale space over X 1}
P \; \subset \; \N\times\Gamma\times X,
\end{equation}\
by setting, 
$$
P \; = \; \{(n, \alpha, x) \in \N\times\Gamma\times X \; \; | \; \alpha \in \Gamma^{n} \; \; \text{and} \; \; x \in U^{n}_{\alpha} \; \; \text{for some} \; \; U^{n}_{\alpha} \in \mathcal{U}^{n} \; \}.
$$
Since $\coprod_{\alpha \in \Gamma^{n}}U^{n}_{\alpha} = X$ for all $n \in \N$, it follows that the projection map 
$P \longrightarrow \N\times X$
is a surjective, local homeomorphism, i.e.\ it is an \textit{etale map}. 
Now, 
the space $P$ has the structure of a topological poset, defined by setting 
\begin{equation}
(m, \alpha, x) \; < \; (n, \beta, y)
\end{equation}
if $m < n$ and $x = y$ (no relation imposed on $\alpha$ and $\beta$).
We let $P_{\bullet}$ denote the semi-simplicial nerve of this topological poset and let $|P_{\bullet}|$ denote its geometric realization. 
The projection map $P_{m} \longrightarrow X$ (for each $m$) yields an augmented semi-simplicial space $P_{\bullet} \longrightarrow P_{-1}$ with $P_{-1} = X$. 
We have the following lemma from \cite[Corollary 6.5]{GRW 14}:
\begin{lemma} \label{lemma: contractibility of P}
The map $|P_{\bullet}| \longrightarrow P_{-1} = X$ induced by the augmentation is a Serre-fibration with contractible fibres. 
\end{lemma}
Since the augmentation map $|P_{\bullet}| \longrightarrow X$ is a Serre-fibration with contractible fibres, it follows that it admits a global section $\zeta: X \longrightarrow |P_{\bullet}|$.
We choose such a section $\zeta: X \longrightarrow |P_{\bullet}|$.
We will now show how to use this section $\zeta$ to concoct a fibrewsie surgery basis for $(M, \hat{\ell})$. 
\begin{Construction} \label{construction: canonical cover}
Consider the geometric realization 
$|P_{\bullet}| = \left(\coprod_{i=0}^{\infty}P_{i}\times\Delta^{i}\right)/\sim.$
An element in $|P_{\bullet}|$ can be expressed as a tuple 
$$\left(((n_{0}, \dots, n_{m}), \; (\alpha_{0}, \dots, \alpha_{m}), \; x), \; (t_{0}, \dots, t_{m})\right) \in P_{m}\times\Int(\Delta^{m})$$
for some $m \in \Z_{\geq 0}$.
Since $(t_{0}, \dots, t_{m})$ is contained in the interior $\Int(\Delta^{m})$, it follows that the above coordinate expression is unique (the uniqueness is using the fact that $P_{\bullet}$ is a semi-simplicial space and thus has no degenerate simplices).

We define an open cover of $|P_{\bullet}|$, indexed by $\N$, as follows:
For $m \in \N$, let $V_{m} \subset |P_{\bullet}|$ be the set of all points $z \in |P_{\bullet}|$ such that when written in its unique coordinate expression
$$
z \; = \; \left(((n_{0}, \dots, n_{p}), \; (\alpha_{0}, \dots, \alpha_{p}), \; x), \; (t_{0}, \dots, t_{p})\right) \in P_{p}\times\Int(\Delta^{p}),
$$
it satisfies: $n_{i} = m$ for some $i = 0, \dots, p$.
It follows that the collection of subsets $V_{m} \subset |P_{\bullet}|$, $m \in \N$, is an open cover of $|P_{\bullet}|$. 
For any increasing sequence of positive integers $\bar{m} = (m_{0}, \dots, m_{k})$, we let $V_{\bar{m}} = \cap^{k}_{i=0}V_{i}$.

For $m \in \N$, let $P(m)$ denote the intersection $P\cap\left(\{m\}\times\Gamma^{m}\times X\right)$.
For each $m$ we define 
\begin{equation} \label{equation: projection to m map}
\pi_{m}: V_{m} \longrightarrow P(m)
\end{equation}
by the formula, 
$
\pi_{m}\left(((n_{0}, \dots, n_{p}), \; (\alpha_{0}, \dots, \alpha_{p}), x), \; (t_{0}, \dots, t_{p})\right) \; = \; (m, \alpha_{i}, x),
$
where $i$ is the unique integer such that $n_{i} = m$.
For each $\alpha \in \Gamma^{m}$ we define the space,
$V_{m, \alpha} \; := \; \pi^{-1}_{m}\left(\{m\}\times\{\alpha\}\times X\right).$
\end{Construction}
Recall the section $\zeta: X \longrightarrow |P_{\bullet}|$ of the augmentation map $|P_{\bullet}| \longrightarrow X$, whose existence is guaranteed by Lemma \ref{lemma: contractibility of P}. 
For each $m \in \N$ and $\alpha \in \Gamma^{m}$, we define 
\begin{equation}
W_{m, \alpha} := \zeta^{-1}(V_{m, \alpha}) \; \subset \; X.
\end{equation}
The collection 
$\mathcal{W} := \{W_{m, \alpha} \; | \; m \in \N, \; \alpha \in \Gamma^{m}\}$ 
is an open cover of $X$.
We define fibrewise embeddings
$$
\Psi_{m, \alpha}: W_{m, \alpha}\times\Lambda^{m}_{\alpha}\times D^{k}\times D^{d-k+1} \longrightarrow W_{m, \alpha}\times\R^{\infty}
$$
by setting, 
$$
\Psi_{m, \alpha} \; = \; \varphi^{m}_{\alpha}\circ\pi_{m}\circ\zeta|_{W_{m, \alpha}}.
$$
Similarly define $\theta$-structures 
$$
\hat{\psi}_{m, \alpha}: W_{m, \alpha}\times T(\Lambda^{m}_{\alpha}\times D^{k}\times D^{d-k+1}) \longrightarrow \theta^{*}\gamma^{d+1},
$$
by the formula, 
$$
\hat{\psi}_{m, \alpha} \; = \; \hat{\phi}^{m}_{\alpha}\circ D\pi_{m}\circ D\zeta|_{W_{m, \alpha}}.
$$
It follows that the tuple $\left(\mathcal{W}, \Psi, \psi, \Gamma\right)$ is a fibrewise surgery semi-basis, namely that it satisfies conditions (i) and (ii) of Definition \ref{defn: surgery basis fold map}. 
We now argue that it also satisfies condition (iii) as well, namely that 
$$
\Image(\Psi_{m, \alpha})\cap\Image(\Psi_{n, \beta}) = \emptyset 
$$
whenever $(m, \alpha) \neq (n, \beta)$.
Suppose first $m \neq n$. 
By how the initial family $(\mathcal{U}^{n}, \varphi^{n}, \phi^{n}, \Gamma^{n})$ was chosen, it follows that 
$
\Image(\varphi^{m}_{\alpha})\cap\Image(\varphi^{n}_{\beta}) = \emptyset
$
for any $\alpha$ or $\beta$.
By how $\Psi$ was constructed, it follows then that $\Image(\Psi_{m, \alpha})\cap\Image(\Psi_{n, \beta}) = \emptyset$ whenever $m \neq n$.
For the other case ($n = m$, but $\alpha \neq \beta$), observe that 
$W_{m, \alpha}\cap W_{m, \beta} = \emptyset$ whenever $\alpha \neq \beta$, 
which follows from the fact that 
$$
V_{m, \alpha}\cap V_{m, \beta} = \emptyset
$$
for all $m$, whenever $\alpha \neq \beta$.
Since the $\Psi_{m, \alpha}$ are fibrewise embeddings, it is then automatic that 
$$
\Image(\Psi_{m, \alpha})\cap\Image(\Psi_{m, \beta}) = \emptyset
$$
for all $m$, whenever $\alpha \neq \beta$.
We have proven that $\Image(\Psi_{m, \alpha})\cap\Image(\Psi_{m, \beta}) = \emptyset$ in all cases that $(m, \alpha) \neq (n, \beta)$. 
It follows that $(\mathcal{W}, \Psi, \psi, \Gamma)$ is a fibrewise surgery basis for $(M, \hat{\ell})$.
This completes the proof of Theorem \ref{theorem: fibrewise surgery lemma}. 

With the proof of Theorem \ref{theorem: fibrewise surgery lemma} complete, we are now in a position to prove Proposition \ref{proposition: contractibility of the surgery category}, which we do below in the following subsection.
\subsection{Proof of Proposition \ref{proposition: contractibility of the surgery category}} \label{subsection: proof of proposition contractible category}
For the purposes of proving Proposition \ref{proposition: contractibility of the surgery category} it will be convenient to work with a sheaf model for the topological category $\mathcal{S}^{k-1}_{M}$.
For what follows let $\Mfds$ denote the category whose objects are smooth manifolds without boundary, and whose morphisms are given by smooth maps. 
We refer the reader to \cite[Section 2.4]{MW 07} for background on the concordance theory of sheaves on $\Mfds$.
The sheaf model for $\mathcal{S}^{k-1}_{M}$ is defined below:
\begin{defn} \label{defn: top category of multiple surgeries sheaf model}
Let $M$ be as in Definition \ref{defn: top category of multiple surgeries}. 
For $X \in \Mfds$, the 
category $\mathcal{S}^{k-1}_{M}(X)$ is defined as follows.
An object consists of a finite subset $\mb{t} \subset \Omega$, a $\theta$-structure $\hat{\ell}$ on $\mb{t}\times D^{k}\times D^{d-k+1}$, and a fibrewise embedding 
$e: X\times\mb{t}\times D^{k}\times D^{d-k+1} \longrightarrow X\times\R^{\infty}$
subject to the following conditions:
\begin{enumerate} \itemsep.2cm
\item[(i)] $e^{-1}(X\times M) = X\times\mb{t}\times S^{k-1}\times D^{d-k+1}$.
\item[(ii)] 
On each fibre of the projection to $X$, $\hat{\ell}_{M}$ agrees with the restriction of $\hat{\ell}$ to $M$. 
\item[(iii)] For each $x \in X$, let $\widetilde{M}_{x}$ denote the $\theta$-manifold obtained by 
performing $\theta$-surgery on $M_{x}$ via 
$e_{\mb{t}}|_{\{x\}\times\mb{t}\times S^{k-1}\times D^{d-k+1}},$ 
and the $\theta$-structure $\hat{\ell}$. 
We require that $\ell_{\widetilde{M}_{x}}: \widetilde{M}_{x} \longrightarrow B$ be $k$-connected for all $x \in X$.
\end{enumerate}
A morphism from $(\mb{t}, e, \hat{\ell}) \longrightarrow (\mb{s}, e', \hat{\ell}')$ is an injective map $j: \mb{t} \longrightarrow \mb{s}$ such that $j^{*}e' = e$  and $j^{*}\hat{\ell}' = \hat{\ell}$.
\end{defn}

From the above definition, the assignment $X \mapsto \mathcal{S}_{M}^{k-1}(X)$ defines a \textit{category-valued sheaf} on $\Mfds$.
Clearly the topological category $\mathcal{S}^{k-1}_{M}$ is weak homotopy equivalent to the representing space $|\mathcal{S}^{k-1}_{M}|$ (see \cite[Definition 2.6]{MW 07}).
Proposition \ref{proposition: contractibility of the surgery category} asserts that the classifying space $B\mathcal{S}^{k-1}_{M}$ is weakly contractible. 
In order to prove this we need to work with a sheaf model for the classifying space $B\mathcal{S}^{k-1}_{M}$. 
We invoke a general construction from \cite{MW 07}.
Choose once and for all an uncountably infinite set $J$. 
\begin{defn}[Cocycle sheaves] \label{defn: cocycle sheaves}
Let $\mathcal{F}$ be a sheaf on $\Mfds$ taking values in the category of small categories. 
There is an associated set valued sheaf $\beta\mathcal{F}$.
For a manifold $X$, an element $\beta\mathcal{F}(X)$ is a pair $(\mathcal{U}, \Phi)$ where $\mathcal{U} = \{U_{j} | j \in J\}$ is a locally finite open cover of $X$ indexed by $J$, and $\Phi$ is a certain collection of morphisms defined as follows:
Given a non-empty finite subset $R \subset J$, let $U_{R}$ denote the intersection $\cap_{j \in R}U_{j}$. 
Then $\Phi$ is a collection $\varphi_{RS} \in \Mor\mathcal{F}(U_{S})$ indexed by the pairs $R \subseteq S$ of non-empty finite subsets of $J$, subject to the conditions:
\begin{enumerate} \itemsep.2cm
\item[(i)] $\varphi_{RR} = \Id_{c_{R}}$ for an object $c_{R} \in \Ob\mathcal{F}(U_{R})$;
\item[(ii)] For each non-empty finite subset $R \subseteq S$, $\varphi_{RS}$ is a morphism from $c_{S}$ to $c_{R}|_{U_{S}}$; 
\item[(iii)] For all triples $R \subseteq S \subseteq T$ of finite non-empty subsets of $J$, we have 
$
\varphi_{RT} = (\varphi_{RS}|_{U_{T}})\circ\varphi_{ST}.
$
\end{enumerate}
\end{defn}
By \cite[Theorem 4.1.2]{MW 07}, for any category-valued sheaf $\mathcal{F}$ there is a weak homotopy equivalence 
$$
|\beta\mathcal{F}| \simeq B|\mathcal{F}|,
$$
where the right-hand side is the classifying space of the topological category $|\mathcal{F}|$.
By this theorem from \cite{MW 07}, to prove Proposition \ref{proposition: contractibility of the surgery category} it will suffice to prove that the representing space $|\beta(\mathcal{S}^{k-1}_{M})^{\op}|$ is weakly contractible, where $(\mathcal{S}^{k-1}_{M})^{\op}$ is the sheaf that assigns to each manifold $X$, the \textit{opposite category} to $\mathcal{S}^{k-1}_{M}(X)$.
To do this it will suffice to show that the sheaf
$\beta(\mathcal{S}^{k-1}_{M})^{\op}$ is weakly equivalent to the constant sheaf sending any manifold $X$ to a singleton. 
This reduces to the following proposition (compare with \cite[Proof of Proposition 6.7 (Page 916)]{MW 07}):
\begin{proposition} \label{claim: extension of germs}
Let $X$ be given with a closed subset $A$ and a germ $s \in \displaystyle{\colim_{U}}\beta(\mathcal{S}^{k-1}_{M})^{\op}(U)$, where $U$ ranges over the neighborhoods of $A$ in $X$. 
Then $s$ extends to an element of $\beta(\mathcal{S}^{k-1}_{M})^{\op}(X)$.
\end{proposition}
\begin{proof}[Proof of Proposition \ref{claim: extension of germs}]
We will show that this follows as a consequence of Theorem \ref{theorem: fibrewise surgery lemma}. 
Choose an open neighborhood $U$ of $A$ in $X$ such that the germ $s \in \displaystyle{\colim_{U}}\beta(\mathcal{S}^{k-1}_{M})^{\op}(U)$ can be represented by some $s_{0} \in \beta(\mathcal{S}^{k-1}_{M})^{\op}(U)$.
The information contained in $s_{0}$ includes a locally finite covering of $U$ by open subsets $U_{j}$ for $j \in J$. 
It also includes a choice of object $\psi_{RR} \in \Ob\mathcal{S}^{k-1}_{M}(U_{R})$ for each finite nonempty subset $R$ of $J$. 

Next, choose an open $X_{0} \subset X$ such that $U \cup X_{0} = X$ and the closure of $X_{0}$ in $X$ avoids $A$. 
Let $N$ be the open subset of $M\times X_{0}$ obtained by removing from $M\times X_{0}$ the closures of the embedded sphere bundles determined by the various $\varphi_{RR}|_{U_{R}\cap X_{0}}$. 
Consider the projection $\pi: N \longrightarrow X_{0}$, which is a submersion. 
Let $\hat{\ell}: T^{\pi}N\oplus\epsilon^{1} \longrightarrow \theta^{*}\gamma^{d+1}$ be the fibrewise $\theta$-structure obtained by restricting $\hat{\ell}_{M}$ to each fibre of $\pi$.
By assumption, the map $\ell_{M}: M \longrightarrow B$ is $(k-1)$-connected. 
Since each fibre $N_{x} = \pi^{-1}(x)$ is obtained from $M$ by deleting a finite number of disjoint embedded copies of $S^{k-1}\times D^{d-k+1}$,
it then follows by a general position argument that the map $\ell_{x}: N_{x} \longrightarrow B$ is $(k-1)$-connected as well, for all $x \in X_{0}$.
It follows that the pair $(N, \hat{\ell})$ determines an element of the set $\Sub_{d, k-1}(X_{0})$ (see Section \ref{subsection: bases for fibrewise surgery}). 
Since $\pi_{k}(M)$ is finitely generated by assumption it follows that $\pi_{k}(N_{x})$ is finitely generated for all $x \in X$, since $N_{x}$ is obtained from $M$ by deleting a finite number of copies of $S^{k-1}\times D^{d-k+1}$.
We have shown that the element $(N, \hat{\ell}) \in \Sub_{d, k-1}(X_{0})$ satisfies the hypotheses of Theorem \ref{theorem: fibrewise surgery lemma}. 
We may then apply Theorem \ref{theorem: fibrewise surgery lemma} to obtain a fibrewise surgery basis $(\mathcal{V}, \varphi, \phi, \Gamma)$ for $(N, \hat{\ell})$. 

We now use the fibrewise surgery basis $(\mathcal{V}, \varphi, \phi, \Gamma)$ to construct an element of $\beta(\mathcal{S}^{k-1}_{M})^{\op}(X)$  that extends the element $s \in \beta(\mathcal{S}^{k-1}_{M})^{\op}(U)$.
Let us first fix some notation.  
Let $\mathcal{U} = \{U_{i} \; | \; i \in J\}$ be the covering of $U \subset X$ associated to the element $s \in \beta(\mathcal{S}^{k-1}_{M})^{\op}(U)$ that was specified in first paragraph of the proof.
Let $\mathcal{V} = \{V_{i} \; | \; i \in \Gamma\}$ be the covering associated to the chosen fibrewise surgery basis, $(\mathcal{V}, \varphi, \phi, \Gamma)$. 
Since the indexing set $J$ is uncountably infinite we may assume that $\Gamma$ is a subset of $J$ with the property that $U_{i} = \emptyset$ for all $i \in \Gamma$. 
We let $\mathcal{V}' = \{V'_{i} \; | \; i \in J\}$ denote the covering of $X_{0}$ obtained by setting 
$$V'_{i} = 
\begin{cases}
V_{i} &\quad \text{if $i \in \Gamma$,}\\
\emptyset &\quad \text{if $i \notin \Gamma$.}
\end{cases}$$
Defined in this way, it follows that for any $i \in J$ the set $U_{i}$ is non-empty only if $V'_{i}$ is empty. 
With this notation in place, we proceed to define a new covering of $X$. 
For $j \in J$, define 
$$Y_{j} := \begin{cases}
U_{j} &\quad \text{if $U_{j} \neq \emptyset$,}\\
V'_{j} &\quad \text{if $V'_{j} \neq \emptyset$,}\\
\emptyset &\quad \text{else.}
\end{cases}$$
The set $\{Y_{j} \; | \; j \in J\}$ defines a locally finite covering for $X$. 
For finite $R \subset J$ with nonempty $Y_{R}$, we can write $Y_{R} = U_{S}\cap V'_{T}$ for disjoint subsets $S, T$of $R$ with $S\cup T = R$. 
Let $\varphi_{RR} \in \Ob\mathcal{S}_{M}^{k-1}(Y_{R})$ be the object given by the union of $\psi_{SS}|_{Y_{R}}$ and the embeddings $\varphi_{j}|_{Y_{R}}$ for all $j \in T$ (where recall that $\psi_{RR}$ was the object associated to the element $s$ specified in the first paragraph of the proof).
Since these embeddings $\psi_{SS}|_{Y_{R}}$ and the $\varphi_{j}|_{Y_{R}}$ are disjoint by construction, it follows that $\varphi_{RR}$ determines a well defined element of $\Ob\mathcal{S}_{M}^{k-1}(Y_{R})$.
The covering $\{ Y_{j} \; | \; j \in J \; \}$ together with the collection of objects $\varphi_{RR}$ for finite non-empty subsets $R \subset J$ is an element of $\beta(\mathcal{S}^{k-1}_{M})^{\op}(X)$  that extends $s \in \beta(\mathcal{S}^{k-1}_{M})^{\op}(U)$.
This proves Claim \ref{claim: extension of germs}.
\end{proof}
With Claim \ref{claim: extension of germs} established, the proof of Proposition \ref{proposition: contractibility of the surgery category} is now complete.

\subsection{Implementation of fibrewise surgery} \label{subsection: implementation of fibrewise surgery}
In this section we show how to use the above results (namely Proposition \ref{proposition: contractibility of the surgery category}) to prove Theorem \ref{theorem: homotopy colimit decomposition}.
We follow closely the methods of \cite[Section 6]{MW 07} but with a few modifications.

First, let $\mathcal{K}^{\{k\}} \subset \mathcal{K}$ be the full subcategory on those $\mb{t}$ where all elements in $\mb{t}$ are labelled by the integer $k$. 
Let $\mathcal{K}^{k-1, \{k\}^{c}} \subset \mathcal{K}^{k-1}$ denote the full subcategory consisting of those $\mb{s}$ such that all points in $\mb{s}$ have labels in the set $\{k+1, \dots, d-k+1\}$. 
The category $\mathcal{K}^{k-1}$ factors as the product $\mathcal{K}^{k-1} = \mathcal{K}^{k-1, \{k\}^{c}}\times\mathcal{K}^{\{k\}}$, and thus there is a homeomorphism,
\begin{equation} \label{equation: fubini principle for homotopy colimits}
\hocolim_{\mb{t} \in \mathcal{K}^{k-1, \{k\}^{c}}}\hocolim_{\mb{s} \in \mathcal{K}^{\{k\}}}\mathcal{W}^{k-1}_{P, \mb{s}\sqcup\mb{t}} \; \cong \; \hocolim_{\mb{u} \in \mathcal{K}^{k-1}}\mathcal{W}^{k-1}_{P, \mb{u}}.
\end{equation}
By the homotopy invariance of homotopy colimits, to prove Theorem \ref{theorem: homotopy colimit decomposition} it will suffice to show that for any $\mb{t} \in \mathcal{K}^{k-1, \{k\}^{c}}$, the map
\begin{equation} \label{equation: split limit fixed t}
\hocolim_{\mb{s} \in \mathcal{K}^{\{k\}}}\mathcal{W}^{k-1, c}_{P, \mb{s}\sqcup\mb{t}} \; \longrightarrow \; \hocolim_{\mb{s} \in \mathcal{K}^{\{k\}}}\mathcal{W}^{k-1}_{P, \mb{s}\sqcup\mb{t}}
\end{equation}
is a weak homotopy equivalence.
Taking the homotopy colimit of these maps with $\mb{t}$ ranging over $\mathcal{K}^{k-1, \{k\}^{c}}$ will then imply Theorem \ref{theorem: homotopy colimit decomposition}.
We will need to introduce some new constructions.
\begin{defn} \label{defn: special diagrams}
We call a diagram, 
$\xymatrix{
\mb{s} \ar[rr]^{(j_{1}, \varepsilon_{1})} && \mb{t} && \mb{u} \ar[ll]_{(j_{2}, \varepsilon_{2})}, 
}
$ 
in $\mathcal{K}^{\{k\}}$ \textit{special} if 
\begin{enumerate} \itemsep.2cm
\item[(i)] $j_{2}(\mb{u})$ contains $j_{1}(\mb{s})$, and
\item[(ii)] $\varepsilon_{1} = +1$, $\varepsilon_{2} = -1$. 
\end{enumerate}
In this situation we define the space $\mathcal{W}^{k-1, c}_{P, \mb{u}\rightarrow\mb{t}}$ by means of the pullback diagram 
$$
\xymatrix{
\mathcal{W}^{k-1, c}_{P, \mb{u} \rightarrow \mb{t}} \ar[d] \ar[rr] && \mathcal{W}^{k-1, c}_{P, \mb{u}} \ar[d]^{\text{inclusion}} \\
\mathcal{W}^{k-1}_{P, \mb{t}} \ar[rr]^{(k_{2}, \varepsilon_{2})^{*}} && \mathcal{W}^{k-1}_{P, \mb{u}}. 
}
$$
The special diagrams $\mb{s} \longrightarrow \mb{t} \longleftarrow \mb{u}$ with a fixed $\mb{s}$ are objects of a category $\mathcal{D}^{\{k\}}_{\mb{s}}$ where the morphisms are commutative diagrams in $\mathcal{K}^{\{k\}}$ of the form, 
$$
\xymatrix{
\mb{s} \ar[rr] \ar[d]^{=} && \mb{t} \ar[d] && \mb{u} \ar[ll] \ar[d] \\
\mb{s} \ar[rr] && \mb{t}' && \mb{u}' \ar[ll]
}
$$
with \textit{special} rows. 
\end{defn}
The rule taking a special diagram $\mb{s} \longrightarrow \mb{t} \longleftarrow \mb{u}$ to $\mathcal{W}^{k-1, c}_{P, \mb{u}\rightarrow \mb{t}}$ defines a contravariant functor on $\mathcal{D}^{\{k\}}_{\mb{s}}$ (see the discussion in \cite[Page 918]{MW 07}).
There is also a natural transformation from that functor on $\mathcal{D}^{\{k\}}_{\mb{s}}$ to the constant functor with value $\mathcal{W}^{k-1}_{P, \mb{s}}$, determined by the composition 
\begin{equation} \label{equation: forgetful natural transformation}
\mathcal{W}^{k-1, c}_{P, \mb{u}\rightarrow \mb{t}} \longrightarrow \mathcal{W}^{k-1}_{P, \mb{t}}\longrightarrow \mathcal{W}^{k-1}_{P, \mb{s}},
\end{equation} 
for $\mb{s} \longrightarrow \mb{t} \longleftarrow \mb{u}$ in $\mathcal{D}^{\{k\}}_{\mb{s}}$.

The following lemma is the same as \cite[Lemma 6.10]{MW 07}. 
We provide a sketch of the proof.
\begin{lemma} \label{lemma: homotopy colimit over special diagrams}
For any object $\mb{s} \in \mathcal{K}^{\{k\}}$, the natural transformation (\ref{equation: forgetful natural transformation}) induces a weak homotopy equivalence, 
$
\displaystyle{\hocolim_{(\mb{s}\rightarrow \mb{t} \leftarrow \mb{u}) \in \mathcal{D}^{\{k\}}_{\mb{s}}}}\mathcal{W}^{k-1, c}_{P, \mb{u}\rightarrow\mb{t}} \; \stackrel{\simeq} \longrightarrow \mathcal{W}^{k-1}_{P, \mb{s}}.
$
\end{lemma}
\begin{proof}[Proof sketch]
Let $(M, e) := (M, (V, \sigma), e)$ be an element of $\mathcal{W}^{k-1}_{P, \mb{s}}$.
It will suffice to prove that the homotopy fibre of the map 
$$\hocolim_{(\mb{s}\rightarrow \mb{t} \leftarrow \mb{u}) \in \mathcal{D}^{\{k\}}_{\mb{s}}}\mathcal{W}^{k-1, c}_{P, \mb{u} \rightarrow \mb{t}} \; \longrightarrow \;\mathcal{W}^{k-1}_{P, \mb{s}}$$
over $(M, e)$ is weakly contractible. 
By \cite[Lemma 6.11]{MW 07}, this homotopy fibre is weakly homotopy equivalent to the homotopy colimit
$$
\hocolim_{(\mb{s}\rightarrow \mb{t} \leftarrow \mb{u})\in \mathcal{D}^{\{k\}}_{\mb{s}}}\left(\hofibre_{(M, e)}\left[\mathcal{W}^{k-1, c}_{P, \mb{u} \rightarrow \mb{t}} \rightarrow \mathcal{W}^{k-1}_{P, \mb{s}}\right]\right).
$$
By the same argument given in the proof of \cite[Lemma 6.10]{MW 07}, this homotopy colimit can be identified with the classifying space $B\mathcal{S}^{k-1}_{M\setminus\Image(e)}$.
More precisely, this homotopy colimit is actually identified with the classifying space of the edgewise subdivision of $\mathcal{S}^{k-1}_{M\setminus\Image(e)}$, but this classifying space is homotopy equivalent to $B\mathcal{S}^{k-1}_{M\setminus\Image(e)}$ none-the-less. 
The manifold $M\setminus\Image(e)$ satisfies the hypotheses of Proposition \ref{proposition: contractibility of the surgery category} and thus $B\mathcal{S}^{k-1}_{M\setminus\Image(e)}$ is weakly contractible. 
This concludes the proof of the lemma. 
\end{proof}

\begin{defn}
Let $\mathcal{D}^{\{k\}}$ be the category of \textit{all} special diagrams $\mb{s} \rightarrow \mb{t} \leftarrow \mb{u}$ in $\mathcal{K}^{\{k\}}$.
A morphism in $\mathcal{D}^{\{k\}}$ is a commutative diagram 
$$
\xymatrix{
\mb{s} \ar[rr] \ar[d] && \mb{t} \ar[d] && \mb{u} \ar[d] \ar[ll] \\
\mb{s}' \ar[rr] && \mb{t}' && \mb{u}' \ar[ll]
}
$$
in $\mathcal{K}^{\{k\}}$ with special rows. 
The rule taking an object $\mb{s} \rightarrow \mb{t} \leftarrow \mb{u}$ of $\mathcal{D}^{\{k\}}$ to $\mathcal{W}^{k-1, c}_{P, \mb{u} \rightarrow \mb{t}}$ defines a contravariant functor.
We may then form the homotopy colimit,
$
\displaystyle{\hocolim_{(\mb{s}\rightarrow \mb{t} \leftarrow \mb{u})\in \mathcal{D}^{\{k\}}}}\mathcal{W}^{k-1, c}_{P, \mb{u} \rightarrow \mb{t}}.
$
\end{defn}

Each morphism $(j, \varepsilon): \mb{s}' \longrightarrow \mb{s}$ in $\mathcal{K}^{\{k\}}$ induces a functor 
$(j, \varepsilon)^{*}:\mathcal{D}^{\{k\}}_{\mb{s}} \longrightarrow \mathcal{D}^{\{k\}}_{\mb{s}'},$
which in turn induces a map,
$$
\hocolim_{(\mb{s}'\rightarrow \mb{t}' \leftarrow \mb{u}')\in \mathcal{D}^{\{k\}}_{\mb{s}}}\mathcal{W}^{k-1, c}_{P, \mb{u}' \rightarrow \mb{t}'} \; \longrightarrow \; \hocolim_{(\mb{s}\rightarrow \mb{t} \leftarrow \mb{u})\in \mathcal{D}^{\{k\}}_{\mb{s}'}}\mathcal{W}^{k-1, c}_{P, \mb{u} \rightarrow \mb{t}}.
$$
\begin{Construction} \label{Construction: map of double colimits}
We will construct a map 
\begin{equation} \label{equation: comparing colimits}
\hocolim_{\mb{s} \in \mathcal{K}^{\{k\}}}\hocolim_{(\mb{s}\rightarrow \mb{t} \leftarrow \mb{u})\in \mathcal{D}^{\{k\}}_{\mb{s}}}\mathcal{W}^{k-1, c}_{P, \mb{u} \rightarrow \mb{t}} \; \; \longrightarrow \; \; \displaystyle{\hocolim_{(\mb{s}\rightarrow \mb{t} \leftarrow \mb{u})\in \mathcal{D}^{\{k\}}}}\mathcal{W}^{k-1, c}_{P, \mb{u} \rightarrow \mb{t}}.
\end{equation}
In order to define this map it will be useful to construct a particular model for the double homotopy colimit appearing on the left-hand side. 
We build this model as follows.
Let $\mb{D}_{\bullet, \bullet}$ be the bi-simplicial space where $D_{p,q}$ consists of tuples
$$
\left(\mb{s}_{0} \rightarrow \cdots \rightarrow \mb{s}_{p}; \; (\mb{s}_{p} \rightarrow \mb{t}_{0} \leftarrow \mb{u}_{0}) \rightarrow \cdots \rightarrow (\mb{s}_{p} \rightarrow \mb{t}_{q} \leftarrow \mb{u}_{q}) ; \; (M, (V, \sigma), e)\right), 
$$ 
where 
\begin{itemize} \itemsep.2cm
\item $\mb{s}_{0} \rightarrow \cdots \rightarrow \mb{s}_{p}$ is a sequence of morphisms in $\mathcal{K}^{\{k\}}$; 
\item $(\mb{s}_{p} \rightarrow \mb{t}_{0} \leftarrow \mb{u}_{0}) \rightarrow \cdots \rightarrow (\mb{s}_{p} \rightarrow \mb{t}_{q} \leftarrow \mb{u}_{q})$ is a sequence of morphisms in $\mathcal{D}^{\{k\}}_{\mb{s}_{p}}$;
\item $(M, (V, \sigma), e)$ is an element of $\mathcal{W}^{k-1, c}_{P, \mb{u}_{q} \rightarrow \mb{t}_{q}}$.
\end{itemize}
The face and degeneracy maps are defined in the obvious way. 
By unpacking the definition of homotopy colimit (Definition \ref{defn: homotopy colimit}), it can be seen that the geometric realization  $|\mb{D}_{\bullet, \bullet}|$
agrees with the double homotopy colimit, 
$$\hocolim_{\mb{s} \in \mathcal{K}^{\{k\}}}\left[\hocolim_{(\mb{s}\rightarrow \mb{t} \leftarrow \mb{u})\in \mathcal{D}^{\{k\}}_{\mb{s}}}\mathcal{W}^{k-1, c}_{P, \mb{u} \rightarrow \mb{t}}\right].$$
 Let $\widetilde{D}_{\bullet}$ be the simplicial space with $p$-simplices given by tuples
 $$
 \left((\mb{s}_{0} \rightarrow \mb{t}_{0} \leftarrow \mb{u}_{0}) \rightarrow \cdots \rightarrow (\mb{s}_{p} \rightarrow \mb{t}_{p} \leftarrow \mb{u}_{p}); \; (M, (V, \sigma), e)\right)
 $$
 where 
 \begin{itemize} \itemsep.2cm
 \item $(\mb{s}_{0} \rightarrow \mb{t}_{0} \leftarrow \mb{u}_{0}) \rightarrow \cdots \rightarrow (\mb{s}_{p} \rightarrow \mb{t}_{p} \leftarrow \mb{u}_{p})$ is a sequence of morphisms in $\mathcal{D}^{\{k\}}$;
 \item $(M, (V, \sigma), e)$ is an element of $\mathcal{W}^{k-1, c}_{P, \mb{u}_{p} \rightarrow \mb{t}_{p}}$. \end{itemize}
 It follows from the definitions that 
 $ \displaystyle{\hocolim_{(\mb{s}\rightarrow \mb{t} \leftarrow \mb{u})\in \mathcal{D}^{\{k\}}}}\mathcal{W}^{k-1, c}_{P, \mb{u} \rightarrow \mb{t}}$ agrees with the geometric realization $|\widetilde{\mb{D}}_{\bullet}|$. 

Now, any $(p, q)$-simplex 
$$\left(\mb{s}_{0} \rightarrow \cdots \rightarrow \mb{s}_{p}; \; (\mb{s}_{p} \rightarrow \mb{t}_{0} \leftarrow \mb{u}_{0}) \rightarrow \cdots \rightarrow (\mb{s}_{p} \rightarrow \mb{t}_{q} \leftarrow \mb{u}_{q}) ; \; (M, (V, \sigma), e)\right) \in \mb{D}_{p,q}$$
determines a $(p+q+1)$-simplex in $\widetilde{\mb{D}}_{p+q+1}$, namely the element
$$
\left((\mb{s}_{0} \rightarrow \mb{t}_{0} \leftarrow \mb{u}_{0}) \rightarrow \cdots \rightarrow  (\mb{s}_{p} \rightarrow \mb{t}_{0} \leftarrow \mb{u}_{0}) \stackrel{\Id} \rightarrow (\mb{s}_{p} \rightarrow \mb{t}_{0} \leftarrow \mb{u}_{0}) \rightarrow \cdots \rightarrow (\mb{s}_{p} \rightarrow \mb{t}_{q} \leftarrow \mb{u}_{q}) ; \; (M, (V, \sigma), e)\right),
$$
where the maps $\mb{s}_{i} \rightarrow \mb{t}_{0}$ for $i \leq p$ are defined by precomposing $\mb{s}_{p} \rightarrow \mb{t}_{0}$ with $\mb{s}_{i} \rightarrow \mb{s}_{p}$.
Notice that in the above formula the middle map 
$(\mb{s}_{p} \rightarrow \mb{t}_{0} \leftarrow \mb{u}_{0}) \stackrel{\Id} \longrightarrow (\mb{s}_{p} \rightarrow \mb{t}_{0} \leftarrow \mb{u}_{0})$ 
is the identity morphism. 
For each $(p, q)$ this association defines a map,
$
F_{p, q}: \mb{D}_{p,q} \longrightarrow \widetilde{\mb{D}}_{p+q+1}.
$
Using these maps we construct a homotopy
$$
[0, 1]\times\mb{D}_{p, q}\times\Delta^{p}\times\Delta^{q} \longrightarrow \widetilde{\mb{D}}_{p+q+1}\times\Delta^{p+q+1}
$$
by the formula
$$
\left(r, (\bar{x}; \; t; \; s)\right) \; \mapsto \; (F_{p, q}(\bar{x}); \; (1-r) t, \; r s),
$$
where $((1-r)t, \; rs)$ is considered to be an element of $\Delta^{p+q+1}$.
These maps for all $(p, q)$ geometrically realize to yield a homotopy 
\begin{equation}
j_{t}: |\mb{D}_{\bullet, \bullet}| \longrightarrow |\widetilde{\mb{D}}_{\bullet}|, \quad t \in [0, 1].
\end{equation}
The promised map (\ref{equation: comparing colimits}) is defined to be $j_{0}$.
\end{Construction} 
The proof of the following lemma is similar to \cite[Proof of Theorem 6.5 (page 920)]{MW 07}.
\begin{lemma} \label{lemma: restrict to k equivalence}
The inclusion, 
$\displaystyle{\hocolim_{\mb{s} \in \mathcal{K}^{\{k\}}}}\mathcal{W}^{k-1, c}_{P, \mb{s}} \; \longrightarrow \; \displaystyle{\hocolim_{\mb{s} \in \mathcal{K}^{\{k\}}}}\mathcal{W}^{k-1}_{P, \mb{s}},$
is a weak homotopy equivalence.
\end{lemma}
\begin{proof}
For each $\mb{s} \in \mathcal{K}^{\{k\}}$, we denote 
$$X_{\mb{s}} := \hocolim_{(\mb{s}\rightarrow \mb{t}\leftarrow \mb{u})\in\mathcal{D}_{\mb{s}}}\mathcal{W}^{k-1, c}_{P, \mb{u}\rightarrow \mb{t}} \quad \text{and} \quad X := \hocolim_{(\mb{s} \rightarrow \mb{t} \leftarrow \mb{u})\in\mathcal{D}^{\{k\}}}\mathcal{W}^{k-1, c}_{P, \mb{u} \rightarrow \mb{t}}.$$
Construction (\ref{Construction: map of double colimits}) yields the homotopy  
$
j_{t}: \displaystyle{\hocolim_{\mb{s}\in \mathcal{K}^{\{k\}}}}X_{\mb{s}} \longrightarrow X, \quad t \in [0, 1].
$
Consider the diagram 
\begin{equation} \label{equation: homotopy commutative diagram}
\xymatrix{
\displaystyle{\hocolim_{\mb{s} \in \mathcal{K}^{\{k\}}}}X_{\mb{s}} \ar[d]^{j_{t}} &&\\
X \ar[drr]^{\pi_{2}} \ar[d]^{\pi_{1}} && \\
\displaystyle{\hocolim_{\mb{s} \in \mathcal{K}^{\{k\}}}}\mathcal{W}^{k-1}_{P, \mb{s}} && \displaystyle{\hocolim_{\mb{u} \in \mathcal{K}^{\{k\}}}}\mathcal{W}^{k-1, c}_{P, \mb{u}}, \ar[ll]_{\iota} \ar[uull]_{s}
}
\end{equation}
where $\iota$ is the inclusion, and $\pi_{1}$ and $\pi_{2}$ are the maps induced by the functors 
$$p_{1}: \mathcal{D}^{\{k\}} \longrightarrow \mathcal{K}^{\{k\}} \quad \text{and} \quad p_{2}: \mathcal{D}^{\{k\}} \longrightarrow \mathcal{K}^{\{k\}}$$ 
given by the projections
\begin{equation} \label{equation: projection transformations}
(\mb{s} \rightarrow \mb{t} \leftarrow \mb{u}) \mapsto \mb{s} \quad \text{and} \quad (\mb{s} \rightarrow \mb{t} \leftarrow \mb{u})\mapsto \mb{u}
\end{equation}
respectively. 
We will explain the map $s$ (which will so happen to be a section of $\pi_{2}\circ j_{1}$) shortly. 
Now, the two functors $p_{1}, p_{2}: \mathcal{D}^{\{k\}} \longrightarrow \mathcal{K}^{\{k\}}$ defined in (\ref{equation: projection transformations}) are connected by a zig-zag of natural transformations. 
Namely, there is a natural transformation from the functor $p_{1}$ to 
$$
p: \mathcal{D}^{\{k\}} \longrightarrow \mathcal{K}^{\{k\}}, \quad (\mb{s} \rightarrow \mb{t} \leftarrow \mb{u}) \mapsto \mb{t},
$$
defined by associating to the object $(\mb{s} \rightarrow \mb{t} \leftarrow \mb{u})$ the arrow $\mb{s} \longrightarrow \mb{t}$. 
Similarly, there is a natural transformation from $p$ to $p_{2}$ defined in the same way. 
This zig-zag of natural transformations 
$$
p_{1} \Longleftarrow p \Longrightarrow p_{2}
$$
yields a homotopy between the maps $\iota\circ \pi_{2}$ and $\pi_{1}$ in the diagram (\ref{equation: homotopy commutative diagram}), and thus the lower triangle is homotopy commutative. 

Now, the composite 
$\pi_{1}\circ j_{0}$ is just the homotopy colimit (taken over $\mb{s} \in \mathcal{K}^{\{k\}}$) of the weak homotopy equivalences from Lemma \ref{lemma: homotopy colimit over special diagrams},
$
\displaystyle{\hocolim_{(\mb{s}\rightarrow \mb{t}\leftarrow \mb{u})\in\mathcal{D}_{\mb{s}}}}\mathcal{W}^{k-1, c}_{P, \mb{u}\rightarrow \mb{t}} \; \stackrel{\simeq} \longrightarrow \; \mathcal{W}^{k-1}_{P, \mb{s}}.
$
By the homotopy invariance of homotopy colimits it follows that $\pi_{1}\circ j_{0}$ is a weak homotopy equivalence. 

By the homotopy commutativity of the lower triangle of (\ref{equation: homotopy commutative diagram}), it follows that the composite
\begin{equation} \label{equation: composite weak equivalence}
\iota\circ\pi_{2}\circ j_{0}: \displaystyle{\hocolim_{\mb{s} \in \mathcal{K}^{\{k\}}}}X_{\mb{s}} \longrightarrow \displaystyle{\hocolim_{\mb{s} \in \mathcal{K}^{\{k\}}}}\mathcal{W}^{k-1}_{P, \mb{s}}
\end{equation}
is a weak homotopy equivalence as well;
this implies that $\iota$ induces a surjection on all homotopy groups $\pi_{i}(\text{--})$.
To finish the proof we will need to show that it induces an injection on homotopy groups; to do this we need to use the map $s$, which still needs to be described. 
Indeed, for each $\mb{s} \in \mathcal{K}^{\{k\}}$ there is a trivial special diagram $\mb{s} \rightarrow \mb{s} \leftarrow \mb{s}$ in $\mathcal{D}_{\mb{s}}$.
For each $\mb{s}$, the association 
$$\mb{s} \mapsto (\mb{s} \rightarrow \mb{s} \leftarrow \mb{s})$$ 
determines the natural transformation,
$
\mathcal{W}^{k-1, c}_{P, \mb{s}} \longrightarrow \mathcal{W}^{k-1, c}_{P, \mb{s} \rightarrow \mb{s}}.
$
The map $s$ in diagram (\ref{equation: homotopy commutative diagram}) is defined to be the homotopy colimit over $\mb{s} \in \mathcal{K}^{\{k\}}$ of these maps. 
By unpacking the definition of the homotopy $j_{t}$ in Construction \ref{Construction: map of double colimits} it follows that,
$$
\pi_{2}\circ j_{1}\circ s = \Id,
$$
and thus we see that $s$ is a section of $\pi_{2}\circ j_{1}$.
This proves that $\pi_{2}\circ j_{1}$ induces a surjection on all homotopy groups (and thus $\pi_{2}\circ j_{t}$ induces a surjection on homotopy groups for all $t \in [0, 1]$). 
From the weak homotopy equivalence (\ref{equation: composite weak equivalence}) it follows that $\pi_{2}\circ j_{t}$ induces an injection (for all $t \in [0, 1]$), and thus we have proven that $\pi_{2}\circ j_{t}$ is a weak homotopy equivalence. 
In view of (\ref{equation: composite weak equivalence}) again, the two-out-of-three property then implies that $\iota$ is a weak homotopy equivalence. 
\end{proof}

Below we show how to use Lemma \ref{lemma: restrict to k equivalence} to prove Theorem \ref{theorem: fibrewise surgery}.
\begin{proof}[Proof of Theorem \ref{theorem: fibrewise surgery}]
As discussed in the beginning of the subsection, 
to prove the theorem it will suffice to show that for any $\mb{t} \in \mathcal{K}^{k-1, \{k\}^{c}}$, the map
$$
\hocolim_{\mb{s} \in \mathcal{K}^{\{k\}}}\mathcal{W}^{k-1, c}_{P, \mb{s}\sqcup\mb{t}} \; \longrightarrow \; \hocolim_{\mb{s} \in \mathcal{K}^{\{k\}}}\mathcal{W}^{k-1}_{P, \mb{s}\sqcup\mb{t}}
$$
is a weak homotopy equivalence.
With this weak homotopy equivalence established, the desired weak homotopy equivalence is obtained by taking the homotopy colimit over $\mb{t} \in \mathcal{K}^{k-1, \{k\}^{c}}$.

Fix $\mb{t} \in \mathcal{K}^{k-1, \{k\}^{c}}$. 
For $\mb{s} \in \mathcal{K}^{\{k\}}$, consider the map 
\begin{equation} \label{equation: localization sequence in t}
\mathcal{W}^{k-1}_{P, \mb{s}\sqcup\mb{t}} \longrightarrow \mathcal{W}_{\loc, \mb{t}},
\end{equation}
defined by sending $(M, (V, \sigma), e) \in \mathcal{W}^{k-1}_{P, \mb{s}\sqcup\mb{t}}$ to $((V, \sigma)|_{\mb{t}}, e|_{\mb{t}})$, which is the component of $((V, \sigma), e)$ that corresponds to the subset $\mb{t}$. 
It is easily verified that (\ref{equation: localization sequence in t}) is a Serre fibration. 
Let $\widehat{\mathcal{W}}^{k-1}_{P, \mb{s}; \mb{t}}$ denote the fibre of this map over some point 
\begin{equation} \label{equation: point under the fibre}
((V', \sigma'), e') \in \mathcal{W}_{\loc, \mb{t}}.
\end{equation}
Keeping $\mb{t}$ constant and taking the homotopy colimit over $\mb{s} \in \mathcal{K}^{\{k\}}$, we obtain a fibre sequence
$$
\hocolim_{\mb{s} \in \mathcal{K}^{\{k\}}}\widehat{\mathcal{W}}^{k-1}_{P, \mb{s}; \mb{t}} \longrightarrow \hocolim_{\mb{s} \in \mathcal{K}^{\{k\}}}\mathcal{W}^{k-1}_{P, \mb{s}\sqcup\mb{t}} \longrightarrow \mathcal{W}_{\loc, \mb{t}}.
$$
We obtain a similar fibre sequence involving $\displaystyle{\hocolim_{\mb{s} \in \mathcal{K}^{\{k\}}}}\mathcal{W}^{k-1, c}_{P, \mb{s}\sqcup\mb{t}}$ and a map of fibre sequences
\begin{equation} \label{equation: map of fibre sequences}
\xymatrix{
\displaystyle{\hocolim_{\mb{s} \in \mathcal{K}^{\{k\}}}}\widehat{\mathcal{W}}^{k-1, c}_{P, \mb{s}; \mb{t}} \ar[d] \ar[rr] && \displaystyle{\hocolim_{\mb{s} \in \mathcal{K}^{\{k\}}}}\widehat{\mathcal{W}}^{k-1}_{P, \mb{s}; \mb{t}}  \ar[d] \\
\displaystyle{\hocolim_{\mb{s} \in \mathcal{K}^{\{k\}}}}\mathcal{W}^{k-1, c}_{P, \mb{s}\sqcup\mb{t}} \ar[d]  \ar[rr] && \displaystyle{\hocolim_{\mb{s} \in \mathcal{K}^{\{k\}}}}\mathcal{W}^{k-1}_{P, \mb{s}\sqcup\mb{t}} \ar[d] \\
\mathcal{W}_{\loc, \mb{t}} \ar[rr]^{\Id} && \mathcal{W}_{\loc, \mb{t}}.
}
\end{equation}
We claim that the top-horizontal map is a weak homotopy equivalence. 
Indeed, $\widehat{\mathcal{W}}^{k-1}_{P, \mb{s}; \mb{t}}$ can be identified with the subspace of $\mathcal{W}^{k-1}_{P, \mb{s}}$ consisting of those 
$(M, (V, \sigma), e) \in \mathcal{W}^{k-1}_{P, \mb{s}\sqcup\mb{t}}$ for which the manifold $M$ contains
$
e'(S(V^{-})\times_{\mb{t}}D(V^{+})).
$
Setting $P' = P \sqcup e'(S(V^{-})\times_{\mb{t}}S(V^{+}))$, it follows that 
there are homeomorphisms 
$$\begin{aligned}
\widehat{\mathcal{W}}^{k-1}_{P, \mb{s}; \mb{t}} &\cong \mathcal{W}^{k-1}_{P', \mb{s}}, \\
\widehat{\mathcal{W}}^{k-1, c}_{P, \mb{s}; \mb{t}} &\cong \mathcal{W}^{k-1, c}_{P', \mb{s}},
\end{aligned}
$$
(see also Proposition \ref{proposition: fibre identification}). 
By these homeomorphisms the top-horizontal map of (\ref{equation: map of fibre sequences}) can be identified with the inclusion 
$$
\displaystyle{\hocolim_{\mb{s} \in \mathcal{K}^{\{k\}}}}\mathcal{W}^{k-1, c}_{P', \mb{s}} \; \hookrightarrow \;\displaystyle{\hocolim_{\mb{s} \in \mathcal{K}^{\{k\}}}}\mathcal{W}^{k-1}_{P', \mb{s}}, 
$$
and thus by Lemma \ref{lemma: restrict to k equivalence} it is a weak homotopy equivalence (Lemma \ref{lemma: restrict to k equivalence} holds for all choice of $P' \in \mathcal{M}_{\theta}$).
Since the top-horizontal map in (\ref{equation: map of fibre sequences}) is a weak homotopy equivalence, it follows that the middle-horizontal map is a weak homotopy equivalence as well, since the vertical columns are fibre sequences. 
This concludes the proof of Theorem \ref{theorem: fibrewise surgery}. 
\end{proof}

\section{Decomposing the Localization Sequence} \label{section: A Fibre Sequence}
In this section we embark on the proof of Theorem \ref{theorem: localization sequence (homotopy colimit)}. 
This theorem asserts that the map 
$$
\hocolim_{\mb{t} \in \mathcal{K}^{k-1}}\mathcal{W}^{k-1, c}_{P, \mb{t}} \; \longrightarrow \;  \hocolim_{\mb{t} \in \mathcal{K}^{k-1}}\mathcal{W}^{\{k, d-k-1\}}_{\locc, \mb{t}}
$$
is a quasi-fibration. 
Let $\mb{t} \in \mathcal{K}^{k-1}$. 
We may uniquely write $\mb{t} = \mb{t}'\sqcup \mb{s}$ with $\mb{t}' \in \mathcal{K}^{k}$ and $\mb{s} \in \mathcal{K}^{\{k, d-k+1\}}$. 
Consider the map
\begin{equation} \label{equation: forgetful map}
\mathcal{W}^{k-1, c}_{P, \mb{t}'\sqcup \mb{s}} \longrightarrow \mathcal{W}^{\{k, d-k-1\}}_{\locc, \mb{s}}, \quad (M, (V, \sigma), e) \mapsto ((V, \sigma), e).
\end{equation}
This map is easily seen to be a Serre-fibration. 
We will need to analyze its fibres.
\begin{defn} \label{defn: fibre sheaf}
For each $\mb{s} \in \mathcal{K}^{\{k, d-k+1\}}$, let us fix once and for all an element 
$$V_{\mb{s}} := ((V_{\mb{s}}, \sigma_{\mb{s}}), e_{\mb{s}}) \in \mathcal{W}^{\{k, d-k-1\}}_{\locc, \mb{s}}.$$ 
In the case that $\mb{s} = \emptyset$, we will still denote the single element in $\mathcal{W}^{\{k, d-k-1\}}_{\locc, \emptyset}$ by $\emptyset$. 
For each $\mb{t} \in \mathcal{K}^{k}$, 
we let $\widehat{\mathcal{W}}^{k-1, c}_{P, \mb{t}; \mb{s}}$ denote the fibre of the map
$
\mathcal{W}^{k-1, c}_{P, \mb{t}\sqcup\mb{s}} \longrightarrow \mathcal{W}^{\{k, d-k-1\}}_{\locc, \mb{s}}
$
over the element $V_{\mb{s}} \in \mathcal{W}^{\{k, d-k-1\}}_{\locc, \mb{s}}$. 
\end{defn}

The lemma below follows from gluing together the Serre fibrations 
$$
\mathcal{W}^{k-1, c}_{P, \mb{t}\sqcup \mb{s}} \longrightarrow \mathcal{W}^{\{k, d-k-1\}}_{\locc, \mb{s}},
$$
and letting $\mb{t}$ range over $\mathcal{K}^{k}$, while keeping $\mb{s}$ constant.
\begin{lemma} \label{corollary: hocolimit of fibration}
For all $\mb{s} \in  \mathcal{K}^{\{k, d-k+1\}}$,
the following sequence is a homotopy fibre sequence, 
$$
\xymatrix{
 \displaystyle{\hocolim_{\mb{t} \in \mathcal{K}^{k}}}\widehat{\mathcal{W}}^{k-1, c}_{P, \mb{t}; \mb{s}} \ar[rr] && \displaystyle{\hocolim_{\mb{t} \in \mathcal{K}^{k}}}\mathcal{W}^{k-1, c}_{P, \mb{t} \sqcup\mb{s}} \ar[rr] && \mathcal{W}^{\{k, d-k-1\}}_{\locc, \mb{s}}.
}
$$
\end{lemma}

We now state the main theorem of this section. 
Its proof is carried out over the course of Sections \ref{section: Cobordism Categories and Handle Attachments} through \ref{section: higher index handles}.  
\begin{theorem} \label{proposition: homotopy equivalence of transition maps}
Let $k < d/2$.
Then for any morphism $(j, \varepsilon): \mb{s} \longrightarrow \mb{s}'$ in $\mathcal{K}^{\{k, d-k+1\}}$
the induced map
$$
(j, \varepsilon)^{*}: \displaystyle{\hocolim_{\mb{t} \in \mathcal{K}^{k}}}\widehat{\mathcal{W}}^{k-1, c}_{P, \mb{t}; \mb{s'}} \; \longrightarrow \; \displaystyle{\hocolim_{\mb{t} \in \mathcal{K}^{k}}}\widehat{\mathcal{W}}^{k-1, c}_{P, \mb{t}; \mb{s}}
$$
 is a weak homotopy equivalence. 
\end{theorem}

By taking the homotopy colimit of the weak homotopy equivalences from Theorem \ref{proposition: homotopy equivalence of transition maps} over $\mb{s} \in \mathcal{K}^{\{k, d-k+1\}}$ we obtain the following corollary.
\begin{corollary} \label{corollary: homology fibration in the limit}
Let $k < d/2$.
Then the map 
$$
\displaystyle{\hocolim_{\mb{s} \in \mathcal{K}^{\{k, d-k+1\}}}}\displaystyle{\hocolim_{\mb{t} \in \mathcal{K}^{k}}}\mathcal{W}^{k-1, c}_{P, \mb{t}\sqcup \mb{s}} \; \longrightarrow \;  \displaystyle{\hocolim_{\mb{s} \in \mathcal{K}^{\{k, d-k+1\}}}}\mathcal{W}^{\{k, d-k-1\}}_{\locc, \mb{s}}
$$
is a quasi-fibration.
The fibre over $\emptyset$ 
is given by the space,
$\displaystyle{\hocolim_{\mb{t} \in \mathcal{K}^{k}}}\widehat{\mathcal{W}}^{k-1, c}_{P, \mb{t}; \emptyset}.$
\end{corollary}
\begin{proof}
Each morphism $(j, \varepsilon): \mb{s} \longrightarrow \mb{s}'$ in  $\mathcal{K}^{\{k, d-k+1\}}$ induces a map of fibre sequences
$$\xymatrix{
\displaystyle{\hocolim_{\mb{t} \in \mathcal{K}^{k}}}\widehat{\mathcal{W}}^{k-1, c}_{P, \mb{t}; \mb{s}'} \ar[d] \ar[rr] &&  \displaystyle{\hocolim_{\mb{t} \in \mathcal{K}^{k}}}\widehat{\mathcal{W}}^{k-1, c}_{P, \mb{t}; \mb{s}}  \ar[d]   \\
\displaystyle{\hocolim_{\mb{t} \in \mathcal{K}^{k}}}\mathcal{W}^{k-1, c}_{P, \mb{t}\sqcup \mb{s}'} \ar[rr] \ar[d] &&  \displaystyle{\hocolim_{\mb{t} \in \mathcal{K}^{k}}}\mathcal{W}^{k-1, c}_{P, \mb{t}\sqcup \mb{s}} \ar[d]  \\
 \mathcal{W}^{\{k, d-k-1\}}_{\locc, \mb{s}'} \ar[rr] &&  \mathcal{W}^{\{k, d-k-1\}}_{\locc, \mb{s}}
}
$$
By Theorem \ref{proposition: homotopy equivalence of transition maps}, the top horizontal map is a weak homotopy equivalence. 
As a consequence of this, it follows from Proposition \ref{proposition: homotopy colimit fibre sequence} that the map 
$$
\displaystyle{\hocolim_{\mb{s} \in \mathcal{K}^{(k, d-k+1)}}}\displaystyle{\hocolim_{\mb{t} \in \mathcal{K}^{k}}}\mathcal{W}^{k-1, c}_{P, \mb{t}\sqcup \mb{s}} \; \longrightarrow \;  \displaystyle{\hocolim_{\mb{s} \in \mathcal{K}^{\{k, d-k+1\}}}}\mathcal{W}^{\{k, d-k-1\}}_{\locc, \mb{s}}
$$
is a quasi-fibration, with fibre over $\emptyset$ given by 
$
 \displaystyle{\hocolim_{\mb{t} \in \mathcal{K}^{k}}}\widehat{\mathcal{W}}^{k-1, c}_{P, \mb{t}; \emptyset}.
$
This concludes the proof of the corollary.
\end{proof}

Using the above result, Theorem \ref{theorem: localization sequence (homotopy colimit)} follows by making a few simple identifications. 
\begin{proof}[Proof of Theorem \ref{theorem: localization sequence (homotopy colimit)} (assuming Proposition \ref{proposition: homotopy equivalence of transition maps})]
We first observe that 
$$\displaystyle{\hocolim_{\mb{t} \in \mathcal{K}^{k}}}\widehat{\mathcal{W}}^{k-1, c}_{P, \mb{t}; \emptyset} \; = \; \displaystyle{\hocolim_{\mb{t} \in \mathcal{K}^{k}}}\mathcal{W}^{k}_{P, \mb{t}}.
$$
This identification follows immediately from the definition of the spaces. 
From this it follows that the fibre of 
$$
\displaystyle{\hocolim_{\mb{s} \in \mathcal{K}^{\{k, d-k+1\}}}}\displaystyle{\hocolim_{\mb{t} \in \mathcal{K}^{k}}}\mathcal{W}^{k-1, c}_{P, \mb{t}\sqcup \mb{s}} \; \longrightarrow \;  \displaystyle{\hocolim_{\mb{s} \in \mathcal{K}^{\{k, d-k+1\}}}}\mathcal{W}^{\{k, d-k-1\}}_{\locc, \mb{s}}
$$
over the element $\emptyset$ is given by the space $\displaystyle{\hocolim_{\mb{t} \in \mathcal{K}^{k}}}\mathcal{W}^{k}_{P, \mb{t}}$.
Using the product factorization of the indexing categories $\mathcal{K}^{k} \cong \mathcal{K}^{k-1}\times\mathcal{K}^{\{k, d-k+1\}}$, we obtain the following homeomorphism,
$$\displaystyle{\hocolim_{\mb{s} \in \mathcal{K}^{(k, d-k+1)}}}\displaystyle{\hocolim_{\mb{t} \in \mathcal{K}^{k}}}\mathcal{W}^{k-1, c}_{P, \mb{t}\sqcup\mb{s}} \; \cong \; 
\displaystyle{\hocolim_{\mb{u} \in \mathcal{K}^{k-1}}}\mathcal{W}^{k-1, c}_{P, \mb{u}}.$$
Theorem \ref{theorem: localization sequence (homotopy colimit)} then follows by plugging the above identifications in to the Serre fibration obtained as a result of Corollary \ref{corollary: homology fibration in the limit}.
\end{proof}

\section{Cobordism Categories and Stability} \label{section: Cobordism Categories and Handle Attachments}
In this section we begin to tackle Theorem \ref{proposition: homotopy equivalence of transition maps} which was the main technical device used to prove Theorem \ref{theorem: localization sequence (homotopy colimit)}. 
Let $k < d/2$.
Theorem \ref{proposition: homotopy equivalence of transition maps} states that for any 
morphism $(j, \varepsilon): \mb{s} \longrightarrow \mb{s}'$ in $\mathcal{K}^{\{k, d-k+1\}}$, the induced map \begin{equation} \label{equation: homotopy colimit morphism map}
(j, \varepsilon)^{*}: \hocolim_{\mb{t} \in \mathcal{K}^{k}}\widehat{\mathcal{W}}^{k-1, c}_{P, \mb{t}; \mb{s}'} \; \longrightarrow \; \hocolim_{\mb{t} \in \mathcal{K}^{k}}\widehat{\mathcal{W}}^{k-1, c}_{P, \mb{t}; \mb{s}},
\end{equation}
 is a weak homotopy equivalence. 
Below, we define a generalization of these maps and will prove a generalization of Theorem \ref{proposition: homotopy equivalence of transition maps}.
We must first identify the spaces $\widehat{\mathcal{W}}^{k-1, c}_{P, \mb{t}; \mb{s}}$ (with $\mb{t} \in \mathcal{K}^{k}$ and $\mb{s} \in \mathcal{K}^{\{k, d-k+1\}}$) with something else a bit more familiar. 
\begin{defn} \label{defn: object S_t}
For each $\mb{s} \in \mathcal{K}^{\{k, d-k+1\}}$, fix once and for all an element 
$S_{\mb{s}} \in \mathcal{M}_{\theta}$ 
with the following properties:
(i) $S_{\mb{s}}\cap P = \emptyset$, (ii) $S_{\mb{s}}$ is diffeomorphic to $\mb{s}\times S^{k-1}\times S^{d-k}$, and 
(iii) the $\theta$-structure $\hat{\ell}_{S_{\mb{s}}}$ admits an extension to the $\theta$-structure on $\mb{s}\times D^{k}\times D^{d-k+1}$ (where we use a diffeomorphism $\mb{s}\times S^{k-1}\times S^{d-k} \cong S_{\mb{s}}$ from (ii) to identify $S_{\mb{s}}$ with $\mb{s}\times S^{k-1}\times S^{d-k}$). 
It follows that the $(d-1)$-dimensional manifold $P\sqcup S_{\mb{s}}$ equipped with the structure $\hat{\ell}_{P}\sqcup\hat{\ell}_{S_{\mb{s}}}$,
 defines an element of the space $\mathcal{M}_{\theta}$, see Definition \ref{defn: space of closed theta manifolds}. 
\end{defn}

\begin{proposition} \label{proposition: fibre identification}
For all $\mb{t} \in \mathcal{K}^{k}$ and $\mb{s} \in \mathcal{K}^{\{k, d-k+1\}}$ there is a homotopy equivalence, 
$$
\mathcal{W}^{k}_{P\sqcup S_{\mb{s}}, \; \mb{t}} \stackrel{\simeq} \longrightarrow \widehat{\mathcal{W}}^{k-1, c}_{P, \mb{t}; \mb{s}}.
$$
\end{proposition}
\begin{proof}
Fix $(V, \sigma) \in (G^{\mf}_{\theta}(\R^{\infty})_{\loc}^{\{k, d-k+1\}})^{\mb{s}}$. 
Choose an embedding 
$$e: D(V^{-})\times_{\mb{s}}D(V^{+}) \longrightarrow (-\infty, 0]\times\R^{\infty-1}$$ 
with image disjoint from $P$, and with 
$
e(S(V^{-})\times_{\mb{s}}S(V^{+}) \; = \; S_{\mb{s}}.
$ 
We define a map 
\begin{equation} \label{equation: glue in surgery data map}
\mathcal{W}^{k}_{P\sqcup S_{\mb{s}}, \; \mb{t}} \longrightarrow \widehat{\mathcal{W}}^{k-1, c}_{P, \mb{t}; \mb{s}}
\end{equation}
by sending $(M, (V', \sigma'), e')$ to the element,
$$\left(M\cup e(S(V^{-})\times_{\mb{s}}D(V^{+})), \; (V'\sqcup V, \sigma\sqcup\sigma'), \; e' \sqcup e\right).$$
Since $\ell_{M}: M \longrightarrow B$ is $k$-connected, it follows that the element 
$$
\left(M\cup e(S(V^{-})\times_{\mb{s}}D(V^{+})), \; (V'\sqcup V, \sigma\sqcup\sigma'), \; e' \sqcup e\right)
$$
is contained in the space $\widehat{\mathcal{W}}^{k-1, c}_{P, \mb{t}; \mb{s}}$, and thus the map is well-defined.
A homotopy inverse  
\begin{equation} \label{equation: homotopy inverse cut out map}
\widehat{\mathcal{W}}^{k-1, c}_{P, \mb{t}; \mb{s}} \longrightarrow \mathcal{W}^{k}_{P\sqcup S_{\mb{s}}, \; \mb{t}}
\end{equation}
to (\ref{equation: glue in surgery data map}) by sending $(M, (V, \sigma), e)$ to  
$
\left(M\setminus e(S(V^{-})\times_{\mb{s}}D(V^{+})), \; (V, \sigma)|_{\mb{t}}, \; e|_{D(V^{-})\times_{\mb{t}}D(V^{+})}\right).
$
Checking that this map is indeed a homotopy inverse is straight-forward and we leave it to the reader. 
\end{proof}

We will now reinterpret the maps (\ref{equation: homotopy colimit morphism map}) using the homeomorphism (\ref{equation: glue in surgery data map}) from the above proposition. 
We will actually consider a generalization of these maps and will prove a generalization of Theorem \ref{proposition: homotopy equivalence of transition maps}.
We will re-interpret these maps as being induced by morphisms from a cobordism category, with objects given by $(d-1)$-dimensional closed manifolds and morphisms given by $d$-dimensional cobordisms defined below. 
\begin{defn} \label{defn: cobordism category}
The (non-unital) topological category $\Cob_{\theta, d}$ has $\mathcal{M}_{\theta}$ (see Definition \ref{defn: space of closed theta manifolds}) for its space of objects. 
A morphism $W: P \rightsquigarrow Q$ is a compact $d$-dimensional submanifold
$$
W \subset [0, 1]\times\R^{\infty-1}
$$
equipped with a $\theta$-structure $\hat{\ell}_{W}: TW\oplus\epsilon^{1} \longrightarrow \theta^{*}\gamma^{d+1}$, for which there exists $0 < \varepsilon < 1/2$ such that 
$$\begin{aligned}
W\cap\left(\{0, 1\}\times\R^{\infty-1}\right) &= \partial W, \\
W\cap\left([0, \varepsilon)\times\R^{\infty-1}\right) \; &= \; [0, \varepsilon)\times P, \\
W\cap\left((1-\varepsilon, 1]\times\R^{\infty-1}\right) \; &= \; (1-\varepsilon, 1]\times Q.
\end{aligned}$$
Let $\Phi: [0, 2]\times\R^{\infty-1} \longrightarrow [0, 1]\times\R^{\infty-1}$ be the diffeomorphism given by $\Phi(t,\; x) = (t/2, \; x)$, where $t \in [0, 2]$ and $x \in \R^{\infty-1}$.
For morphisms $W: P \rightsquigarrow Q$ and $W': Q \rightsquigarrow R$, composition is defined by 
\begin{equation} \label{equation: composition formula}
W'\circ W \; = \; \Phi(W'\cup W).
\end{equation} 
The category $\Cob_{\theta, d}$ is topologized in the standard way following \cite{GMTW 08}, \cite{GRW 14}. 
\end{defn}

\begin{remark} \label{remark: associativity up to homotopy}
With the composition rule in (\ref{equation: composition formula}), the category is not strictly associative; it is associative up to homotopy. 
This lack of strict associativity won't affect any of our latter constructions. 
\end{remark}

We will need to consider certain subcategories of $\Cob_{\theta, d}$.
\begin{defn} \label{defn: cob cat l connected}
Let $l \in \Z_{\geq -1}$ be an integer. 
The topological subcategory $\Cob^{l}_{\theta, d} \subset \Cob_{\theta, d}$ has the same objects. 
Its morphisms consist of those cobordisms $W: P \rightsquigarrow Q$ for which the pair $(W, P)$ is $l$-connected. 
\end{defn}

Fix an integer $k$. 
We now proceed to make $P \mapsto \mathcal{W}^{k}_{P, \mb{t}}$ into a functor on $\Cob^{k-1}_{\theta, d}$.
Let $\psi: (-\infty, 1] \longrightarrow (-\infty, 0]$ be a diffeomorphism that is the identity on $(-\infty, -1)$
and let
$$\Psi := \psi\times\textstyle{\Id_{\R^{\infty-1}}}: (-\infty, 1]\times\R^{\infty-1} \stackrel{\cong} \longrightarrow (-\infty, 0]\times\R^{\infty-1}.$$
Let $W: P \rightsquigarrow Q$ be a morphism in $\Cob^{k-1}_{\theta, d}$.
For any $\mb{t} \in \mathcal{K}^{k}$, such a morphism determines a map 
\begin{equation} \label{equation: concatenation map}
\mb{S}_{W}: \mathcal{W}^{k}_{P, \mb{t}} \longrightarrow \mathcal{W}^{k}_{Q, \mb{t}}
\end{equation}
by sending 
$(M, (V, \sigma), e) \in \mathcal{W}^{k}_{P, \mb{t}}$ to the element 
$\left(\Psi(M\cup_{P}W), \; (V, \sigma), \; e\right) \in \mathcal{W}^{k}_{Q, \mb{t}}.$
Let us denote $M\cup_{P}W := \Psi(M\cup_{P}W)$.
Notice that since $(W, P)$ is $(k-1)$-connected, the new $\theta$-structure $\hat{\ell}_{M\cup_{P}W}$ on $M\cup_{P}W$
is such that $\ell_{M\cup_{P}W}: M\cup_{P}W \longrightarrow B$ is $k$-connected, and thus determines a well-defined element of the space $\mathcal{W}^{k}_{Q, \mb{t}}$.
From this definition, it follows that for any morphism $(j, \varepsilon): \mb{t} \longrightarrow \mb{t}'$ in $\mathcal{K}^{k}$, the diagram 
$$
\xymatrix{
\mathcal{W}^{k}_{P, \mb{t}'} \ar[rr]^{\mb{S}_{W}} \ar[d]^{(j, \varepsilon)^{*}} && \mathcal{W}^{k}_{Q, \mb{t}'} \ar[d]^{(j, \varepsilon)^{*}} \\
\mathcal{W}^{k}_{P, \mb{t}} \ar[rr]^{\mb{S}_{W}} && \mathcal{W}^{k}_{Q, \mb{t}}
}
$$
commutes.
For morphisms $W: P \rightsquigarrow Q$ in $\Cob^{k-1}_{\theta, d}$ we will study the effect of the induced map
$$
\mb{S}_{W}: \hocolim_{\mb{t}\in\mathcal{K}^{k}}\mathcal{W}^{k}_{P, \mb{t}} \longrightarrow \hocolim_{\mb{t}\in\mathcal{K}^{k}}\mathcal{W}^{k}_{Q, \mb{t}}.
$$
In this way we may view the correspondence $P \mapsto \displaystyle{\hocolim_{\mb{t}\in\mathcal{K}^{k}}}\mathcal{W}^{k}_{P, \mb{t}}$ to be a functor on the category $\Cob^{k-1}_{\theta, d}$.
In order to formulate our main theorem we will need to make some preliminary definitions.
\begin{defn} \label{defn: trivial and primitive surgery}
Let $W: P \rightsquigarrow Q$ be a morphism in $\Cob_{\theta, d}$. 
Let $0 \leq l < d$, and suppose that 
 $W$ is diffeomorphic relative $P$ to the trace of a surgery along an embedding 
$
f: S^{l}\times D^{d-l-1} \longrightarrow P.
$
\begin{enumerate} \itemsep.2cm
\item[(i)] The cobordism $W: P \rightsquigarrow Q$ is said to be \textit{trivial of degree $l+1$}  if the embedding $f$ factors through an embedding $\R^{d-1} \hookrightarrow P$. 
\item[(ii)] The cobordism $W: P \rightsquigarrow Q$ is said to be \textit{primitive of degree $l+1$} if there exists another embedding $g: D^{l}\times S^{d-l-1} \longrightarrow P$ such that $f(S^{l}\times\{0\})$ and $g(\{0\}\times S^{d-l-1})$ intersect transversally in $P$ at exactly one point.  
\end{enumerate}
Notice that the above definition only makes sense in the case that $0 \leq l < d$. 
It will be useful to us latter on to make the following convention.  
A morphism $W: P \rightsquigarrow Q$ in $\Cob_{\theta, d}$ is said to be \textit{primitive of degree $0$} if $W$ is isomorphic (as a $\theta$-manifold) to the disjoint union
$
([0, 1]\times P)\sqcup S^{d},
$
where $([0, 1]\times P)$ is equipped with the $\theta$-structure induced by $\hat{\ell}_{P}$, and $S^{d}$ is equipped with some $\theta$-structure $\hat{\ell}_{S^{d}}$ that admits an extension to $D^{d}$.
Equivalently, we may also refer to the same cobordism $W$ as being \textit{trivial of degree $d$}.
\end{defn}
Our main result of this section is the following theorem:
\begin{theorem} \label{theorem: morphism induce homotopy equivalence}
Let $0 \leq k < d/2$. 
Let $W: P \rightsquigarrow Q$ be a morphism in $\Cob^{k-1}_{\theta, d}$. 
Suppose that $W$ is a composite of elementary cobordisms $V_{m}\circ \cdots \circ V_{1}$, such that each $V_{i}$ satisfies one of the following conditions:
\begin{enumerate} \itemsep.2cm
\item[(a)] $V_{i}$ is primitive of degree $k$, 
\item[(b)] $V_{i}$ is trivial of degree $d-k$, or 
\item[(c)] the pair $(V_{i}, \partial_{\inn}V_{i})$ is $(d-k)$-connected, where $\partial_{\inn}V_{i}$ is the source object of the morphism $V_{i}$. 
\end{enumerate}
Then the induced map 
$\mb{S}_{W}: \displaystyle{\hocolim_{\mb{t}\in\mathcal{K}^{k}}}\mathcal{W}^{k}_{P, \mb{t}} \longrightarrow \displaystyle{\hocolim_{\mb{t}\in\mathcal{K}^{k}}}\mathcal{W}^{k}_{Q, \mb{t}}$
is a weak homotopy equivalence.
\end{theorem}
The proof of the above theorem will span the next three sections. 
Let us now show how Theorem \ref{theorem: morphism induce homotopy equivalence} implies Proposition \ref{proposition: homotopy equivalence of transition maps}, which is the whole point of these new definitions. 
\begin{proof}[Proof of Theorem \ref{proposition: homotopy equivalence of transition maps} assuming Theorem \ref{theorem: morphism induce homotopy equivalence}]
Let $\mb{s} \in \mathcal{K}^{\{k, d-k+1\}}$. 
Let $\mb{s'} = \mb{s}\sqcup\{x\}$, where $x$ has either label $k$ or $d-k+1$.
Let $(j, \varepsilon): \mb{s} \longrightarrow \mb{s}'$ be a morphism with $j: \mb{s} \hookrightarrow \mb{s}'$ the inclusion. 
It will suffice to prove that such a morphism induces a weak homotopy equivalence 
$$
\hocolim_{\mb{t} \in \mathcal{K}^{k}}\widehat{\mathcal{W}}^{k-1, c}_{P, \mb{t}; \mb{s}'} \; \stackrel{\simeq} \longrightarrow \; \hocolim_{\mb{t} \in \mathcal{K}^{k}}\widehat{\mathcal{W}}^{k-1, c}_{P, \mb{t}; \mb{s}}.
$$
We will show that the composite map
\begin{equation} \label{equation: cobordism map composite}
\xymatrix{
\displaystyle{\hocolim_{\mb{t} \in \mathcal{K}^{k}}}\widehat{\mathcal{W}}^{k-1, c}_{P, \mb{t}; \mb{s}'} \ar[rr] && \displaystyle{\hocolim_{\mb{t} \in \mathcal{K}^{k}}}\widehat{\mathcal{W}}^{k-1, c}_{P, \mb{t}; \mb{s}} \ar[d]_{\simeq} \\
\displaystyle{\hocolim_{\mb{t} \in \mathcal{K}^{k}}}\mathcal{W}^{k}_{P\sqcup S_{\mb{s}'}, \; \mb{t}} \ar[u]^{\simeq} && \displaystyle{\hocolim_{\mb{t} \in \mathcal{K}^{k}}}\mathcal{W}^{k}_{P\sqcup S_{\mb{s}}, \; \mb{t}}
}
\end{equation}
agrees with the map induced by $\mb{S}_{W}$, for some cobordism $W: P\sqcup S_{\mb{s}'} \rightsquigarrow P\sqcup S_{\mb{s}}$ covered by Theorem \ref{theorem: morphism induce homotopy equivalence}; the vertical maps in the diagram are the homotopy equivalences (\ref{equation: glue in surgery data map}) and (\ref{equation: homotopy inverse cut out map}) from the proof of Proposition \ref{proposition: fibre identification}. 
Let us describe that cobordism.
Recall that $S_{\mb{s}}$ is the product $\mb{s}\times S^{k-1}\times S^{d-k}$ and $S_{\mb{s}'} = S_{\mb{s}}\sqcup(\{x\}\times S^{k-1}\times S^{d-k})$.
Let $W_{1}$ be the manifold given by 
$$
(S_{\mb{s}'}\sqcup P)\times[0, 1]\bigcup_{\{x\}\times S^{k-1}\times S^{d-k}}\{x\}\times D^{k}\times S^{d-k}.
$$
Similarly, let $W_{2}$ be the manifold given by 
$$
(S_{\mb{s}'}\sqcup P)\times[0, 1]\bigcup_{\{x\}\times S^{k-1}\times S^{d-k}}\{x\}\times S^{k-1}\times D^{d-k+1}.
$$
Both of these define cobordisms from $P_{\mb{s}'}$ to $P_{\mb{s}}$.
Tracing through the definition we see that (\ref{equation: cobordism map composite}) coincides with the map induced by 
$\mb{S}_{W_{1}}$ when the label on $x$ is $k$ and $\varepsilon = +1$, or in the case when the label on $x$ is $d-k+1$ and $\varepsilon = -1$.
Similarly, (\ref{equation: cobordism map composite}) coincides with the map induced by 
$\mb{S}_{W_{2}}$ when the label on $x$ is $d-k+1$ and $\varepsilon = -1$, or when the label on $x$ is $k$ and $\varepsilon = +1$.
Now, we observe that the cobordisms $W_{1}$ and $W_{2}$ are both covered by Theorem \ref{theorem: morphism induce homotopy equivalence}. 
Indeed, $W_{1}$ is constructed by first attaching a primitive $k$-handle and then attaching a $d$-handle; this is covered by conditions (a) and (c). 
Similarly, $W_{2}$ is constructed by attaching a $(d-k+1)$-handle and then a $d$-handle. 
Since $P$ is assumed to be non-empty, it follows from Theorem \ref{theorem: morphism induce homotopy equivalence} that in all cases, the map under consideration is a weak homotopy equivalence. 
This completes the proof of Theorem \ref{proposition: homotopy equivalence of transition maps} assuming Theorem \ref{theorem: morphism induce homotopy equivalence}. 
\end{proof}

\begin{remark} \label{remark: on the k = 0 case}
The argument in the above proof does also include the case where $k = 0$. 
In this case, it follows that $S_{\mb{s}'} = \mb{s}\times S^{k-1}\times S^{d-k} =  \emptyset$, and the cobordism $W_{1}$ is isomorphic (as a $\theta$-manifold) to $([0, 1]\times P)\sqcup S^{d}$.
By Definition \ref{defn: trivial and primitive surgery}, the cobordism $W_{1}$ is  primitive of degree-$0$ (or equivalently trivial of degree-$d$) 
and thus it is covered by Theorem \ref{theorem: morphism induce homotopy equivalence} parts (a) or (b). 
Part (c) of Theorem \ref{theorem: morphism induce homotopy equivalence} is not needed in this case. 
The cobordism $W_{2}$ from the above proof is irrelevant in this $k = 0$ case and need not be considered.
\end{remark}

Below we introduce some notation to keep us organized while proving Theorem \ref{theorem: morphism induce homotopy equivalence}.
\begin{Notation} \label{Notation: bookkeeping device}
We let $\mathcal{V}^{k-1} \subset \Cob^{k-1}_{\theta, d}$ be the subcategory consisting of all morphisms $W: P \rightsquigarrow Q$ with the property that 
$\mb{S}_{W}: \displaystyle{\hocolim_{\mb{t}\in\mathcal{K}^{k}}}\mathcal{W}^{k}_{P, \mb{t}} \longrightarrow \displaystyle{\hocolim_{\mb{t}\in\mathcal{K}^{k}}}\mathcal{W}^{k}_{Q, \mb{t}}$
is a weak homotopy equivalence. 
To prove Theorem \ref{theorem: morphism induce homotopy equivalence}, we will need to show that $\mathcal{V}^{k-1}$ contains all morphisms that satisfy one of the conditions (a), (b), or (c).
\end{Notation}

\begin{remark} \label{remark: concept of the proof}
We remark that our proof of Theorem \ref{theorem: morphism induce homotopy equivalence} is conceptually similar to the homological stability theorem proven by Galatius and Randal-Williams in \cite{GRW 16}.
Of course in our case, Theorem \ref{theorem: morphism induce homotopy equivalence} asserts that the maps under consideration are weak homotopy equivalences rather than homology equivalences. 
In view of what is done in Section \ref{section: stable moduli spaces} (Theorem \ref{theorem: stable stability} in particular), 
one can actually consider the homological stability theorem of \cite{GRW 16} as the analogue of Theorem \ref{theorem: morphism induce homotopy equivalence} for the case that $d = 2n$ and $k = n$.
\end{remark}

\section{The morphism $H_{k, d-k}(P)$}
In this section we begin the process of proving Theorem \ref{theorem: morphism induce homotopy equivalence} by verifying that 
$$\mb{S}_{W}: \hocolim_{\mb{t}\in\mathcal{K}^{k}}\mathcal{W}^{k}_{P, \mb{t}} \longrightarrow \hocolim_{\mb{t}\in\mathcal{K}^{k}}\mathcal{W}^{k}_{Q, \mb{t}}$$
is a weak homotopy equivalence for certain basic cobordisms $W: P \rightsquigarrow Q$.
In particular we will prove the special case of Theorem \ref{theorem: morphism induce homotopy equivalence} for when $k = 0$, see Remark \ref{remark: primitive surgery of degree 0}.
\subsection{The morphism $H_{k, d-k}(P)$} \label{subsection: morphism H(P)}
Fix an object $P \in \Ob\Cob^{k-1}_{\theta, d}$. 
We construct a morphism
\begin{equation} \label{equation: k-torus}
H_{k, d-k}(P): P \rightsquigarrow P
\end{equation}
of $\Cob^{k-1}_{\theta, d}$ as follows.
Choose an embedding $j: S^{k}\times S^{d-k} \hookrightarrow (0, 1)\times\R^{\infty-1}$ with image disjoint from $[0, 1]\times P$. 
Fix a $\theta$-structure $\hat{\ell}_{k, d-k}$ on $S^{k}\times S^{d-k}$ with the property that $\ell_{k, d-k}: S^{k}\times S^{d-k} \longrightarrow B$ is null-homotopic. 
We then let $H_{k, d-k}(P) \subset [0, 1]\times\R^{\infty-1}$ be the cobordism obtained by forming the connected sum of $[0, 1]\times P$ with $j(S^{k}\times S^{d-k})$ along some embedded arc connecting the two submanifolds.
We equip $H_{k, d-k}(P)$ with a $\theta$-structure that agrees with $\hat{\ell}_{k, d-k}$ on the connect-summand of $S^{k}\times S^{d-k}$, and with $\hat{\ell}_{[0, 1]\times P}$ away from the connect-summand.
This self-cobordism $H_{k, d-k}(P): P \rightsquigarrow P$ equipped with its $\theta$-structure yields a morphism in $\Cob^{k-1}_{\theta, d}$. 
We have the following proposition.
\begin{proposition} \label{proposition: H-k induces iso}
Let $0 \leq k < d/2$. 
Then for any non-empty object $P \in \Ob\Cob^{k-1}_{\theta, d}$, the morphism 
$H_{k, d-k}(P): P \rightsquigarrow P$ 
is contained in $\mathcal{V}^{k-1}$ (see Notational Convention \ref{Notation: bookkeeping device}).
\end{proposition}

\begin{remark} \label{remark: primitive surgery of degree 0}
Notice that in the case that $k = 0$, for any $P \in \Ob\Cob^{k-1}_{\theta, d}$ the cobordism $H_{0, d}(P)$ is isomorphic as a $\theta$-manifold to $([0, 1]\times P)\sqcup S^{d}$, where the sphere $S^{d}$ is equipped with a $\theta$-structure that extends to the disk.
By Definition \ref{defn: trivial and primitive surgery}, the cobordism $H_{0, d}(P)$ is then \textit{primitive of degree $0$} (or equivalently trivial of degree $d$). 
In this case, Proposition \ref{proposition: H-k induces iso} then implies Theorem \ref{theorem: morphism induce homotopy equivalence} (parts (a) and (b)), and thus implies Theorem \ref{proposition: homotopy equivalence of transition maps} (recall from Remark \ref{remark: on the k = 0 case} that part (c) of Theorem \ref{theorem: morphism induce homotopy equivalence} is not needed in this $k = 0$ case). 
We remark that what we do in the $k = 0$ case is similar to what was done in \cite[Section 6.3 ``Annihilation of $d$-Spheres'']{MW 07}.
\end{remark}

We will show that 
$
\mb{S}_{H_{k, d-k}(P)}: \displaystyle{\hocolim_{\mb{t}\in\mathcal{K}^{k}}}\mathcal{W}^{k}_{P, \mb{t}} \longrightarrow \displaystyle{\hocolim_{\mb{t}\in\mathcal{K}^{k}}}\mathcal{W}^{k}_{Q, \mb{t}}
$
is homotopic to the identity map.
The proof will be easiest if we reformulate the statement to one about the classifying space of the \textit{transport category,} $\mathcal{K}^{k}\wr\mathcal{W}^{k}_{P, (\--)}$. 
Recall that an object of $\mathcal{K}^{k}\wr\mathcal{W}^{k}_{P, (\text{--})}$ is a pair $(\mb{t}, x)$ where $\mb{t} \in \mathcal{K}^{k}$ and $x \in \mathcal{W}^{k}_{P, \mb{t}}$.
A morphism in $\mathcal{K}^{k}\wr\mathcal{W}^{k}_{P, (\text{--})}$ from $(\mb{t}, x)$ to $(\mb{s}, y)$ is defined to be a morphism $(j, \varepsilon): \mb{t} \longrightarrow \mb{s}$ in $\mathcal{K}^{k}$ such that the induced map 
$
(j, \varepsilon)^{*}: \mathcal{W}^{k}_{P, \mb{s}} \longrightarrow \mathcal{W}^{k}_{P, \mb{t}} 
$
sends $y$ to $x$.
The homotopy colimit, $\displaystyle{\hocolim_{\mb{t}\in\mathcal{K}^{k}}}\mathcal{W}^{k}_{P, \mb{t}}$, is then by definition the classifying space of this category, i.e.
$$
\hocolim_{\mb{t}\in\mathcal{K}^{k}}\mathcal{W}^{k}_{\mb{t}} = B\left(\mathcal{K}^{k}\wr\mathcal{W}^{k}_{P, (\text{--})}\right).
$$
The map $\mb{S}_{H_{k, d-k}(P)}$ then induces a self-functor of $\mathcal{K}^{k}\wr\mathcal{W}^{k}_{P, (\text{--})}$ which we denote by
\begin{equation} \label{equation: self-functor T}
\mb{T}^{k}(P): \mathcal{K}^{k}\wr\mathcal{W}^{k}_{P, (\text{--})} \longrightarrow \mathcal{K}^{k}\wr\mathcal{W}^{k}_{P, (\text{--})}.
\end{equation}
To prove Proposition \ref{proposition: H-k induces iso} it will suffice to show that this functor induces a weak equivalence on classifying spaces. 
To do this we will construct a zig-zag of natural transformations 
$$\xymatrix{
\mb{T}^{k}(P) & s_{k}  \ar@{=>}[l]  \ar@{=>}[r] & \cyl  
}
$$
of self-functors of $\mathcal{K}^{k}\wr\mathcal{W}^{k}_{P, (\text{--})}$, where $\cyl: \mathcal{K}^{k}\wr\mathcal{W}^{k}_{P, (\text{--})} \longrightarrow \mathcal{K}^{k}\wr\mathcal{W}^{k}_{P, (\text{--})}$ is a homotopy equivalence of topological categories. 
Since natural transformations of functors induce homotopies between their induced maps on classifying spaces, it will follow that $B|\mb{T}^{k}(P)|$ is homotopic to $B|\cyl|$ and thus is a weak homotopy equivalence. 
We proceed to define the functors $\cyl$ and $s_{k}$.
\begin{defn} \label{defn: natural transformations}
The functor $\cyl: \mathcal{K}^{k}\wr\mathcal{W}^{k}_{P, (\text{--})} \longrightarrow \mathcal{K}^{k}\wr\mathcal{W}^{k}_{P, (\text{--})}$ is defined as follows. 
Let 
$$(\mb{t}, (M, (V, \sigma), e)) \in \Ob(\mathcal{K}^{k}\wr\mathcal{W}^{k}_{P, (\text{--})}).$$
Let $M' \subset (-\infty, 1]\times\R^{\infty-1}$ be the submanifold defined by the pushout,
$
M' = M\cup_{P}([0, 1]\times P).
$
By keeping $V$ and $e$ the same, the tuple $( \mb{t}, (M', (V, \sigma), e))$ determines a well defined element of $\Ob(\mathcal{K}^{k}\wr\mathcal{W}^{k}_{P, (\text{--})})$.
The functor $\cyl$ is defined by 
\begin{equation} \label{equation: definition of cal}
\cyl\left(\mb{t}, (M, (V, \sigma), e)\right) \; = \; \left(\mb{t}, (M', (V, \sigma), e)\right),
\end{equation}
with $M'$ defined as above.
\end{defn}

Next, we define the natural transformation 
$s_{k}: \mathcal{K}^{k}\wr\mathcal{W}^{k}_{P, (\text{--})}   \longrightarrow \mathcal{K}^{k}\wr\mathcal{W}^{k}_{P, (\text{--})}.$
To do this we need to specify some preliminary data. 
Fix once and for all a symmetric bilinear form 
$$\sigma^{d-k}_{\R^{d+1}}: \R^{d+1}\otimes\R^{d+1} \longrightarrow \R$$ 
of index $d-k$. 
Fix a $\theta$-structure $\hat{\ell}_{\R^{d+1}}$ on $\R^{d+1}$. 
For $\R^{d+1}$ equipped with $\hat{\ell}_{\R^{d+1}}$, the pair 
$(\R^{d+1}, \sigma^{d-k}_{\R^{d+1}})$ determines an element of the space $G^{\mf}_{\theta}(\R^{\infty})_{\loc}$. 

\begin{defn} \label{defn: functor s-k}
The functor $s_{k}: \mathcal{K}^{k}\wr\mathcal{W}^{k}_{P, (\text{--})}   \longrightarrow \mathcal{K}^{k}\wr\mathcal{W}^{k}_{P, (\text{--})} $ is defined as follows. 
Let
$(\mb{t}, (M, (V, \sigma), e)) \in \Ob(\mathcal{K}^{k}\wr\mathcal{W}^{k}_{P, (\text{--})}).$ 
Fix an embedding  
\begin{equation} \label{equation: cylindrical surgery data}
\phi: D^{d-k}\times D^{k+1} \longrightarrow (0, 1)\times\R^{\infty-1},
\end{equation}
that satisfies the following conditions:
\begin{enumerate} \itemsep.2cm
\item[(i)] $\phi^{-1}([0, 1]\times P) = S^{d-k-1}\times D^{k+1}$; 
\item[(ii)] the restriction of $\hat{\ell}_{[0, 1]\times P}$ to $S^{d-k-1}\times D^{k+1}$ agrees with the $\theta$-structure $\hat{\ell}_{\R^{d+1}}|_{S^{d-k-1}\times D^{k+1}}$. 
\item[(iii)] the surgered manifold, 
$$\left[([0, 1]\times P)\setminus\phi(S^{d-k-1}\times D^{k+1})\right]\bigcup\phi(D^{d-k}\times S^{k}),$$
agrees with  
$$H_{k, d-k}(P) \subset [0, 1]\times\R^{\infty-1}$$
as a $\theta$-manifold. 
\end{enumerate}
Since $H_{k, d-k}(P)$ is obtained by forming the connected sum of $S^{k}\times S^{d-k}$ with $[0, 1]\times P$, in order to satisfy condition (ii), it will suffice to choose $\phi$ so that the embedding
$$\phi|_{S^{d-k-1}\times D^{k+1}}: S^{d-k-1}\times D^{k+1} \longrightarrow [0, 1]\times P$$ 
factors through some disk $D^{d} \hookrightarrow [0, 1]\times P$.
Indeed, performing surgery on an embedded $(d-k-1)$-sphere that bounds a disk has the effect of creating a connect-summand of $S^{d-k}\times S^{k}$.

We proceed to define 
$s_{k}(\mb{t}, (M, (V, \sigma), e)):$
\begin{enumerate} \itemsep.3cm
\item[(i)]
Let $M' = M\bigcup_{P}([0,1]\times P)$ be as in Definition \ref{defn: natural transformations}. 
\item[(ii)]
Let $\mb{t}' \in \mathcal{K}^{k}$ be the object obtained by adjoining one more point, $x$, to $\mb{t}$, with label $d-k$. 
\item[(iii)] 
The element $(V', \sigma') \in \Maps(\mb{t}', \; G^{\mf}_{\theta}(\R^{\infty})_{\loc})$ is defined by setting 
$$(V'(i), \sigma'(i)) = (V(i), \sigma(i))$$ 
for $i \in \mb{t}$, and $(V'(x), \sigma'(x)) = (\R^{d+1}, \sigma^{d-k}_{\R^{d+1}})$.
\item[(iv)]
The embedding
$e': D((V')^{-})\times_{\mb{t}'}D((V')^{+}) \longrightarrow (-1, 0]\times(-1, 1)^{\infty-1}$
is set equal to $e$ over $\mb{t}$, and to agree with $\phi$ from (\ref{equation: cylindrical surgery data}) 
on the component corresponding to the element $x \in \mb{t}'$. 
\end{enumerate}
We define $s_{k}(\mb{t}, (M, (V, \sigma), e))$ by setting 
\begin{equation}
s_{k}(\mb{t}, (M, V, e)) \; := \; (\mb{t}', (M', (V', \sigma'), e')).
\end{equation}
Since any morphism $\mb{s} \longrightarrow \mb{t}$ in $\mathcal{K}^{k}$ induces a unique morphism $\mb{s}\sqcup\{x\} \longrightarrow \mb{t}\sqcup\{x\}$, it follows that the operation $s_{k}$ is functorial. 
\end{defn} 
We now have all of the ingredients ready to finish the proof of Proposition \ref{proposition: H-k induces iso}. 
\begin{proof}[Proof of Proposition \ref{proposition: H-k induces iso}]
The proof will be complete once we construct the natural transformations 
\begin{equation} \label{equation: natural transformations}
\xymatrix{
\mb{T}^{k}(P) & s_{k}  \ar@{=>}[l]  \ar@{=>}[r] & \cyl.  
}
\end{equation}
First, let $b_{d-k}: \mathcal{K}^{k} \longrightarrow \mathcal{K}^{k}$ be the functor defined by the formula $b_{d-k}(\mb{t}) = \mb{t}\sqcup\{x\}$, where $x$ is a point with label $d-k$.
For $\mb{t} \in \mathcal{K}^{k}$, let 
$$(i_{\mb{t}}, \varepsilon^{x}_{\pm}): \mb{t} \longrightarrow b_{d-k}(\mb{t})$$
be the morphisms (one for each case $+1$ or $-1$)
 defined by letting $i_{\mb{t}}$ be the inclusion 
 $$\mb{t} \hookrightarrow s_{k}(\mb{t}) = \mb{t}\sqcup\{x\},$$ 
 and $\varepsilon^{x}_{\pm}(x) = \pm1$.
 The assignments 
 $$\mb{t} \mapsto (i_{\mb{t}}, \varepsilon^{x}_{+}) \quad \text{and} \quad \mb{t} \mapsto (i_{\mb{t}}, \varepsilon^{x}_{-})$$ 
 yield two different natural transformations between the identity functor and the functor $b_{d-k}$.
 We will use these natural transformations to construct the natural transformations (\ref{equation: natural transformations}).
Let 
$$(\mb{t}, (M, (V, \sigma), e)) \in \mathcal{K}^{k}\wr\mathcal{W}^{k}_{P, (\text{--})}.$$ 
The morphism $(i_{\mb{t}}, \varepsilon^{x}_{+}): \mb{t} \longrightarrow b_{d-k}(\mb{t})$ induces a morphism 
$$(i_{T}, \varepsilon^{x}_{+})^{*}: s_{k}(b_{d-k}(\mb{t}), (M, (V, \sigma), e)) \longrightarrow \cyl(\mb{t}, (M,(V, \sigma), e))$$
in $\mathcal{K}^{k}\wr\mathcal{W}^{k}_{P, (\text{--})}$.
This has the effect of forgetting the extra surgery data $\phi$ defined in (\ref{equation: cylindrical surgery data}) over the point $x \in b_{d-k}(\mb{t})$. 
This assignment 
$(\mb{t}, (M, (V, \sigma), e)) \; \mapsto \; (i_{\mb{t}}, \varepsilon^{x}_{+})^{*}$
did not depend on the data $(M, (V, \sigma), e)$, and thus it defines a natural transformation between $s_{k}$ and $\cyl$. 

For the other natural transformation $s_{k} \implies \mb{T}^{k}(P)$, we observe that $(i_{\mb{t}}, \varepsilon^{x}_{-}): \mb{t} \longrightarrow b_{d-k}(\mb{t})$ induces a morphism 
$$
(i_{\mb{t}}, \varepsilon^{x}_{-})^{*}: s_{k}(b_{d-k}(\mb{t}), (M, (V, \sigma), e)) \longrightarrow \mb{T}^{k}(P)(\mb{t}, (M, (V, \sigma), e)).
$$
Indeed, $(i_{\mb{t}}, \varepsilon^{x}_{-})^{*}$ is defined by performing the fibrewise surgery associated to the point $x = b_{d-k}(\mb{t})\setminus \mb{t}$. 
Since this surgery is given by the embedding $\phi$ from (\ref{equation: cylindrical surgery data}), it follows that this surgery creates a connected sum factor of $S^{k}\times S^{d-k}$, which is precisely what the functor $\mb{T}^{k}(P)$ does. 
This assignment 
$$
(\mb{t}, (M, (V, \sigma), e)) \mapsto (i_{\mb{t}}, \varepsilon^{x}_{-})^{*}
$$
yields the desired natural transformation between $s_{k}$ and $\mb{T}^{k}(P)$.
It follows that the map on classifying spaces induced by $\mb{T}^{k}(P)$ is homotopic to the map on classifying spaces induced by $\cyl$. 
Since $\cyl$ induces a weak homotopy equivalence on classifying spaces (which is actually homotopic to the identity map) it follows that $\mb{T}^{k}(P)$ induces a weak homotopy equivalence as well. 
This completes the proof of the proposition.
\end{proof}

\section{Surgeries and Traces} \label{section: surgeries and traces}
\subsection{Surgeries and Traces} \label{subsection: surgeries and traces}
We now use Proposition \ref{proposition: H-k induces iso} to show that many other cobordisms beyond $H_{k, d-k}(P)$ are contained in $\mathcal{V}^{k-1}$.
We will need to work with morphisms in the category $\Cob^{k-1}_{\theta, d}$ that arise as the traces of surgeries on objects in $\Cob^{k-1}_{\theta, d}$.
In order to do this we introduce a series of new general constructions.

\begin{Construction}[Standard Trace] \label{Construction: trace of surgery}
Fix a non-negative integer $l < d$.
We fix once and for all a submanifold 
$$
T_{l} \; \subset \; [0, 1]\times D^{l}\times D^{d-l} \; \subset \; \R^{l}\times\R^{d-l}
$$
with the following properties:
\begin{enumerate} \itemsep.2cm
\item[(a)] The height function, $T_{l} \hookrightarrow [0, 1]\times D^{l}\times D^{d-l} \stackrel{\text{proj}} \longrightarrow [0, 1]$, is Morse with one critical point in the interior with index $l$. 
\item[(b)] The intersections with $\{0\}\times S^{l-1}\times D^{d-l}$, $\{1\}\times D^{l}\times S^{d-l-1}$, and $[0,1]\times S^{l-1}\times S^{d-l-1}$ are given by:
$$
\begin{aligned}
T_{l}\cap(\{0\}\times S^{l-1}\times D^{d-l}) \; &= \; \{0\}\times S^{l-1}\times D^{d-l}, \\
T_{l}\cap(\{1\}\times D^{l}\times S^{d-l-1}) \; &= \; \{1\}\times D^{l}\times S^{d-l-1}, \\
T_{l}\cap([0,1]\times S^{l-1}\times S^{d-l-1}) \; &= \; [0,1]\times S^{l-1}\times S^{d-l-1}.
\end{aligned}
$$
\item[(c)] 
The boundary $\partial T_{l}$ is given by the union 
$$
\partial T_{l} \; = \; (\{0\}\times S^{l-1}\times D^{d-l})\cup(\{1\}\times D^{l}\times S^{d-l-1})\cup([0,1]\times S^{l-1}\times S^{d-l-1}).
$$
\end{enumerate}
We refer to $T_{l}$ as the \textit{trace} of surgery along $S^{l-1}\times D^{d-l} \subset S^{l-1}\times\R^{d-l}$.
The Morse trajectories associated to the height function $T_{l} \longrightarrow [0, 1]$ determine embeddings:
$$\begin{aligned}
\beta_{0}: (D^{l}\times D^{d-l}, \; S^{l-1}\times D^{d-l}) &\longrightarrow (T_{l}, \; \{0\}\times S^{l-1}\times D^{d-l}), \\
\beta_{1}: (D^{l}\times D^{d-l}, \; D^{l}\times S^{d-l-1}) &\longrightarrow (T_{l}, \; \{1\}\times D^{l}\times S^{d-l-1}).
\end{aligned}$$
We call these embeddings the \textit{standard core handles} associated to $T_{l}$. 
We note any choice of $\theta$-structure  
$
\phi: T(D^{l}\times D^{d-l})\oplus\epsilon^{1} \longrightarrow \theta^{*}\gamma^{d+1}
$
that extends $\hat{\ell}$ determines a $\theta$-structure on $T_{l}$ uniquely up to homotopy. 
\end{Construction}
We use the above trace construction to define morphisms in the category $\Cob^{k-1}_{\theta, d}$.
\begin{defn}[Trace of a surgery]  \label{defn: trace of surgery}
Let $P \in \Ob\Cob^{k-1}_{\theta, d}$. 
Let $\sigma: \R^{l}\times \R^{d-l} \longrightarrow \R^{\infty-1}$ be an embedding that satisfies,
$\sigma^{-1}(P) = S^{l-1}\times\R^{d-l}.$
Let 
$\phi: T(\R^{l}\times\R^{d-l}) \longrightarrow \theta^{*}\gamma^{d}$ 
be a $\theta$-structure that extends $\hat{\ell}_{P}\circ D\sigma: T(S^{l-1}\times D^{d-l})\oplus\epsilon^{1}  \longrightarrow  \theta^{*}\gamma^{d}$.
The pair $(\sigma, \phi)$ will be called \textit{surgery data of degree $l-1$}. 
We denote by $\Surg_{l-1}(P)$ the set of all such surgery data $(\sigma, \phi)$ of degree $l-1$.

Using $\sigma := (\sigma, \phi) \in \Surg_{l-1}(P)$ we define a morphism in $\Cob^{k-1}_{\theta, d}$,
\begin{equation} \label{equation: trace of generic surgery}
T(\sigma): P \rightsquigarrow P(\sigma),
\end{equation}
by setting,
$$
T(\sigma) \; = \; \left[(I\times P)\setminus(I\times\sigma(S^{l-1}\times D^{d-l}))\right]\bigcup(\textstyle{\Id_{I}}\times\sigma)(T_{l}). 
$$
The manifold $P(\sigma)$ is defined to be $T(\sigma)\cap(\{1\}\times\R^{\infty-1})$.
The $\theta$-structure $\hat{\ell}_{T(\sigma)}$ on $T(\sigma)$ is the one induced by the extended $\theta$-structure $\phi$. 
\end{defn}

\begin{defn}[Dual surgery] \label{defn: dual surgery}
Let $\sigma \in \Surg_{l-1}(P)$ and consider the trace $T(\sigma): P \rightsquigarrow P(\sigma)$.
The embedding 
$\beta_{1}: (D^{l}\times D^{d-l}, \; D^{l}\times S^{d-l-1}) \longrightarrow (T_{l}, \; \{1\}\times D^{l}\times S^{d-l-1})$
from Definition \ref{Construction: trace of surgery} determines an embedding 
$\beta_{1}(\sigma): (D^{l}\times D^{d-l}, \;D^{l}\times S^{d-l-1}) \longrightarrow (T(\sigma), P(\sigma)).$
Let 
$$
\bar{\sigma}: \R^{l}\times\R^{d-l} \longrightarrow \R^{\infty-1}
$$
be an embedding with $\bar{\sigma}^{-1}(P(\sigma)) = \R^{l}\times S^{d-l-1}$, that extends $\beta_{1}(\sigma)$. 
Let 
$$\bar{\phi}: T(\R^{l}\times\R^{d-l})\oplus\epsilon^{1} \longrightarrow \theta^{*}\gamma^{d+1}$$ 
be a $\theta$-structure obtained by restricting $\hat{\ell}_{T(\sigma)}$. 
The pair 
$
\bar{\sigma} = (\bar{\sigma}, \bar{\phi})
$
defines surgery data in $P(\sigma)$ of degree $d-l-1$.
We call $\bar{\sigma}$ the \textit{surgery dual to $\sigma$}.
\end{defn}

\begin{remark} \label{equation: choices involved in dual surgery}
In the above definition, the construction of the dual surgery $\bar{\sigma} = (\bar{\sigma}, \bar{\phi})$ depended on choices.
However, the $\theta$-diffeomorphism class of the trace $T(\bar{\sigma})$ does not depend on any of these choices made. 
In particular, the $\theta$-diffeomorphism class of the resulting surgered manifold $P(\sigma)(\bar{\sigma})$ does not depend on the choices, and it follows that 
$P(\sigma)(\bar{\sigma}) \; \cong \; P$
for all such choices of $\bar{\sigma}$. 
For this reason, this indeterminacy in the definition of the dual surgery $\bar{\sigma}$ wont affect any of our results.
\end{remark}

We will also have to form the traces of multiple disjoint surgeries simultaneously. 
\begin{defn}[Simultaneous surgeries] \label{defn: simultaneous trace of multiple surgeries}
Let $(\sigma_{1}, \phi_{1}), \dots, (\sigma_{p}, \phi_{p})$ be a collection of surgery data for $P \in \Ob\Cob^{k-1}_{\theta , d}$ that satisfies,
$\sigma_{i}(\R^{l_{i}}\times\R^{d-l_{i}})\cap\sigma_{j}(\R^{l_{j}}\times\R^{d-l_{j}}) \; = \; \emptyset,$ for all $i \neq j$.
We may form their \textit{simultaneous trace}
\begin{equation} \label{equation: simultaneous trace}
T(\sigma_{1}, \dots, \sigma_{p}): P \rightsquigarrow P(\sigma_{1}, \dots, \sigma_{p})
\end{equation}
by setting $T(\sigma_{1}, \dots, \sigma_{p})$ equal to,
$$
\left[(I\times P)\setminus\cup_{i=1}^{p}(I\times\sigma_{i}(S^{l-1}\times D^{d-l}))\right]\bigcup\cup_{i=1}^{p}(\textstyle{\Id_{I}}\times\sigma_{i})(T_{l}).
$$
The associated $\theta$-structure $\ell_{T(\sigma_{1}, \dots, \sigma_{p})}$ on $T(\sigma_{1}, \dots, \sigma_{p})$ is defined in the same way as above using $\phi_{1}, \dots, \phi_{p}$.
\end{defn}

\begin{Construction}[Interchange of Disjoint Surgeries] \label{construction: interchange of surgeries}
Let $P \in \Ob\Cob^{k-1}_{\theta, d}$. 
Let $\sigma_{1}$ and $\sigma_{2}$ be disjoint surgery data for $P$ of degrees $l_{1}$ and $l_{2}$ respectively (by disjoint we mean that $\Image(\sigma_{1})\cap\Image(\sigma_{2}) = \emptyset$).
Consider the traces:
$$\begin{aligned}
T(\sigma_{1}):& P \rightsquigarrow P(\sigma_{1}), \\
T(\sigma_{2}):& P \rightsquigarrow P(\sigma_{2}).
\end{aligned}$$
Notice that the cobordism $T(\sigma_{2})$ agrees with the cylinder $[0, 1]\times P \subset [0, 1]\times\R^{\infty-1}$ outside of the submanifold
$
(\textstyle{\Id_{I}}\times\sigma_{2})(T_{l_{1}}) \; \subset \; T(\sigma_{2}).
$
Furthermore, 
$$
P(\sigma_{2})\setminus\Image(\sigma_{2}) \; = \; P\setminus\Image(\sigma_{2}).
$$
Since the images of $\sigma_{1}$ and $\sigma_{2}$ are disjoint, it follows that 
$$\{1\}\times\sigma_{1}(S^{l_{1}}\times D^{d-l_{1}}) \subset P(\sigma_{2})\setminus\Image(\sigma_{2}).$$
Thus, $\sigma_{1}$ determines surgery data of degree $l_{1}$ in $P(\sigma_{2})$.
We denote this new surgery data by $\mathcal{R}_{\sigma_{2}}(\sigma_{1})$. 
Its trace yields a cobordism 
\begin{equation} \label{equation: interchange of trace}
T(\mathcal{R}_{\sigma_{2}}(\sigma_{1})): P(\sigma_{2}) \rightsquigarrow P(\sigma_{1}, \sigma_{2}).
\end{equation}
The manifold on the right is precisely the result of the target of the simultaneous trace described in Definition \ref{defn: trace of surgery}. 
It follows directly from the construction that 
\begin{equation} \label{equation: simultaneous + composed surgery}
T(\mathcal{R}_{\sigma_{2}}(\sigma_{1}))\circ T(\sigma_{2}) \; = \; T(\sigma_{1}, \sigma_{2})
\end{equation}
as cobordisms from $P$ to $P(\sigma_{1}, \sigma_{2})$.
\end{Construction}

\begin{defn} \label{defn: linked surgeries}
Let $P \in \Ob\Cob^{k-1}_{\theta, d}$ and let $\sigma \in \Surg_{l-1}(P)$. 
We denote by $\mathcal{L}(P, \sigma)$ the set of all surgeries $\alpha \in \Surg_{d-l-1}(P)$ that satisfy:
\begin{enumerate} \itemsep.2cm
\item[(i)] $\alpha(D^{d-l}\times D^{l})\cap\sigma(D^{l}\times D^{d-l}) = \emptyset$;
\item[(ii)] the embedding $\alpha|_{S^{d-l-1}\times\{0\}}: S^{d-l-1}\times\{0\} \longrightarrow P\setminus\sigma(S^{l-1}\times\{0\})$ is isotopic to the embedding 
$$
\sigma|_{\{x\}\times S^{d-l-1}}: \{x\}\times S^{d-l-1} \longrightarrow P\setminus\sigma(S^{l-1}\times\{0\})
$$
for some $x \in S^{l-1}$. 
\end{enumerate}
\end{defn}

For any surgery datum $\sigma$, the set $\mathcal{L}(P, \sigma)$ is always non-empty. 
To find an element of $\mathcal{L}(P, \sigma)$ one simply takes the thickening of a small displacement of $\sigma(\{x\}\times S^{d-l-1})$. 
The following proposition is proven directly by unpacking Construction \ref{construction: interchange of surgeries}. 
\begin{proposition} \label{proposition: interchange of linked surgeries}
For $P \in \Ob\Cob^{k-1}_{\theta, d}$, let $\sigma_{1}, \sigma_{2}$, and $\sigma_{3}$ be mutually disjoint surgery data. 
Suppose that $\sigma_{3} \in \mathcal{L}(P, \sigma_{2})$. 
Then $\mathcal{R}_{\sigma_{1}}(\sigma_{3}) \in \mathcal{L}(P(\sigma_{1}), \mathcal{R}_{\sigma_{1}}(\sigma_{2}))$.
\end{proposition}

\subsection{First results} \label{subsection: first results}
We now apply the constructions from the previous subsection. 
We need one more definition (which is just a rewording of Definition \ref{defn: trivial and primitive surgery}, (part (i)).
\begin{defn} \label{defn: trivial surgery data}
Let $P \in \Ob\Cob^{k-1}_{\theta, d}$.
Surgery data $\sigma \in \Surg_{l}(P)$ (of any degree $l$) is said to be \textit{trivial} if the embedding 
$\sigma|_{S^{l}\times D^{d-l-1}}: S^{l}\times D^{d-l-1} \longrightarrow P$ 
factors through some embedding $\R^{d-1} \hookrightarrow P$. 
We denote the set of all trivial surgery data of degree $l$ by $\Surg^{\tr}_{l}(P) \subset \Surg_{l}(P)$.
\end{defn}

\begin{proposition} \label{proposition: tori composite}
Let $P \in \Ob\Cob^{k-1}_{\theta, d}$.
Let $\sigma \in \Surg^{\tr}_{i}(P)$, with $i = k-1$ or $d-k-1$. 
Let $\alpha \in \mathcal{L}(P, \sigma)$. 
Then the cobordism 
$
T(\sigma, \alpha): P \rightsquigarrow P(\sigma, \alpha)
$
lies in $\mathcal{V}^{k-1}$.
\end{proposition}
\begin{proof}
We first prove the proposition in the special case that $P \cong S^{d-1}$. 
With this assumption, since $\alpha \in \mathcal{L}(P, \sigma)$ it follows directly from Definition \ref{defn: linked surgeries} that the embedded spheres $\sigma(S^{i}), \alpha(S^{d-i-2}) \; \subset \; P$ have linking number equal to $1$ (this linking number is well defined because $P \cong S^{d-1}$). 
It follows that the manifold obtained by attaching handles to $P$ along $\sigma(S^{i}\times D^{d-i-1})$ and $\alpha(S^{d-i-2}\times D^{i+1})$ is diffeomorphic to $([0, 1]\times P)\#(S^{k}\times S^{d-k})$.
From this it follows that $T(\sigma, \alpha)$ is equivalent to the cobordism $H_{k, d-k}(P)$, and thus by
Proposition \ref{proposition: H-k induces iso} $T(\sigma, \alpha)$ is contained in $\mathcal{V}^{k-1}$.
This proves the proposition in the special case that $P \cong S^{d-1}$. 
For the general case, we observe that since $\sigma$ is a trivial surgery by assumption, it follows that both $\sigma(S^{i}\times D^{d-i-1})$ and $\alpha(S^{d-i-2}\times D^{i+1})$ are contained in a disk in $P$. 
Using this fact, the general case reduces to the case where $P \cong S^{d-1}$ which was proven above. 
\end{proof}

The next theorem provides a partial upgrade of the previous proposition. 
Its proof requires significantly more work and thus is postponed until the next section.
\begin{theorem} \label{theorem: arbitrary handles}
Let $P \in \Ob\Cob^{k-1}_{\theta, d}$. 
Let $\sigma \in \Surg_{k-1}(P)$ be an \textbf{arbitrary} surgery and let $\alpha \in \mathcal{L}(P, \sigma)$. 
Then $T(\sigma, \alpha)$ lies in $\mathcal{V}^{k-1}$. 
\end{theorem}

Before embarking on the proof of the above theorem we present a corollary.
\begin{corollary} \label{corollary: trivial d-l handle}
Let $P \in \Ob\Cob^{k-1}_{\theta, d}$. 
Let $T(\sigma): P \rightsquigarrow P(\sigma)$ be the trace of a trivial surgery in degree $d-k-1$. 
Then $T(\sigma)$ lies in $\mathcal{V}^{k-1}$.
\end{corollary}
\begin{proof}
Denote $\sigma_{0} := \sigma$.
Let $\sigma_{1} \in \mathcal{L}(P, \sigma_{0})$, and then choose $\sigma_{2} \in \mathcal{L}(P, \sigma_{1})$ so that 
$$\Image(\sigma_{2})\cap\left[\Image(\sigma_{0})\cup\Image(\sigma_{1})\right] = \emptyset.$$
We consider the composite
$$\xymatrix{
P \ar[rr]^{T(\sigma_{0})} && P(\sigma_{0}) \ar[rrr]^{T(\mathcal{R}_{\sigma_{0}}(\sigma_{1}))} &&& 
P(\sigma_{0}, \sigma_{1}) \ar[rrr]^{T(\mathcal{R}_{\sigma_{0}, \sigma_{1}}(\sigma_{2}))} &&& P(\sigma_{0}, \sigma_{1}, \sigma_{2}).
}
$$
The first composite $T(\mathcal{R}_{\sigma_{0}}(\sigma_{1}))\circ T(\sigma_{0})$ is equal to $T(\sigma_{0}, \sigma_{1})$ and thus is in $\mathcal{V}^{k-1}$ by Proposition \ref{proposition: tori composite} because $\sigma_{0}$ is a trivial surgery of degree $d-k-1$.
The second composite, 
$T(\mathcal{R}_{\sigma_{0}, \sigma_{1}}(\sigma_{2}))\circ T(\mathcal{R}_{\sigma_{0}}(\sigma_{1}))$, 
is equal to 
$T\left(\mathcal{R}_{\sigma_{0}}(\sigma_{1}), \mathcal{R}_{\sigma_{0}}(\sigma_{2})\right).$
Now, $\mathcal{R}_{\sigma_{0}}(\sigma_{1})$ is a surgery of degree $k-1$, and 
$\mathcal{R}_{\sigma_{1}}(\sigma_{2})$ is contained in $\mathcal{L}(P(\sigma_{0}), \mathcal{R}_{\sigma_{0}}(\sigma_{1})).$
It follows from Theorem \ref{theorem: arbitrary handles} that the cobordism $T(\mathcal{R}_{\sigma_{0}, \sigma_{1}}(\sigma_{2}))\circ T(\mathcal{R}_{\sigma_{0}}(\sigma_{1}))$
lies in $\mathcal{V}^{k-1}$.
Combining both of our observations it follows that the map, 
$\displaystyle{\hocolim_{\mb{t}\in\mathcal{K}^{k}}}\mathcal{W}^{k}_{P(\sigma_{0}), \mb{t}} \longrightarrow \displaystyle{\hocolim_{\mb{t}\in\mathcal{K}^{k}}}\mathcal{W}^{k}_{P(\sigma_{0}, \sigma_{1}), \mb{t}},$
induced by the cobordism $T(\mathcal{R}_{\sigma_{0}}(\sigma_{1}))$ 
induces an injection and surjection on all homotopy groups, and thus is in $\mathcal{V}^{k-1}$.
Finally, it follows that $T(\sigma_{0}) = T(\sigma)$ lies in $\mathcal{V}^{k-1}$ by the $2$-out-of-$3$ property. 
\end{proof}

We will derive one more corollary. 
To state it we need one more definition and a lemma.
The following definition is a rewording of Definition \ref{defn: trivial and primitive surgery}, part (ii).
\begin{defn}  \label{defn: primitive surgeries 1}
Let $P \in \Ob\Cob^{k-1}_{\theta, d}$.
A  surgery $\sigma \in \Surg_{j}(P)$ is said to be \textit{primitive} if there exists another surgery $\alpha \in \Surg_{d-j-1}(P)$ such that 
$\sigma(S^{j}\times\{0\}) \subset P$ and $\alpha(S^{d-j-1}\times\{0\}) \subset P$ intersect transversally in $P$ at exactly one point. 
We denote by $\Surg^{\pr}_{j}(P)$ the set of all primitive surgeries in $P$ of degree $j$. 
\end{defn}

The following lemma says that the condition of being a trivial surgery is ``dual'' to the condition of being primitive. 
\begin{lemma} \label{lemma: primitive and trivial surgeries}
Let $P \in \Ob\Cob^{k-1}_{\theta, d}$ and $l$ be a non-negative integer. 
Let $\sigma \in \Surg^{\tr}_{l}(P)$ be a trivial surgery. 
Then its dual $\bar{\sigma}$ is a primitive surgery.
Similarly, if $\sigma$ is a primitive surgery then its dual $\bar{\sigma}$ is trivial.
\end{lemma}
\begin{proof}
Let $\sigma$ be a primitive surgery in $P$ of degree $l$. 
We first observe that in the special case where
\begin{itemize} \itemsep.2cm
\item[(i)] $P \cong S^{l}\times S^{d-l-1}$, and
\item[(ii)] $\sigma$ corresponds to (a thickening of) the standard inclusion $S^{l}\times\{0\} \hookrightarrow S^{l}\times S^{d-l-1}$,
\end{itemize}
it follows that $P(\sigma) \cong S^{d-1}$, and thus the dual $\bar{\sigma} \in \Surg_{d-l-1}(S^{d-1})$ has no choice but to be trivial. 
We claim that the general case of the lemma can always be reduced to this special case. 
Indeed, if $\sigma$ is primitive, then by definition there exists another embedding $\alpha: D^{l}\times S^{d-l-1} \longrightarrow P$ such that $\sigma(S^{l}\times\{0\})$ and $\alpha(\{0\}\times S^{d-l-1})$ intersect transversally in exactly one point. 
By shrinking down in the normal direction if necessary, we may assume that the intersection $\sigma(S^{l}\times D^{d-l-1})\cap\alpha(D^{l}\times S^{d-l-1})$ is diffeomorphic to a disk.
It follows that there is a diffeomorphism
\begin{equation} \label{equation: identification with a product of spheres}
S^{l}\times S^{d-l-1}\setminus\Int(D^{d-1}) \; \cong \; \sigma(S^{l}\times D^{d-l-1})\cup\alpha(D^{l}\times S^{d-l-1}).
\end{equation}
The surgery $\sigma$ is contained in the image of $S^{l}\times S^{d-l-1}\setminus\Int(D^{d-1})$ in $P$ under this diffeomorphism, and it follows that $\sigma$ corresponds to the standard embedding, $S^{l}\times\{0\} \hookrightarrow S^{l}\times S^{d-l-1}\setminus\Int(D^{d-1})$.
The manifold obtained from $S^{l}\times S^{d-l-1}\setminus\Int(D^{d-1})$ by performing surgery along this standard embedding is diffeomorphic to a disk, and thus it follows that the dual surgery $\bar{\sigma}$ factors through a disk in $P(\sigma)$. 
This proves that $\bar{\sigma}$ is trivial whenever $\sigma$ is primitive. 

The other claim in the statement of the lemma, that $\bar{\sigma}$ is primitive whenever $\sigma$ is trivial, is proven by a similar argument and we leave that to the reader. 
\end{proof}

\begin{corollary} \label{corollary: primitive surgery of degree k-1}
Let $k < d/2$ and $P \in \Cob^{k-1}_{\theta, d}$. 
Then for any primitive surgery $\sigma \in \Surg^{\pr}_{k-1}(P)$, the trace $T(\sigma)$ belongs to $\mathcal{V}^{k-1}$.
\end{corollary}
\begin{proof}
Consider the surgered manifold $P(\sigma)$ and the surgery dual to $\sigma$, $\bar{\sigma} \in \Surg_{d-k-1}(P(\sigma))$. 
Since $\sigma$ was a primitive surgery, it follows by Lemma \ref{lemma: primitive and trivial surgeries} that the dual surgery $\bar{\sigma}$ is a trivial surgery.
By Theorem \ref{corollary: trivial d-l handle} it then follows that $T(\bar{\sigma}) \in \mathcal{V}^{k-1}$. 
We then consider the composite
$$
\xymatrix{
P(\sigma) \ar[rr]^{T(\bar{\sigma})} && P \ar[rr]^{T(\sigma)} && P(\sigma). 
}
$$
It follows easily by inspection that the composite $T(\sigma)\circ T(\bar{\sigma})$ is equivalent to the cobordism $H_{k, d-k}(P)$ and thus is contained in $\mathcal{V}^{k-1}$ as well. 
The corollary then follows by the two-out-of-three property.
\end{proof}

\begin{remark} \label{remark: statements (a) and (b)}
We note that by Corollaries \ref{corollary: trivial d-l handle} and \ref{corollary: primitive surgery of degree k-1}, we have established parts (a) and (b) of Theorem \ref{theorem: morphism induce homotopy equivalence}. 
These corollaries of course depend on Theorem \ref{theorem: arbitrary handles} which we prove in the next section. 
Once Theorem \ref{theorem: arbitrary handles} is proven, the only thing left to do is prove part (c) of Theorem \ref{theorem: morphism induce homotopy equivalence}.
\end{remark}

\section{Resolving Nontrivial Handle Attachments} \label{section: Resolving Nontrivial Handle Attachments}
This section is geared toward proving Theorem \ref{theorem: arbitrary handles} which asserts that if
$\sigma$ is an \textbf{arbitrary} surgery in degree $k-1$ and $\alpha \in \mathcal{L}(P, \sigma)$, then
the trace $T(\sigma, \alpha)$ lies in $\mathcal{V}^{k-1}$.
\subsection{A semi-simplicial resolution} \label{subsection: a semi-simplicial resolution}
We will need to work with a semi-simplicial space analogous to the one constructed in \cite[Section 4.3]{GRW 16}.
\begin{defn} \label{defn: semi-simplicial space 1}
Let $l < d$. 
Fix an element 
$z := (M, (V, \sigma), e) \in \mathcal{W}^{k}_{P, \mb{t}}$ 
together with the following data:
\begin{itemize} \itemsep.2cm
\item an embedding 
$\chi: S^{l-1}\times(1, \infty)\times\R^{d-l-1} \longrightarrow P;$
\item a $1$-parameter family $\hat{\ell}^{\std}_{t}$, $t \in (2, \infty)$, of $\theta$-structures on $D^{l}\times D^{d-l}$ such that 
$$
\hat{\ell}^{\std}_{t}|_{\partial D^{l}\times D^{d-l}} = \chi^{*}\hat{\ell}_{P}|_{\partial D^{l}\times(t\cdot e_{1} + D^{d-l})},
$$
where $e_{1} \in (1, \infty)\times\R^{d-k-1}$ is the basis vector corresponding to the first coordinate. 
\end{itemize}
\vspace{.2cm}

We define $\mb{X}_{0}(z; \chi)$ to be the set of triples $(t, \phi, \hat{L})$ consisting of a $t \in (2, \infty)$, an embedding 
$\phi: (D^{l}\times D^{d-l}, \; \partial D^{l}\times D^{d-l}) \longrightarrow (M, P),$
and a path of $\theta$-structures $\hat{L}$ on $D^{l}\times D^{d-l}$ 
such that:
\begin{enumerate} \itemsep.3cm
\item[(i)] the restriction of the map 
$\phi$ to $\partial D^{l}\times D^{d-l}$ 
satisfies the equation
$\phi(y, v) = \chi(\tfrac{y}{|y|}, v + t\cdot e_{1})$ 
for all $(y, v) \in \partial D^{l}\times D^{d-l}$.
\item[(ii)] the image $\phi(D^{l}\times D^{d-l})$ is disjoint from the image of the embedding $e$, associated to the element $z = (M, (V, \sigma), e)$.
\item[(iii)] the family of $\theta$-structures $\hat{L}$ satisfies: $\hat{L}(0) = \phi^{*}\hat{\ell}_{M}$, $\hat{L}(1) = \hat{\ell}^{\std}_{t}$, and
$\hat{L}(s)|_{\partial D^{l}\times D^{d-l}}$ is independent of $s \in [0,1]$. 
\end{enumerate}
For $p \in \Z_{\geq 0}$,
let $\mb{X}_{p}(z; \chi) \subset \left(\mb{X}_{0}(z; \chi)\right)^{p+1}$ be the subspace consisting of those
$$\left((t_{0}, \phi_{0}, \hat{L}_{0}), \dots, (t_{p}, \phi_{p}, \hat{L}_{p})\right)$$ 
such that 
\begin{enumerate} \itemsep.2cm
\item[(i)] the images $\phi_{i}(D^{l}\times D^{d-l})$ are disjoint for $i = 0, \dots, p$, 
\item[(ii)] $t_{0} < t_{1} < \cdots < t_{p}$.
\end{enumerate}
The collection $\mb{X}_{\bullet}(z; \chi)$ has the structure of a semi-simplicial space, where the $i$th face map forgets the entry $(t_{i}, \phi_{i})$. 
The definition of the semi-simplicial space $\mb{X}_{\bullet}(z; \chi)$ does definitely depend on the one-parameter family $\hat{\ell}^{\std}_{t}$.
However, we drop it from the notation in order to save space. 
\end{defn}

We now combine all of the semi-simplicial spaces $\mb{X}_{\bullet}(z; \chi)$ together into one semi-simplicial space augmented over $\mathcal{W}^{k}_{P, \mb{t}}$.
\begin{defn} \label{defn: semi simplicial resolution}
Choose an integer $k$. 
Fix the following data:
\begin{itemize} \itemsep.2cm
\item an object $\mb{t} \in \mathcal{K}^{k}$;
\item an object $P \in \Ob\Cob^{k-1}_{\theta, d}$; 
\item an embedding $\chi: \partial D^{l}\times(1, \infty)\times\R^{d-l-1} \longrightarrow P$, and one parameter family of $\theta$-structures $\hat{\ell}^{\std}_{t}$ as in Definition \ref{defn: semi-simplicial space 1}. 
\end{itemize}
For $p \in \Z_{\geq 0}$, we define $\mathcal{M}^{k}_{P, \mb{t}}(\chi)_{p}$ be the space of tuples $(y; x)$ with $y \in \mathcal{W}^{k}_{P, \mb{t}}$ and $x \in \mb{X}_{p}(y; \chi)$.
The assignment $[p] \mapsto \mathcal{M}^{k}_{P, \mb{t}}(\chi)_{p}$ defines a semi-simplicial space $\mathcal{M}^{k}_{P, \mb{t}}(\chi)_{\bullet}$.
The forgetful maps 
$$
\mathcal{M}^{k}_{P, \mb{t}}(\chi)_{p} \longrightarrow \mathcal{W}^{k}_{P, \mb{t}}, \quad (y; x) \mapsto y
$$
yield an augmented semi-simplicial space, 
$
\mathcal{M}^{k}_{P, \mb{t}}(\chi)_{\bullet} \longrightarrow \mathcal{M}^{k}_{P, \mb{t}}(\chi)_{-1} = \mathcal{W}^{k}_{P, \mb{t}}.
$
\end{defn}

We have the following proposition. 

\begin{proposition} \label{proposition: augmentation is weak equivalence}
Let $P \in \Ob\Cob^{k-1}_{\theta, d}$.
Let $\chi: \partial D^{l}\times(1, \infty)\times\R^{d-l-1} \longrightarrow P$ and $\hat{\ell}^{\std}_{t}$ be as in Definition \ref{defn: semi-simplicial space 1}. 
If $l \leq k < d/2$,
then the map induced by augmentation
$$
|\mathcal{M}^{k}_{P, \mb{t}}(\chi)_{\bullet}| \longrightarrow \mathcal{M}^{k}_{P, \mb{t}}(\chi)_{-1} = \mathcal{W}^{k}_{P, \mb{t}}
$$
is a weak homotopy equivalence. 
\end{proposition}
\begin{proof}
To prove the theorem we will need to replace $\mathcal{M}^{k}_{P, \mb{t}}(\chi)_{\bullet}$ with a weakly equivalent semi-simplicial space that is a bit more flexible to work with. 
For each $p \in \Z_{\geq 0}$, let $\widetilde{\mathcal{M}}^{k}_{P, \mb{t}}(\chi)_{p}$ consist of those tuples
$((M, (V, \sigma), e); \; (t_{0}, \phi_{0}, \hat{L}_{0}), \dots, (t_{p}, \phi_{p}, \hat{L}_{p}))$
defined as in Definition \ref{defn: semi simplicial resolution} (and Definition \ref{defn: semi-simplicial space 1}), except now we relax condition (i) and only require, 
$
\phi_{i}(D^{l}\times\{0\})\cap\phi_{j}(D^{l}\times\{0\}) = \emptyset
$
whenever $i \neq j$.
It follows by an argument similar the one used in \cite[Proposition 6.9]{GRW 14} that the inclusion 
$
\mathcal{M}^{k}_{P, \mb{t}}(\chi)_{\bullet} \hookrightarrow \widetilde{\mathcal{M}}^{k}_{P, \mb{t}}(\chi)_{\bullet}
$
is a level-wise weak homotopy equivalence.
To prove the proposition it will suffice to prove that $|\widetilde{\mathcal{M}}^{k}_{P, \mb{t}}(\chi)_{\bullet}| \longrightarrow \widetilde{\mathcal{M}}^{k}_{P, \mb{t}}(\chi)_{-1}$ is a weak homotopy equivalence.
Now, the augmented semi-simplicial space $\widetilde{\mathcal{M}}^{k}_{P, \mb{t}}(\chi)_{\bullet} \longrightarrow \widetilde{\mathcal{M}}^{k}_{P, \mb{t}}(\chi)_{-1}$ is an augmented topological flag complex in the sense of \cite[Definition 6.1]{GRW 14}.
By \cite[Theorem 6.2]{GRW 14}, to prove the proposition it will suffice to verify that it has the following three properties (compare with the proof of Proposition \ref{proposition: cobcat to long manifolds}):
\begin{enumerate} \itemsep.1cm
\item[(i)] The augmentation map $\epsilon_{0}: \widetilde{\mathcal{M}}^{k}_{P, \mb{t}}(\chi)_{0} \longrightarrow \widetilde{\mathcal{M}}^{k}_{P, \mb{t}}(\chi)_{-1}$ has local lifts of any map from a disk. 
\item[(ii)] For any $x \in \widetilde{\mathcal{M}}^{k}_{P, \mb{t}}(\chi)_{-1}$ the fibre $\epsilon_{0}^{-1}(x)$ is non-empty.
\item[(iii)] For any $x \in \widetilde{\mathcal{M}}^{k}_{P, \mb{t}}(\chi)_{-1}$, given any non-empty set of zero-simplices $v_{1}, \dots, v_{m} \in \epsilon_{0}^{-1}(x)$, there exists another zero simplex $v\in \epsilon_{0}^{-1}(x)$ such that $(v, v_{i}) \in \widetilde{\mathcal{M}}^{k}_{P, \mb{t}}(\chi)_{1}$ for all $i = 1, \dots, m$.
\end{enumerate}
Property (i) follows from the fact that $\epsilon_{0}: \widetilde{\mathcal{M}}^{k}_{P, \mb{t}}(\chi)_{0} \longrightarrow \widetilde{\mathcal{M}}^{k}_{P, \mb{t}}(\chi)_{-1}$ is clearly a locally trivial fibre bundle. 
Property (iii) follows from a general position argument. 
Indeed, let $(t_{1}, \phi_{1}, \hat{L}_{1}), \dots, (t_{m}, \phi_{m}, \hat{L}_{m})$ be a collection of zero simplices in $\epsilon_{0}^{-1}(M, (V, \sigma), e)$, for some element $(M, (V, \sigma), e) \in \widetilde{\mathcal{M}}^{k}_{P, \mb{t}}(\chi)_{-1}$. 
Let $\phi = \phi_{1}$. 
Since $l \leq k < d/2$, we may perturb $\phi$ to a new embedding $\phi'$ with
$$\phi'(D^{l}\times\{0\})\cap\phi_{i}(D^{l}\times\{0\}) = \emptyset, \quad \text{for $i = 1, \dots, m$.}$$ 
By choosing $t'$ with $t' \neq t_{i}$ for all $i$ and setting $\hat{L}' = \hat{L}_{1}$, it follows that $(t, \phi', \hat{L}')$ is a zero simplex with $((t', \phi', \hat{L}'), \; (t_{i}, \phi_{i}, \hat{L})) \in \widetilde{\mathcal{M}}^{k}_{P, \mb{t}}(\chi)_{1}$ for all $i$.

We now just have to verify property (ii). 
Let $(M, (V, \sigma), e) \in \widetilde{\mathcal{M}}^{k}_{P, \mb{t}}(\chi)_{-1}$.
Pick a number $t \in (2, \infty)$ and consider the commutative square
\begin{equation} \label{equation: theta structure commutative square}
\xymatrix{
\partial D^{l}\times D^{d-l-1} \ar[d] \ar[rrr]^{\chi(\text{--}, \text{--}+t\cdot e_{1})} &&& P \ar@{->}[r] & M \ar[d]^{\ell_{M}} \\
D^{l}\times D^{d-l} \ar@{-->}[urrrr]^{g} \ar[rrrr]^{\ell^{\std}_{t}} &&&& B.
}
\end{equation}
By definition of $\widetilde{\mathcal{M}}^{k}_{P, \mb{t}}(\chi)_{-1}$ the map $\ell_{M}: M \longrightarrow B$ is $k$-connected. 
Since $l \leq k$, it follows that the lift $g$ in diagram (\ref{equation: theta structure commutative square}) exists, making both the upper and lower triangles commute.
Since both $\ell_{M}$ and $\ell^{\std}_{t}$ are covered by bundle maps, this provides a bundle map $\hat{g}: T(D^{l}\times D^{d-l}) \longrightarrow M$ covering $g$, such that $\hat{\ell}_{M}\circ\hat{g}$ is homotopic to $\hat{\ell}^{\std}_{t}$ through bundle maps, via a homotopy that is constant over $\partial D^{l}\times D^{d-l}$. 
By the \textit{Smale-Hirsch theorem} the pair $(g, \hat{g})$ may be homotoped relative to $\partial D^{l}\times D^{d-l}$ to a pair of the form $(\phi, D(\phi))$, for an immersion $\phi$ and $D\phi$ its differential. 
Since $k < \dim(M)/2 = d/2$, by putting $\phi(D^{l}\times\{0\})$ into a general position we may assume that the restriction $\phi$ to $(D^{l}\times\{0\})\cup(\partial D^{l}\times D^{d-l})$ is an embedding. 
By precomposing $\phi$ with an isotopy of $D^{l}\times D^{d-l}$ that compresses $D^{l}\times D^{d-l}$ down to the core $(D^{l}\times\{0\})\cup(\partial D^{l}\times D^{d-l})$, 
we may then make $\phi$ into an embedding. 
The $\theta$-structure $\phi^{*}\hat{\ell}_{M} = \hat{\ell}_{M}\circ D\phi$ is still homotopic to $\hat{\ell}^{\std}_{t}$, and if we choose such a homotopy, $\hat{L}$, then we have constructed an element $(t, \phi, \hat{L}) \in \epsilon^{-1}_{0}(M, (V, \sigma), e)$.
This establishes condition (ii) and concludes the proof of the proposition.
\end{proof}

The correspondence $\mb{t} \mapsto (\mathcal{M}^{k}_{P, \mb{t}}(\chi)_{\bullet} \longrightarrow \mathcal{M}^{k}_{P, \mb{t}}(\chi)_{-1})$ defines a functor from $\mathcal{K}^{k}$ to the category of augmented semi-simplicial spaces. 
Taking the levelwise homotopy colimit yields the augmented semi-simplicial space, 
$
\displaystyle{\hocolim_{\mb{t} \in \mathcal{K}^{k}}}\mathcal{M}^{k}_{P, \mb{t}}(\chi)_{\bullet} \; \longrightarrow \; \displaystyle{\hocolim_{\mb{t} \in \mathcal{K}^{k}}}\mathcal{M}^{k}_{P, \mb{t}}(\chi)_{-1}.
$
We have the following corollary.
\begin{corollary} \label{corollary: homotopy colimit augmentation}
Let $P \in \Ob\Cob^{k-1}_{\theta, d}$.
Let $\chi: \partial D^{l}\times(1, \infty)\times\R^{d-l-1} \longrightarrow P$ and $\hat{\ell}^{\std}_{t}$ be as in Definition \ref{defn: semi-simplicial space 1}. 
If $l \leq k < d/2$,
then the map induced by augmentation
$$
|\hocolim_{\mb{t} \in \mathcal{K}^{k}}\mathcal{M}^{k}_{P, \mb{t}}(\chi)_{\bullet}| \longrightarrow \hocolim_{\mb{t} \in \mathcal{K}^{k}}\mathcal{M}^{k}_{P, \mb{t}}(\chi)_{-1}$$
is a weak homotopy equivalence. 
\end{corollary}
\begin{proof}
By Proposition \ref{proposition: augmentation is weak equivalence} the map $|\mathcal{M}^{k}_{P, \mb{t}}(\chi)_{\bullet}| \longrightarrow \mathcal{M}^{k}_{P, \mb{t}}(\chi)_{-1}$ is a weak homotopy equivalence for all $\mb{t} \in \mathcal{K}^{k}$.
By the homotopy invariance of homotopy colimits, the augmentation map induces a weak homotopy equivalence, 
$\displaystyle{\hocolim_{\mb{t} \in \mathcal{K}^{k}}}|\mathcal{M}^{k}_{P, \mb{t}}(\chi)_{\bullet}| \; \stackrel{\simeq}  \longrightarrow \; \displaystyle{\hocolim_{\mb{t} \in \mathcal{K}^{k}}}\mathcal{M}^{k}_{P, \mb{t}}(\chi)_{-1}.$
The corollary then follows by combining this with the weak homotopy equivalence,
$$
\hocolim_{\mb{t} \in \mathcal{K}^{k}}|\mathcal{M}^{k}_{P, \mb{t}}(\chi)_{\bullet}| \simeq  |\hocolim_{\mb{t} \in \mathcal{K}^{k}}\mathcal{M}^{k}_{P, \mb{t}}(\chi)_{\bullet}|.
$$
\end{proof}

\subsection{Resolving the map $\mb{S}_{T(\sigma)}$} \label{subsection: resolving composition}
We now show how to use the above semi-simplicial resolution to prove Theorem \ref{theorem: arbitrary handles}. 
For the rest of this section, let us fix once and for all an object $P \in \Ob\Cob^{k-1}_{\theta, d}$. 
For each $l \in \Z_{\geq 0}$, 
fix once and for all a diffeomorphism 
\begin{equation} \label{equation: collar extension diffeomorphism}
\gamma_{l}: D^{l}\times D^{d-l} \stackrel{\cong} \longrightarrow D^{l}\times D^{d-l}\bigcup S^{l-1}\times[0,1]\times D^{d-l}
\end{equation}
which is the identity on $D^{l}(1/2)\times D^{d-l}$ (where $D^{l}(1/2) \subset D^{l}$ is the disk centered at the origin with radius $1/2$).
\begin{Construction}[A semi-simplicial map] \label{construction: induced trace simplicial map}
Let $P \in \Ob\Cob^{k-1}_{\theta, d}$. 
Let 
$$\chi: \partial D^{l}\times(1, \infty)\times\R^{d-l-1} \longrightarrow P$$ 
be as in Definition \ref{defn: semi-simplicial space 1}. 
Let $\sigma \in \Surg_{j-1}(P)$
and suppose that 
$$
\sigma(D^{j+1}\times D^{d-j-1})\cap\chi(\partial D^{l}\times(1, \infty)\times\R^{d-l-1}) = \emptyset.
$$ 
From this disjointness condition, we can consider $\chi$ to be an embedding into $P(\sigma)$ as well and thus the semi-simplicial space $\mathcal{M}^{k}_{P(\sigma), \mb{t}}(\chi)_{\bullet}$ is well defined. 
Let 
$$(W; \; (t_{0}, \phi_{0}, \hat{L}_{0}), \dots, (t_{p}, \phi_{p}, \hat{L}_{p})) \in \mathcal{M}^{k}_{P, \mb{t}}(\chi)_{p}.$$ 
To this $p$-simplex we associate a new $p$-simplex in $\mathcal{M}^{k}_{P(\sigma), \mb{t}}(\chi)_{\bullet}$ as follows: 
First note that from the definition of $T(\sigma)$ (Definition \ref{defn: trace of surgery}) for each $i$ we have
$[0,1]\times\phi_{i}(S^{l-1}\times D^{d-l}) \; \subset \; T(\sigma).$
For each $i$, we may define a new embedding 
$
\phi_{i}(\sigma): D^{l}\times D^{d-l} \longrightarrow W\cup_{P}T(\sigma),
$
by precomposing
$$
\xymatrix{
(D^{l}\times D^{d-l})\bigcup \left([0,1]\times(S^{l-1}\times D^{d-l})\right) \ar[rrrr]^{ \ \ \ \ \ \ \ \ \ \ \ \ \ \ \ \phi_{i}\cup(\Id_{[0,1]}\times\partial\phi_{i})} &&&& 
W\bigcup_{P}T(\sigma).
}
$$
with the diffeomorphism 
$$
\gamma_{l}: (D^{l}\times D^{d-l}) \stackrel{\cong} \longrightarrow
D^{l}\times D^{d-l}\bigcup[0,1]\times(S^{l-1}\times D^{d-l}).
$$
Using this diffeomorphism $\gamma_{l}$, each $\hat{L}_{i}$ can also be extended cylindrically up the trace $T(\sigma)$, to yield new one-parameter families of $\theta$-structures 
$\hat{L}_{0}(\sigma), \dots, \hat{L}_{p}(\sigma),$ 
satisfying the conditions of Definition \ref{defn: semi-simplicial space 1} with respect to $(t_{0}, \phi_{0}(\sigma)), \dots, (t_{p}, \phi_{p}(\sigma))$.
The correspondence
$$
\left(M; (t_{0}, \phi_{0}, \hat{L}_{0}), \dots, (t_{p}\phi_{p}, \hat{L}_{p})\right) \; \mapsto \; \left(M\cup T(\sigma); \; (t_{0}, \phi_{0}(\sigma), \hat{L}_{0}(\sigma)), \dots,  (t_{p}, \phi_{p}(\sigma), \hat{L}_{p}(\sigma))\right)
$$
defines a map 
$$
\mb{T}(\sigma)_{p}: \mathcal{M}^{k}_{P, \mb{t}}(\chi)_{p} \; \longrightarrow \; \mathcal{M}^{k}_{P(\sigma), \mb{t}}(\chi)_{p}.
$$
It is clear that $d_{i}\circ\mb{T}(\sigma)_{p} \; = \; \mb{T}(\sigma)_{p}\circ d_{i}$ for all face maps $d_{i}$, and thus we obtain a semi-simplicial map 
\begin{equation} \label{equation: induced trace map p level}
\mb{T}(\sigma)_{\bullet}: \mathcal{M}^{k}_{P, \mb{t}}(\chi)_{\bullet} \; \longrightarrow \; \mathcal{M}^{k}_{P(\sigma), \mb{t}}(\chi)_{\bullet}.
\end{equation}
\end{Construction}
Notice that for $p = -1$, the map $\mb{T}(\sigma)_{-1}$ is precisely the map from Section \ref{section: Cobordism Categories and Handle Attachments},
$$
\mb{S}_{T(\sigma)}: \mathcal{W}^{k}_{P, \mb{t}} \; \longrightarrow \; \mathcal{W}^{k}_{P(\sigma), \mb{t}},
$$
induced by the cobordism $T(\sigma): P \rightsquigarrow P(\sigma)$.
Thus the semi-simplicial map $\mb{T}(\sigma)_{\bullet}$ covers $\mb{S}_{T(\sigma)}$.

We will need to use one more construction.
\begin{Construction} \label{Construction: auxiliary displaced trace}
Let $P \in \Cob^{k-1}_{\theta, d}$ be the same object fixed from the beginning of this subsection. 
Let 
$\chi: \partial D^{l}\times(1, \infty)\times\R^{d-l-1} \longrightarrow P$ 
and $\hat{\ell}^{\std}_{t}$
be as in Definition \ref{defn: semi-simplicial space 1}, with respect to $P$. 
Choose an extension of the embedding $\chi$, 
\begin{equation} \label{equation: extension of chi}
\widetilde{\chi}: D^{l}\times(1, \infty)\times\R^{d-l-1} \longrightarrow \R^{\infty}
\end{equation}
with the properties: 
\begin{itemize} \itemsep.2cm
\item $\widetilde{\chi}^{-1}(P) = \partial D^{l}\times(1, \infty)\times\R^{d-l-1}$,
\item $\widetilde{\chi}|_{\partial D^{l}\times(1, \infty)\times\R^{d-l-1}} \; = \; \chi$.
\end{itemize}
For $i \in \Z_{\geq 0}$, we define an embedding 
\begin{equation} \label{equation: chi i embedding}
 \chi_{i}: D^{l}\times D^{d-l} \longrightarrow \R^{\infty}
\end{equation}
by the formula
$$
\chi_{i}(x, \; y) \; = \; \widetilde{\chi}(x, \; 3(i+1)e_{0} + y)
$$
(in the formula $3(i+1)e_{0} + y$, the term $e_{0}$ denotes the basis vector corresponding to the first coordinate in $(1, \infty)\times\R^{d-k-1}$).
For each $i$, we define a $\theta$-structure 
$$\hat{L}_{i}(\chi): TD^{l}\times D^{d-l} \longrightarrow \theta^{*}\gamma^{d}$$ 
by setting 
$\hat{L}_{i}(\chi) = \chi_{i}^{*}\hat{\ell}^{\std}_{3i+1}.$
Equipped with the $\theta$-structure $\hat{L}_{i}(\chi)$, the embedding $\chi_{i}$ is surgery data of degree $l-1$ in the manifold $P$. 

For each $i$, consider the surgery data given by $\chi_{i}$. 
For $p \in \Z_{\geq 0}$, let us denote by $\chi(p)$ the truncated sequence of surgery data $(\chi_{0}, \chi_{1}, \dots, \chi_{p})$.
We let $\bar{\chi}(p)$ denote the sequence of surgery data $(\bar{\chi}_{0}, \bar{\chi}_{1}, \dots, \bar{\chi}_{p})$, where each $\bar{\chi}_{i}$ is the surgery in $P(\chi(p))$ dual to $\chi_{i}$ (see Definition \ref{defn: dual surgery}). 
Since the embeddings $\bar{\chi}_{i}$ are all disjoint, we may form their (simultaneous) trace,
$T(\bar{\chi}(p)): P(\chi(p)) \rightsquigarrow P,$
and consider its induced map 
$
\mb{S}_{T(\bar{\chi}(p))}: \mathcal{W}^{k}_{P(\chi(p)), \mb{t}} \longrightarrow \mathcal{W}^{k}_{P, \mb{t}}.
$
By Construction \ref{Construction: trace of surgery}, the trace $T(\bar{\chi}(p))$ is equipped with embeddings  
\begin{equation} \label{equation: psi-handle i}
\psi_{i}: (D^{l}\times D^{d-l}, S^{l-1}\times D^{d-l}) \; \longrightarrow \; (T(\bar{\chi}(p)), \; P)
\end{equation}
for $i = 0, \dots, p$, where 
$\psi_{i}(D^{l}\times D^{d-l})\cap\psi_{j}(D^{l}\times D^{d-l}) = \emptyset$ for $i \neq j$.
These embeddings are the \textit{standard core handles} as described in Construction \ref{Construction: trace of surgery}. 
By how $\chi_{i}$ was defined, it follows that each restricted embedding $\partial\psi_{i} = \psi_{i}|_{S^{l-1}\times D^{d-l}}$ satisfies the equation
$$
\partial\psi_{i}(x, y) \; = \; \chi(x, y + 3(i+1)\cdot e_{0})
$$
for all $(x, y) \in D^{l}\times D^{d-l}$.
Using this, 
we may define a map
\begin{equation} \label{equation: levelwise identification}
\mathcal{H}(\chi)_{p}: \mathcal{W}^{k}_{P(\chi(p)), \mb{t}} \; \longrightarrow  \; \mathcal{M}^{k}_{P, \mb{t}}(\chi)_{p} 
\end{equation}
by sending 
$x \in \mathcal{W}^{k}_{P(\chi(p)), \mb{t}}$ 
to the $p$-simplex given by 
$$
\left(\mb{S}_{T(\bar{\chi}(p))}(x); \; \left((1, \psi_{0}, \hat{L}_{0}(\chi)),\; \dots, \; (3i+1, \psi_{i}, \hat{L}_{i}(\chi)), \; \dots, \; (3p+1, \psi_{p}, \hat{L}_{p}(\chi_{p}))\right)\right),
$$
where each $\hat{L}_{i}(\chi)$ is considered to be the constant one-parameter family of $\theta$-structures as defined above.
By how each $\chi_{i}$ was defined, it follows that this formula does indeed give a well defined map.
\end{Construction}

The following lemma is similar to \cite[Lemma 6.10]{GRW 16}.
\begin{lemma} \label{lemma: levelwise identification} 
For all $p \in \Z_{\geq 0}$, the map induced by $\mathcal{H}(\chi)_{p}$,
$$
\hocolim_{\mb{t} \in \mathcal{K}^{k}}\mathcal{W}^{k}_{P(\chi(p)), \mb{t}} \; \longrightarrow  \; \hocolim_{\mb{t} \in \mathcal{K}^{k}}\mathcal{M}^{k}_{P, \mb{t}}(\chi)_{p}, 
$$
is a weak homotopy equivalence. 
\end{lemma}
\begin{proof}
It will suffice to show that $\mathcal{W}^{k}_{P(\chi(p)), \mb{t}} \; \longrightarrow  \; \mathcal{M}^{k}_{P, \mb{t}}(\chi)_{p}$ is a weak homotopy equivalence for all $\mb{t}$. 
We will do this explicitly for the case that $p = 0$;
the case for general $p$ being similar. 
Let $\mathcal{E}_{P}(\chi)$ denote the set of tuples $(t, \phi, \hat{L})$ with 
\begin{itemize} \itemsep.2cm
\item $t \in (1, \infty)$, 
\item $\phi: \left(D^{l}\times D^{d-l}, \; \partial D^{l}\times D^{d-l}\right) \; \longrightarrow \; \left((-\infty, 0]\times\R^{\infty-1}, \; \{0\}\times P\right)$ is an embedding, 
\item $\hat{L}$ is a one-parameter family of $\theta$-structures on $D^{l}\times D^{d-l}$,
\end{itemize}
subject to conditions similar to those from Definition \ref{defn: semi-simplicial space 1}. 
Since the target space of the embedding $\phi$ in the above definition is infinite dimensional Euclidean space, it follows that $\mathcal{E}_{P}(\chi)$ is weakly contractible. 
There is a natural transformation 
\begin{equation} \label{equation: natural transformation}
\mathcal{M}^{k}_{P, \mb{t}}(\chi)_{0} \; \longrightarrow \; \mathcal{E}_{P}(\chi)
\end{equation} 
defined by 
$$\left((M, (V, \sigma), e), (t, \phi, \hat{L})\right) \; \mapsto \; (t, i_{M}\circ\phi, \hat{L}),$$
where on the right $i_{M}: (M, P) \hookrightarrow ((-\infty, 0]\times\R^{\infty-1}, \; \{0\}\times P)$ is the inclusion map. 
This natural transformation has the concordance lifting property. 
The proof then follows by identifying $\mathcal{W}^{k}_{P(\chi(0)), \mb{t}}$ with the fibre of (\ref{equation: natural transformation}).
This is similar to what was done in \cite[Lemma 6.10]{GRW 16} and we leave this step to the reader. 
\end{proof}

With the above constructions in place we can now prove Theorem \ref{theorem: arbitrary handles}. 
\begin{proof}[Proof of Theorem \ref{theorem: arbitrary handles}]
Let $P \in \Ob\Cob^{k-1}_{\theta, d}$ be as in the statement of Theorem \ref{theorem: arbitrary handles}.
Let 
$\sigma \in \Surg_{k-1}(P)$.
Let $\alpha \in \mathcal{L}(P, \sigma)$. 
Choose an embedding 
$\chi: \partial D^{k}\times(1, \infty)\times\R^{d-k-1} \longrightarrow P$
as in Definition \ref{defn: semi-simplicial space 1}, but with the extra properties:
\begin{enumerate} \itemsep.2cm
\item[(a)] $\chi(\partial D^{k}\times(1, \infty)\times\R^{d-k-1})\cap \sigma(\partial D^{k}\times D^{d-k}) = \emptyset$,
\item[(b)] $\chi(\partial D^{k}\times(1, \infty)\times\R^{d-k-1})\cap \alpha(D^{d-k}\times D^{k}) = \emptyset,$
\item[(c)] $\chi|_{\partial D^{k}\times\{0\}}: S^{k-1} \longrightarrow P$ is isotopic to the embedding $\sigma|_{S^{k-1}\times\{0\}}$.
\end{enumerate}
Such an embedding $\chi$ can always be found by picking a small displacement of $\sigma$. 
Chosen in this way, we consider the semi-simplicial space $\mathcal{M}^{k}_{P, \mb{t}}(\chi)_{\bullet}$.
We may utilize Constructions \ref{construction: induced trace simplicial map} and \ref{Construction: auxiliary displaced trace} for these particular choices of $\chi$ and $\sigma$. 

Let $p \in \Z_{\geq 0}$.
Consider the cobordism
$
T(\mathcal{R}_{\chi(p)}(\sigma, \alpha)): P(\chi(p)) \rightsquigarrow P(\mathcal{R}_{\chi(p)}(\sigma, \alpha)),
$
and its map induced by $\mb{S}_{T(\mathcal{R}_{\chi(p)}(\sigma, \alpha))}$,
\begin{equation} \label{equation: induced concatenation map}
\hocolim_{\mb{t}\in\mathcal{K}^{k}}\mathcal{W}^{k}_{P(\chi(p)), \mb{t}} \; \longrightarrow  \; \hocolim_{\mb{t}\in\mathcal{K}^{k}}\mathcal{W}^{k}_{P(\mathcal{R}_{\chi(p)}(\sigma, \alpha)), \mb{t}}. 
\end{equation}
By condition (c), it follows that in the manifold $P(\chi(p))$ the embedding 
$\sigma|_{S^{l-1}\times\{0\}}: S^{k-1} \longrightarrow P(\chi(p))$
is null-homotopic, and since $k < d/2$ it follows further that this embedding factors through an embedding $\R^{d-1} \hookrightarrow P(\chi(p))$.
It follows that $\mathcal{R}_{\chi(p)}(\sigma)$ is a trivial surgery, and thus 
by Proposition \ref{proposition: tori composite} the map 
(\ref{equation: induced concatenation map}) is a homotopy equivalence, and we conclude that the cobordism $T(\mathcal{R}_{\chi(p)}(\sigma, \alpha))$ lies in $\mathcal{V}^{k-1}$ and thus (\ref{equation: induced concatenation map}) is a weak homotopy equivalence. 
Let us denote the map (\ref{equation: induced concatenation map}) by 
$F_{p}$.

Recall the map
$$
\mathcal{H}(\chi)_{p}: \hocolim_{\mb{t}\in\mathcal{K}^{k}}\mathcal{W}^{k}_{P(\chi(p)), \mb{t}} \; \longrightarrow  \; \hocolim_{\mb{t} \in \mathcal{K}^{k}}\mathcal{M}^{k}_{P, \mb{t}}(\chi)_{p},
$$
which is a weak homotopy equivalence by Lemma \ref{lemma: levelwise identification}.
By the disjointness conditions (a) and (b), this map induces a similar map 
$$
\mathcal{H}(\chi)'_{p}: \hocolim_{\mb{t} \in \mathcal{K}^{k}}\mathcal{W}^{k}_{P(\mathcal{R}_{\chi(p)}(\sigma, \alpha)), \; \mb{t}}
\; \longrightarrow \; \hocolim_{\mb{t} \in \mathcal{K}^{k}}\mathcal{M}^{k}_{P(\sigma, \alpha), \mb{t}}(\chi)_{p}
$$
defined in a similar way, and which is also a weak homotopy equivalence. 
We obtain the commutative diagram, 
$$
\xymatrix{
\displaystyle{\hocolim_{\mb{t}\in\mathcal{K}^{k}}}\mathcal{W}^{k}_{P(\chi(p)), \mb{t}} \ar[d]^{\mathcal{H}(\chi)_{p}}_{\simeq} \ar[rrr]^{F_{p}}_{\simeq} &&& \displaystyle{\hocolim_{\mb{t}\in\mathcal{K}^{k}}}\mathcal{W}^{k}_{P(\mathcal{R}_{\chi(p)}(\sigma, \alpha)), \mb{t}} \ar[d]^{\mathcal{H}(\chi)'_{p}}_{\simeq} \\
\displaystyle{\hocolim_{\mb{t} \in \mathcal{K}^{k}}}\mathcal{M}^{k}_{P, \mb{t}}(\chi)_{p}\ar[rrr]^{\mb{T}(\sigma, \alpha)_{p}}   &&& \displaystyle{\hocolim_{\mb{t} \in \mathcal{K}^{k}}}\mathcal{M}^{k}_{P(\sigma, \alpha), \mb{t}}(\chi)_{p}.  
}
$$
By commutativity, it follows that $\mb{T}(\sigma, \alpha)_{p}$ is a weak homotopy equivalence for each $p$ and thus $\mb{T}(\sigma, \alpha)_{\bullet}$ induces a weak homotopy equivalence on geometric realization, 
$$
|\displaystyle{\hocolim_{\mb{t} \in \mathcal{K}^{k}}}\mathcal{M}^{k}_{P, \mb{t}}(\chi)_{\bullet}| \stackrel{\simeq} \longrightarrow |\displaystyle{\hocolim_{\mb{t} \in \mathcal{K}^{k}}}\mathcal{M}^{k}_{P(\sigma, \alpha), \; \mb{t}}(\chi)_{\bullet}|.
$$
Theorem \ref{theorem: arbitrary handles} then follows from the commutative diagram
$$
\xymatrix{
|\displaystyle{\hocolim_{\mb{t} \in \mathcal{K}^{k}}}\mathcal{M}^{k}_{P, \mb{t}}(\chi)_{\bullet}| \ar[rrr]^{\simeq} \ar[d]^{\simeq} &&&  |\displaystyle{\hocolim_{\mb{t} \in \mathcal{K}^{k}}}\mathcal{M}^{k}_{P(\sigma, \alpha), \; \mb{t}}(\chi)_{\bullet}| \ar[d]^{\simeq} \\
\displaystyle{\hocolim_{\mb{t} \in \mathcal{K}^{k}}}\mathcal{W}^{k}_{P, \mb{t}} \ar[rrr] &&& \displaystyle{\hocolim_{\mb{t} \in \mathcal{K}^{k}}}\mathcal{W}^{k}_{P(\sigma, \alpha), \mb{t}},
}
$$
whose vertical arrows are weak homotopy equivalences by Corollary \ref{corollary: homotopy colimit augmentation}. 
\end{proof}

\section{Higher Index Handles} \label{section: higher index handles}
In this section we prove the following proposition.
\begin{proposition} \label{proposition: arbitrary d-l surgeries}
Let $k < d/2$ and $P \in \Ob\Cob^{k-1}_{\theta, d}$.
Let $\sigma \in \Surg_{j}(P)$ with $d-k \leq j < d-1$. 
Then the trace $T(\sigma): P \rightsquigarrow P(\sigma)$ lies in $\mathcal{V}^{k-1}$.
\end{proposition}
The proof of the above proposition will use Corollary \ref{corollary: trivial d-l handle} together with the results developed in the following subsection.
\subsection{Primitive surgeries} \label{subsection: primitive surgeries}
We will need to work with primitive surgeries.
Let $P \in \Ob\Cob^{k-1}_{\theta, d}$.
Recall from Definition \ref{defn: primitive surgeries 1} that a surgery $\sigma \in \Surg_{j}(P)$ is said to be \textit{primitive} if there exists another surgery $\alpha \in \Surg_{d-j-1}(P)$ such that 
$\sigma(S^{j}\times\{0\}) \subset P$ and $\alpha(S^{d-j-1}\times\{0\}) \subset P$ intersect transversally in $P$ at exactly one point. 
As before we denote by $\Surg^{\prm}_{j}(P) \subset \Surg_{j}(P)$ the set of all primitive surgeries of degree $j$.

Now, let $\sigma \in \Surg^{\prm}_{j}(P)$ and let $\alpha \in \Surg_{d-j-1}(P)$ be such that the spheres 
$\sigma(S^{j}\times\{0\})$ and $\alpha(S^{d-j-1}\times\{0\})$
intersect transversally at one point in $P$.
Consider the dual surgery data $\bar{\alpha}$ in $P(\alpha)$ and its trace $T(\bar{\alpha}): P(\alpha) \rightsquigarrow P$.
The lemma below follows from the handle cancelation technique used in the proof of the \textit{H-cobordism theorem} \cite[Theorem 5.4]{M 65}. 
\begin{lemma} \label{lemma: composition cylinder}
The underlying manifold of the composite cobordism 
$$
T(\sigma)\circ T(\bar{\alpha}): P(\alpha) \rightsquigarrow P(\sigma)
$$
is diffeomorphic to the cylinder, $[0,1]\times P(\alpha)$, and thus the morphism $T(\sigma)\circ T(\bar{\alpha})$ lies in $\mathcal{V}^{k-1}$. 
\end{lemma}

Our first result is the following:
\begin{proposition} \label{proposition: stability for primitive handles}
Let $k < d/2$
and  $P \in \Ob\Cob^{k-1}_{\theta, d}$. 
Let $\sigma \in \Surg^{\pr}_{d-k}(P)$ be a primitive surgery. 
Then the trace $T(\sigma): P \rightsquigarrow P(\sigma)$ lies in $\mathcal{V}^{k-1}$.
\end{proposition}
\begin{proof}
Let $\alpha \in \Surg_{k-1}(P)$ be such that $\alpha(S^{k-1}\times\{0\})$ and $\sigma(S^{d-k}\times\{0\})$ intersect transversally in $P$ at exactly one point.
Let $\bar{\alpha} \in \Surg_{d-k-1}(P(\alpha))$ be dual to $\alpha$. 
By the primitivity of $\alpha$ it follows that $\bar{\alpha}$ is a trivial surgery.
By Corollary \ref{corollary: trivial d-l handle} the trace $T(\bar{\alpha}): P(\alpha) \rightsquigarrow P$ lies in $\mathcal{V}^{k-1}$.
We then consider the composite cobordism, 
$
\xymatrix{
P(\alpha) \ar[rr]^{T(\bar{\alpha})} && P \ar[rr]^{T(\sigma)} && P(\sigma).
}
$
By Lemma \ref{lemma: composition cylinder} its underlying manifold is diffeomorphic to a cylinder, thus it lies in $\mathcal{V}^{k-1}$. 
By the $2$-out-of-$3$ property it follows that $T(\sigma)$ lies in $\mathcal{V}^{k-1}$ as well. 
This concludes the proof of the proposition.
\end{proof}

The following proposition is proven using the exact same argument as the previous one. 
\begin{proposition} \label{proposition: bootstrapping proposition}
Let $l < k < d/2$.
Suppose that it has been proven that the trace of every element 
of $\Surg^{\tr}_{d-l-1}(P)$ is contained in $\mathcal{V}^{k-1}$. 
Then it follows that the trace of every element of $\Surg^{\pr}_{d-l}(P)$ is contained in $\mathcal{V}^{k-1}$. 
\end{proposition}

\subsection{Proof of Proposition \ref{proposition: arbitrary d-l surgeries}} \label{subsection: proof of higher handles prop}
We will need to use the semi-simplicial space constructed in Definition \ref{defn: semi simplicial resolution}. 
For what follows, let $P \in \Ob\Cob^{k-1}_{\theta, d}$, and let $\sigma \in \Surg_{d-l}(P)$ be with $d-l < d-1$.
Choose an embedding 
\begin{equation} \label{equation: chi embedding}
\chi: \partial D^{l-1}\times(1,\infty)\times\R^{d-l} \longrightarrow P
\end{equation}
and a one parameter family of $\theta$-structures $\hat{\ell}^{\std}_{t}$
as in Definition \ref{defn: semi-simplicial space 1} (with respect to $P$), 
with the following additional properties:
\begin{enumerate} \itemsep.2cm
\item[(a)] $\chi(\partial D^{l-1}\times(1,\infty)\times\R^{d-l})\cap\sigma(\partial D^{d-l+1}\times D^{l-1}) = \emptyset$;
\item[(b)] the embedding $\chi|_{\partial D^{l-1}}: \partial D^{l-1} \longrightarrow P\setminus\sigma(\partial D^{d-l+1}\times\{0\})$ is isotopic to 
$$\sigma|_{\{s\}\times\partial D^{l-1}}: \{s\}\times\partial D^{l-1} \longrightarrow P\setminus\sigma(\partial D^{d-l+1}\times\{0\})$$
for some $s \in \partial D^{d-l+1}$.
\end{enumerate}
For $p \in \Z_{\geq 0}$, let $\chi(p) = (\chi_{0}, \dots, \chi_{p})$ be the list of mutually disjoint surgery data defined exactly as in Construction \ref{Construction: auxiliary displaced trace}, using the embedding $\chi$.
Notice that in order for such an embedding to exist, it is necessary and sufficient that $d-l < d-1$. 
Consider the transported surgery data $\mathcal{R}_{\chi(p)}(\sigma)$ in $P(\chi(p))$. 
The main lemma that we will need is the following:
\begin{lemma} \label{lemma: new primitive embedding}
Let $\chi$ and $\sigma$ be chosen so that they satisfy conditions (a) and (b) as above. 
Let $p \in \Z_{\geq 0}$.
Then the surgery $\mathcal{R}_{\chi(p)}(\sigma) \in \Surg_{d-l}(P(\chi(p)))$ is a primitive surgery.  
\end{lemma}
\begin{proof}
Let us first assume that $p = 0$.
Pick a point $s \in \partial D^{d-l+1}$. 
Let $\widehat{\chi}_{0}: D^{l-1} \longrightarrow P$ be an embedding with the following properties:
\begin{itemize} \itemsep.2cm
\item $\widehat{\chi}_{0}(D^{l-1}) \subset P\setminus\Int(\chi_{0}(S^{l-2}\times D^{d-l+1}))$,
\item $\widehat{\chi}_{0}^{-1}(\chi_{0}(S^{l-2}\times\{s\})) \; = \; S^{l-2}$,
\item $\widehat{\chi}_{0}(D^{l-1})$ intersects $\sigma(S^{d-l}\times\{0\})$ transversally in $P\setminus\Int(\chi_{0}(S^{l-2}\times D^{d-l+1}))$ at exactly one point. 
\end{itemize}
It follows from condition (b) that such an embedding $\widehat{\chi}_{0}$ can always be found;
we may choose $\widehat{\chi}_{0}$ to be the inclusion of a fibre of the normal disk-bundle of $\sigma(S^{d-l}\times\{0\})$. 
Let $D_{-}^{l-1} \subset P$ denote $\widehat{\chi}_{0}(D^{l-1})$.
Consider the surgered manifold 
$$
P(\chi_{0}) \; = \; P\setminus\Int(\chi_{0}(S^{l-2}\times D^{d-l+1}))\bigcup D^{l-1}\times S^{d-l}.
$$
Notice that by how $D^{l-1}_{-}$ was chosen, $D^{l-1}_{-}$ is contained in $P(\chi_{0})$.
Let $D^{l-1}_{+} \subset P(\chi_{0})$ be the disk $D^{l-1}\times\{s\} \subset D^{l-1}\times S^{d-l}$ coming from the right-hand side of the above union.
The union 
$$D_{-}^{l-1}\cup D^{l-1}_{+} \subset P(\chi_{0})$$ 
is an embedded $(l-1)$-sphere which is easily seen to have trivial normal bundle. 
Let use denote by $\beta: S^{l-1} \longrightarrow P(\chi_{0})$ the inclusion of this sphere. 
By construction, $\beta(S^{l-1})$ intersects the sphere $\sigma(S^{d-l}\times\{0\}) \subset P(\chi_{0})$ transversally at a single point. 
This proves that $\mathcal{R}_{\chi_{0}}(\sigma)$ is primitive. 
Now suppose that $p > 0$. 
If we perform surgery on the first embedding $\chi_{0}$, by the above argument $\mathcal{R}_{\chi_{0}}(\sigma)$ is primitive in $P(\chi_{0})$. 
We then observe that $\beta(S^{l-1}) \subset P(\chi_{0})$ is disjoint from the embedding that corresponds to the surgery data $\mathcal{R}_{\chi_{0}}(\chi_{1})$, and thus $\beta(S^{l-1})$ determines an embedded sphere with trivial normal bundle (with the same name), $\beta(S^{l-1}) \subset P(\chi_{1}, \chi_{2})$.
In $P(\chi_{1}, \chi_{2})$, the sphere $\beta(S^{l-1})$ still intersects $\sigma(S^{d-l}\times\{0\})$ transversally at a single point, and thus $\mathcal{R}_{\chi_{1}, \chi_{2}}(\sigma)$ is primitive. 
The proof of the lemma then follows by proceeding inductively. 
\end{proof}

With the above lemma established we may now prove Proposition \ref{proposition: arbitrary d-l surgeries}. 
\begin{proof}[Proof of Proposition \ref{proposition: arbitrary d-l surgeries}]
We prove the theorem by induction on the value $j = d-l \geq d-k$. 
The induction step is based on the following claim:
\begin{claim} \label{claim: trivial d-l-1 implies arbitary d-l}
Suppose that it has been proven that the trace of every element of 
$\Surg^{\tr}_{d-l-1}(P)$ is contained in $\mathcal{V}^{k-1}$.
Then it follows that $T(\sigma) \in \mathcal{V}^{k-1}$ for all 
$\sigma \in \Surg_{d-l}(P)$.
\end{claim}
Let us explain how this claim implies Proposition \ref{proposition: arbitrary d-l surgeries}. 
By Corollary \ref{corollary: trivial d-l handle}, all trivial surgeries $\alpha \in \Surg^{\tr}_{d-k-1}(P)$ have their trace $T(\alpha)$ contained in $\mathcal{V}^{k-1}$.
By the above claim, it follows that all surgeries in $\Surg_{d-k}(P)$ (trivial or not) have trace contained in $\mathcal{V}^{k-1}$.
This establishes the base case of the induction.
We then may apply Claim \ref{claim: trivial d-l-1 implies arbitary d-l} again to find that all surgeries of degree $d-k+1$ have trace contained in $\mathcal{V}^{k-1}$. 
Continuing by induction then establishes Proposition \ref{proposition: arbitrary d-l surgeries}. 

We now proceed to prove Claim \ref{claim: trivial d-l-1 implies arbitary d-l}. 
Suppose it has been proven that the trace of every element of $\Surg^{\tr}_{d-l-1}(P)$ is contained in $\mathcal{V}^{k-1}$.
Let $\sigma \in \Surg_{d-l}(P)$ and let the embedding 
$$
\chi: \partial D^{l-1}\times(1,\infty)\times\R^{d-l} \longrightarrow P
$$
be chosen exactly as in (\ref{equation: chi embedding}).  
For each $p \in \Z_{\geq 0}$, consider the commutative diagram
$$
\xymatrix{
\displaystyle{\hocolim_{\mb{t} \in \mathcal{K}^{k}}}\mathcal{W}^{k}_{P(\chi(p)), \mb{t}}  \ar[rrrr]^{\mb{S}_{T(\mathcal{R}_{\chi(p)}(\sigma))}} \ar[d]^{\mathcal{H}(\chi)_{p}}_{\simeq} &&&& 
\displaystyle{\hocolim_{\mb{t} \in \mathcal{K}^{k}}}\mathcal{W}^{k}_{P(\chi(p), \mathcal{R}_{\chi(p)}(\sigma)), \; \mb{t}}
 \ar[d]^{\mathcal{H}(\chi)_{p}}_{\simeq} \\
\displaystyle{\hocolim_{\mb{t} \in \mathcal{K}^{k}}}\mathcal{M}^{k}_{P, \mb{t}}\left(\chi\right)_{p}  \ar[rrrr]^{\mb{T}(\sigma)_{p}} &&&& \displaystyle{\hocolim_{\mb{t} \in \mathcal{K}^{k}}}\mathcal{M}^{k}_{P(\sigma), \mb{t}}\left(\chi\right)_{p},
}
$$
where the top-horizontal map is induced by concatenation with $T(\mathcal{R}_{\chi(p)}(\sigma))$, the bottom horizontal map is induced by concatenation with $T(\sigma)$, and the vertical maps are the weak homotopy equivalences defined as in Construction \ref{Construction: auxiliary displaced trace}. 
By Lemma \ref{lemma: new primitive embedding}, $\mathcal{R}_{\chi(p)}(\sigma)$ is a primitive surgery of degree $d-l$. 
Since the trace of a trivial surgery of degree $d-l-1$ is contained in $\mathcal{V}^{k-1}$ by assumption, it follows by Proposition \ref{proposition: bootstrapping proposition} that the top vertical map in the above diagram is a weak homotopy equivalence. 
By commutativity, it then follows that the bottom horizontal map is a weak homotopy equivalence, for all $p \in \Z_{\geq 0}$. 
It then follows that $\mb{T}(\sigma)_{\bullet}$ induces a homotopy equivalence on geometric realization. 
Claim \ref{claim: trivial d-l-1 implies arbitary d-l} then follows by considering the commutative square
$$
\xymatrix{
|\displaystyle{\hocolim_{\mb{t} \in \mathcal{K}^{k}}}\mathcal{M}^{k}_{P, \mb{t}}\left(\chi\right)_{\bullet}| \ar[rrr]^{|\mb{T}(\sigma)_{\bullet}|}_{\simeq} \ar[d]^{\simeq} &&& |\displaystyle{\hocolim_{\mb{t} \in \mathcal{K}^{k}}}\mathcal{M}^{k}_{P(\sigma), \mb{t}}\left(\chi\right)_{\bullet}| \ar[d]^{\simeq} \\
\displaystyle{\hocolim_{\mb{t} \in \mathcal{K}^{k}}}\mathcal{W}^{k}_{P, \mb{t}}  \ar[rrr]^{\mb{S}_{T(\sigma)}} &&& \displaystyle{\hocolim_{\mb{t} \in \mathcal{K}^{k}}}\mathcal{W}^{k}_{P(\sigma), \mb{t}},
}
$$
whose vertical arrows are weak homotopy equivalences by Corollary \ref{corollary: homotopy colimit augmentation}. 
\end{proof}

\subsection{Index $d$-handles} \label{subsection: index d handles}
In Proposition \ref{proposition: arbitrary d-l surgeries}, it was required that the degree of the surgery $\sigma$ be strictly less then $d-1$.
We now handle the final case where $\sigma$ has $\deg(\sigma) = d-1$.
\begin{proposition} \label{proposition: d-1 surgeries}
Let $P \in \Ob\Cob^{k-1}_{\theta, d}$ be an object with strictly more than one path component. 
Let $\sigma \in \Surg_{d-1}(P)$. 
Then $T(\sigma) \in \mathcal{V}^{k-1}$.
\end{proposition}
\begin{proof}
We will prove that $\sigma$ is primitive. 
We first observe that for dimensional reasons, the submanifold $\sigma(S^{d-1}) \subset P$ is a whole path component of $P$. 
Let $S^{0} \hookrightarrow P$ be an embedding that sends one point to $\sigma(S^{d-1}) \subset P$, and the other point of $S^{0}$ to a different path component of $P$. 
Since $P$ has more than one path component, such an embedding exists.
The existence of this embedding proves that $\sigma$ is primitive.
Since it has been proven that trace of every surgery of degree $d-2$ is contained in $\mathcal{V}^{k-1}$, Proposition \ref{proposition: bootstrapping proposition} then implies that $T(\sigma) \in \mathcal{V}^{k-1}$. 
This concludes the proof of the proposition.
\end{proof}

We now deal with the general case where the object $P$ is path-connected, and hence be diffeomorphic to $S^{d-1}$. 
\begin{proposition} \label{proposition: d-1 surgeries general case}
Let $P \in \Ob\Cob^{k-1}_{\theta, d}$ be any object (that is possibly path-connected). 
Let $\sigma \in \Surg_{d-1}(P)$. 
Then $T(\sigma) \in \mathcal{V}^{k-1}$.
\end{proposition}
The proof of the above proposition will require a new construction. 
\begin{defn} \label{defn: disk resolution}
Let $P \in \Ob\Cob^{k-1}_{\theta, d}$.
Choose once and for all a standard $\theta$-structure $\hat{\ell}_{D^{d}}$ on the disk $D^{d}$. 
Fix an element 
$z := (M, (V, \sigma), e) \in \mathcal{W}^{k}_{P, \mb{t}}$. 
We define $\mb{Y}_{0}(z)$ to be the space of pairs $(\phi, \hat{L})$ where: 
\begin{itemize} \itemsep.2cm
\item $\phi: D^{d} \longrightarrow W$ is an embedding with image disjoint from the image of $e$;
\item $\hat{L}(t)$, $t \in [0,1]$, is a one-parameter family of $\theta$-structures on $D^{d}$ with $\hat{L}(0) = \phi^{*}\hat{\ell}_{W}$, and $\hat{L}(1) = \hat{\ell}_{D^{d}}$. 
\end{itemize}
For $p \in \Z_{\geq 0}$, the space $\mb{Y}_{p}(z)$ consists of $(p+1)$-tuples $((\phi_{0}, \hat{L}_{0}), \dots, (\phi_{p}, \hat{L}_{p}))$, that satisfy:
$$
\phi_{i}(D^{d})\cap\phi_{j}(D^{d}) = \emptyset \quad \text{whenever $i \neq j$.}
$$
The assignment $[p] \mapsto \mb{Y}_{p}(z)$ defines a semi-simplicial space.
For each $p \in \Z_{\geq 0}$, the space $\mb{D}^{k}_{P, \mb{t}, p}$ is defined to consist of tuples, 
$(z; ((\phi_{0}, \hat{L}_{0}), \dots, (\phi_{p}, \hat{L}_{p}))),$
where $z \in \mathcal{W}^{k}_{P, \mb{t}}$, and $((\phi_{0}, \hat{L}_{0}), \dots, (\phi_{p}, \hat{L}_{p})) \in \mb{Y}_{p}(z)$.
The assignment $[p] \mapsto \mb{D}^{k}_{P, \mb{t}, p}$ defines a semi-simplicial space, $\mb{D}^{k}_{P, \mb{t}, \bullet}$.
The forgetful maps 
$$\mb{D}^{k}_{P, \mb{t}, p} \longrightarrow \mathcal{W}^{k}_{P, \mb{t}}, \quad (z; (\phi_{0}, \hat{L}_{0}), \dots, (\phi_{p}, \hat{L}_{p})) \mapsto z$$
yield an augmented semi-simplicial space,
$\mb{D}^{k}_{P, \mb{t}, \bullet} \longrightarrow \mb{D}^{k}_{P, \mb{t}, -1}$
with $\mb{D}^{k}_{P, \mb{t}, -1} = \mathcal{W}^{k}_{P, \mb{t}}$.
It follows from the definition that $\mb{t} \mapsto \mb{D}^{k}_{P, \mb{t}, \bullet}$ defines a contravariant functor on $\mathcal{K}^{k}$ valued in augmented semi-simplicial spaces.
By taking the level-wise homotopy colimits over $\mathcal{K}^{k}$, we obtain the augmented semi-simplicial space, 
$
\displaystyle{\hocolim_{\mb{t} \in \mathcal{K}^{k}}}\mb{D}^{k}_{P, \mb{t}, \bullet} \longrightarrow \displaystyle{\hocolim_{\mb{t} \in \mathcal{K}^{k}}}\mb{D}^{k}_{P, \mb{t}, -1}.
$
\end{defn}

We have the following proposition.
\begin{proposition} \label{proposition: contractibility for disk resolution}
For all $P \in \Ob\Cob^{k-1}_{\theta, d}$, the map induced by augmentation,
$$
|\hocolim_{\mb{t} \in \mathcal{K}^{k}}\mb{D}^{k}_{P, \mb{t}, \bullet}| \longrightarrow \hocolim_{\mb{t} \in \mathcal{K}^{k}},\mb{D}^{k}_{P, \mb{t}, -1},
$$
 is a weak homotopy equivalence. 
\end{proposition}
\begin{proof}
For each $\mb{t} \in \mathcal{K}^{k}$, the augmentation map $|\mb{D}^{k}_{P, \mb{t}, \bullet}| \longrightarrow \mb{D}^{k}_{P, \mb{t}, -1}$ is a weak homotopy equivalence. 
This follows by the exact same argument employed in the proof of Proposition \ref{proposition: augmentation is weak equivalence} (in this case it even easier because there is no need to use the \textit{Smale-Hirsch theorem}).
From the homotopy invariance of homotopy colimits, it then follows that the induced map 
$$
\hocolim_{\mb{t} \in \mathcal{K}^{k}}|\mb{D}^{k}_{P, \mb{t}, \bullet}| \longrightarrow \hocolim_{\mb{t} \in \mathcal{K}^{k}}\mb{D}^{k}_{P, \mb{t}, -1}
$$
is a weak homotopy equivalence. 
The proof then follows from the homotopy equivalence 
$$
\hocolim_{\mb{t} \in \mathcal{K}^{k}}|\mb{D}^{k}_{P, \mb{t}, \bullet}| \simeq |\hocolim_{\mb{t} \in \mathcal{K}^{k}}\mb{D}^{k}_{P, \mb{t}, \bullet}|,
$$
see also the proof of Corollary \ref{corollary: homotopy colimit augmentation}.
\end{proof}

We now prove Proposition \ref{proposition: d-1 surgeries general case}.
\begin{proof}[Proof of Proposition \ref{proposition: d-1 surgeries general case}]
Fix $P \in \Ob\Cob_{\theta, d}$. 
Let $\sigma \in \Surg_{d-1}(P)$ and consider the trace $T(\sigma): P \rightsquigarrow P(\sigma)$. 
If $P$ has more than one path component, then $T(\sigma)$ belongs to $\mathcal{V}^{k-1}$ by Proposition \ref{proposition: d-1 surgeries}, and so there is nothing to show. 
So, assume that $P$ is path connected. 
Consider the augmented semi-simplicial space $\mb{D}^{k}_{P, \mb{t}, \bullet} \longrightarrow \mb{D}^{k}_{P, \mb{t}, -1}$. 
For each integer $p \in \Z_{\geq 0}$, fix an object $S_{p} \in \Ob\Cob_{\theta, d}$, with the properties:
\begin{enumerate} \itemsep.2cm
\item[(i)] there is a diffeomorphism $S_{p} \cong (S^{d-1})^{\sqcup(p+1)}$;
\item[(ii)] $S_{\mb{t}}$ is disjoint from $P$ as a submanifold of $\R^{\infty}$;
\item[(ii)] on each component of $S_{p}$, the $\theta$-structure $\hat{\ell}_{S_{p}}$ agrees with $\hat{\ell}_{D^{d}}|_{S^{d-1}}$, via the diffeomorphism from part (i).
\end{enumerate}
By condition (i), it follows that the disjoint union $P\sqcup S_{p}$ is an object of $\Cob_{\theta, d}$.
By the same argument used in the proof of Lemma \ref{lemma: levelwise identification}, for each $p \in \Z_{\geq 0}$ there is a weak homotopy equivalence, 
\begin{equation} \label{equation: equivalence to space of manifolds}
\mathcal{H}_{\mb{t}, p}: \mathcal{W}^{k}_{P\sqcup S_{p}, \mb{t}}  \stackrel{\simeq} \longrightarrow \mb{D}^{k}_{P, \mb{t}, p}
\end{equation}
(see also the constructions in Section \ref{subsection: resolving composition}).
For each $p \in \Z_{\geq 0}$ we may consider the commutative diagram,
\begin{equation} \label{equation: trace equivalence on p-th level}
\xymatrix{
\displaystyle{\hocolim_{\mb{t} \in \mathcal{K}^{k}}}\mathcal{W}^{k}_{P\sqcup S_{p}, \mb{t}} \ar[d]^{\simeq}_{\mathcal{H}_{p, \mb{t}}} \ar[rr]^{\--\cup T(\sigma)} && \displaystyle{\hocolim_{\mb{t} \in \mathcal{K}^{k}}}\mathcal{W}^{k}_{P(\sigma)\sqcup S_{p}, \mb{t}} \ar[d]_{\mathcal{H}_{\mb{t}, p}}^{\simeq} \\
\displaystyle{\hocolim_{\mb{t} \in \mathcal{K}^{k}}}\mb{D}^{k}_{P, \mb{t}, p} \ar[rr]^{\--\cup T(\sigma)} && \displaystyle{\hocolim_{\mb{t} \in \mathcal{K}^{k}}}\mb{D}^{k}_{P(\sigma), \mb{t}, p}.
}
\end{equation}
Since the object $P\sqcup S_{p}$ is not path-connected, it follows from Proposition \ref{proposition: d-1 surgeries} that the top vertical map in (\ref{equation: trace equivalence on p-th level}) is a weak homotopy equivalence. 
By commutativity of the diagram it follows that the bottom-horizontal map is a weak homotopy equivalence as well.
By geometrically realizing it follows that $\--\cup T(\sigma)$ induces the weak homotopy equivalence,
$$
|\displaystyle{\hocolim_{\mb{t} \in \mathcal{K}^{k}}}\mb{D}^{k}_{P, \mb{t}, \bullet}| \stackrel{\simeq} \longrightarrow |\displaystyle{\hocolim_{\mb{t} \in \mathcal{K}^{k}}}\mb{D}^{k}_{P(\sigma), \mb{t}, \bullet}|.
$$
The proposition then follows by commutativity of the diagram,
$$
\xymatrix{
|\displaystyle{\hocolim_{\mb{t} \in \mathcal{K}^{k}}}\mb{D}^{k}_{P, \mb{t}, \bullet}| \ar[d]_{\simeq} \ar[rr]^{\simeq} &&  |\displaystyle{\hocolim_{\mb{t} \in \mathcal{K}^{k}}}\mb{D}^{k}_{P(\sigma), \mb{t}, \bullet}| \ar[d]_{\simeq} \\
\displaystyle{\hocolim_{\mb{t} \in \mathcal{K}^{k}}}\mb{D}^{k}_{P, \mb{t}, -1} \ar[d]_{=}  \ar[rr] && \displaystyle{\hocolim_{\mb{t} \in \mathcal{K}^{k}}}\mb{D}^{k}_{P(\sigma), \mb{t}, -1} \ar[d]_{=} \\
 \displaystyle{\hocolim_{\mb{t} \in \mathcal{K}^{k}}}\mathcal{W}^{k}_{P, \mb{t}} \ar[rr] && \displaystyle{\hocolim_{\mb{t} \in \mathcal{K}^{k}}}\mathcal{W}^{k}_{P(\sigma), \mb{t}}.
}
$$
\end{proof} 

By collecting our results from the last three sections, we observe that we have proven Theorem \ref{theorem: morphism induce homotopy equivalence}. 
Indeed, the theorem follows by combining Corollary \ref{corollary: trivial d-l handle}, Corollary \ref{corollary: primitive surgery of degree k-1}, Proposition \ref{proposition: arbitrary d-l surgeries}, and Proposition \ref{proposition: d-1 surgeries general case}. 

\section{Stable Moduli Spaces of Even Dimensional Manifolds} \label{section: stable moduli spaces}
Recall that Theorem \ref{theorem: main theorem} required the condition that $k < d/2$. 
In this section we show how to recover a modified version of Theorem \ref{theorem: main theorem} in the case the $d = 2n$ and $k = n$. 
As described in the introduction, this will let us recover the proof of the \textit{Madsen-Weiss} theorem from \cite{MW 07} and the theorem of Galatius and Randal-Williams from \cite{GRW 14}. 
To state this modified version of Theorem \ref{theorem: main theorem} we will need a new construction. 

In this section we assume that $d = 2n$ for a positive integer $n$.
Fix a tangential structure $\theta: B \longrightarrow B(d+1)$. 
Let $P \in \Ob\Cob_{\theta, d}$.
We consider the functor $\mb{t} \mapsto \mathcal{W}^{n-1, c}_{P, \mb{t}}$  from Definition \ref{defn: intermediate high connected sheaf}. 
Recall from Subsection \ref{subsection: morphism H(P)} the self-cobordism 
$
H_{n, n}(P): P \rightsquigarrow P
$
and the map
$$
\mb{S}_{H_{n, n}(P)}: \mathcal{W}^{n-1, c}_{P, \mb{t}} \; \longrightarrow \; \mathcal{W}^{n-1, c}_{P, \mb{t}}
$$
that it induces. 
By the same argument employed in the proof of Proposition \ref{proposition: H-k induces iso}, it follows that upon taking homotopy colimits, the map $\mb{S}_{H_{n, n}(P)}$ induces a weak homotopy equivalence 
\begin{equation} \label{equation: homotopy equivalence on colimit}
\hocolim_{\mb{t} \in \mathcal{K}^{n-1}}\mathcal{W}^{n-1, c}_{P, \mb{t}} \; \stackrel{\simeq} \longrightarrow \; \hocolim_{\mb{t} \in \mathcal{K}^{n-1}}\mathcal{W}^{n-1, c}_{P, \mb{t}}.
\end{equation}

\begin{defn} \label{defn: stabilization}
Using the map $\mb{S}_{H_{n, n}(P)}$ above, for each $k \leq n$ and $\mb{t} \in \mathcal{K}^{k}$ we define 
$$
\mathcal{W}^{k, \stb}_{P, \mb{t}} := \hocolim\left[\mathcal{W}^{k}_{P, \mb{t}} \rightarrow \mathcal{W}^{k}_{P, \mb{t}} \rightarrow \cdots \right]
$$
where the direct system is given by iterating the map $\mb{S}_{H_{n, n}(P)}$.
The space $\mathcal{W}^{k, c, \stb}_{P, \mb{t}}$ is defined similarly.
\end{defn}
We will mainly be interested in the space $\mathcal{W}^{n-1, c, \stb}_{P, \mb{t}}$. 
\begin{proposition} \label{proposition: inclusion into the colimit}
The inclusion 
$$\displaystyle{\hocolim_{\mb{t} \in \mathcal{K}^{n-1}}}\mathcal{W}^{n-1, c}_{P, \mb{t}} \longrightarrow \displaystyle{\hocolim_{\mb{t} \in \mathcal{K}^{n-1}}}\mathcal{W}^{n-1, c, \stb}_{P, \mb{t}}$$
is a weak homotopy equivalence. 
\end{proposition}
\begin{proof}
The proposition follows from (\ref{equation: homotopy equivalence on colimit}) combined with the weak homotopy equivalence 
$$
\begin{aligned}
\hocolim_{\mb{t} \in \mathcal{K}^{n-1}}\left(\hocolim\left[\mathcal{W}^{n-1, c}_{P, \mb{t}} \rightarrow \mathcal{W}^{n-1, c}_{P, \mb{t}} \rightarrow \cdots \right]\right) \; \;  \simeq \; \; 
\hocolim\left[\hocolim_{\mb{t} \in \mathcal{K}^{n-1}}\mathcal{W}^{n-1, c}_{P, \mb{t}} \rightarrow \hocolim_{\mb{t} \in \mathcal{K}^{n-1}}\mathcal{W}^{n-1, c}_{P, \mb{t}} \rightarrow \cdots \right].
\end{aligned}
$$
\end{proof}

We now analyze the fibres of the map, 
$
\mathcal{W}^{n-1, c, \stb}_{P, \mb{t}} \longrightarrow \mathcal{W}^{\{n, n+1\}}_{\locc, \mb{t}}.
$

\begin{defn} \label{defn: fibre sheaf degree n}
For each $\mb{s} \in \mathcal{K}^{\{n, n+1\}}$, let us fix once and for all an element 
$$((V_{\mb{s}}, \sigma_{\mb{s}}), e_{\mb{s}}) \in \mathcal{W}^{\{n, n+1\}}_{\locc, \mb{s}}.$$ 
As before, we denote the unique element in $\mathcal{W}^{\{n, n+1\}}_{\locc, \emptyset}$ by $\emptyset$.
We let $\widehat{\mathcal{W}}^{n-1, c, \stb}_{P, \mb{s}}$ denote the fibre of the 
map
$
\mathcal{W}^{n-1, c, \stb}_{P, \mb{s}} \longrightarrow \mathcal{W}^{\{n, n+1\}}_{\locc, \mb{s}}
$
over the element $((V_{\mb{s}}, \sigma_{\mb{s}}), e_{\mb{s}}) \in \mathcal{W}^{\{n, n+1\}}_{\locc, \mb{s}}$.
\end{defn}

Proceeding as in Section \ref{section: Cobordism Categories and Handle Attachments}, for each $\mb{t} \in \mathcal{K}^{\{n, n+1\}}$
we fix once and for all a closed submanifold 
$$
S_{\mb{t}} \subset \R^{\infty-1}
$$
disjoint from the manifold $P$, and diffeomorphic to the product $\mb{s}\times S^{n-1}\times S^{n}$. 
Choose a $\theta$-structure $\hat{\ell}_{S_{\mb{t}}}$ on $S_{\mb{t}}$ that admits an extension to a $\theta$-structure on $\mb{t}\times D^{n}\times S^{n}$. 
The manifold $P\sqcup S_{\mb{t}}$ equipped with the structure $\hat{\ell}_{P}\sqcup\hat{\ell}_{S_{\mb{t}}}$ defines an object in the cobordism category $\Cob_{\theta, 2n}$. 

By the same augment used in Proposition \ref{proposition: fibre identification} we obtain:
\begin{proposition} \label{proposition: cut out the embedding homeo}
For all $P \in \Ob\Cob_{\theta, 2n}$ and $\mb{t} \in \mathcal{K}^{\{n, n+1\}}$ there is a homeomorphism 
$$
\mathcal{W}^{n, \stb}_{P\sqcup S_{\mb{t}}, \emptyset} \stackrel{\cong} \longrightarrow \widehat{\mathcal{W}}^{n-1, c, \stb}_{P, \mb{t}}.
$$
\end{proposition}
Let us now analyze the spaces $\mathcal{W}^{n}_{P, \emptyset}$ more closely. 
For any object $P \in \Ob\Cob_{\theta, 2n}$, the space $\mathcal{W}^{n}_{P, \emptyset}$ is precisely the space of all $M \in \mathcal{N}_{\theta, P}$ with the property that $\ell_{M}: M \longrightarrow B$ is $n$-connected. 
It follows that for all $P$ there is a weak homotopy equivalence 
$$
\mathcal{W}^{n}_{P, \emptyset} \simeq \coprod_{[M]}\textstyle{\BDiff_{\theta, n}}(M, P)
$$
where the union ranges over all diffeomorphism classes of compact manifolds $M$ equipped with an identification $\partial M = P$.
The space $\textstyle{\BDiff_{\theta, n}}(M, P)$ is the homotopy quotient
$$
\textstyle{\Bun_{n}}(TM\oplus\epsilon^{1}, \theta^{*}\gamma^{d+1}; \hat{\ell}_{P})//\Diff(M, P)
$$
where $\textstyle{\Bun_{n}}(TM\oplus\epsilon^{1}, \theta^{*}\gamma^{d+1}; \hat{\ell}_{P})$ is the space of $\theta$-structures $\hat{\ell}_{M}$ on $M$, that agree with $\hat{\ell}_{P}$ on the boundary, 
such that the underlying map $\ell_{M}: M \longrightarrow B$ is $n$-connected. 
It follows that the limiting space $\mathcal{W}^{n, \stb}_{P, \emptyset}$ agrees with the \textit{stable moduli space of $\theta$-manifolds} studied in \cite{GRW 14} and in \cite{GRW 16}.

As was observed in Section \ref{section: Cobordism Categories and Handle Attachments}, any morphism $W: P \rightsquigarrow Q$ in $\Cob^{n-1}_{\theta, d}$ induces a map
$$
\mb{S}_{W}: \mathcal{W}^{n, \stb}_{P, \emptyset} \longrightarrow \mathcal{W}^{n, \stb}_{Q, \emptyset}, \quad M \mapsto M\cup_{P}W.
$$
The theorem stated below is a translation of the homological stability theorem of Galatius and Randal-Williams proven in \cite{GRW 16}.
\begin{theorem} \label{theorem: stable stability}
For any morphism $W: P \rightsquigarrow Q$ in $\Cob^{n-1}_{\theta, d}$, the map 
$
\mb{S}_{W}: \mathcal{W}^{n, \stb}_{P, \emptyset} \longrightarrow \mathcal{W}^{n, \stb}_{Q, \emptyset}
$
is an Abelian homological equivalence.
\end{theorem}

\begin{corollary} \label{theorem: stable stability}
For any morphism $(j, \varepsilon): \mb{t} \longrightarrow \mb{s}$ in $\mathcal{K}^{n-1}$, the induced map 
$$
(j, \varepsilon)^{*}: \widehat{\mathcal{W}}^{n-1, c, \stb}_{P, \mb{s}} \longrightarrow \widehat{\mathcal{W}}^{n-1, c, \stb}_{P, \mb{t}}
$$
is an Abelian homological equivalence. 
\end{corollary}
\begin{proof} 
This is proven in the same way as Theorem \ref{proposition: homotopy equivalence of transition maps}, using Theorem \ref{theorem: stable stability} instead of Theorem \ref{theorem: morphism induce homotopy equivalence}. 
The map $(j, \varepsilon):^{*}: \widehat{\mathcal{W}}^{n-1, c, \stb}_{P, \mb{s}} \longrightarrow \widehat{\mathcal{W}}^{n-1, c, \stb}_{P, \mb{t}}$ fits into a commutative diagram 
$$
\xymatrix{
\widehat{\mathcal{W}}^{n-1, c, \stb}_{P, \mb{s}}  \ar[rr]^{(j, \varepsilon)^{*}} && \widehat{\mathcal{W}}^{n-1, c,  \stb}_{P, \mb{t}} \ar[d]_{\simeq} \\
\mathcal{W}^{n, \stb}_{P\sqcup S_{\mb{s}}, \emptyset} \ar[u]^{\simeq} \ar[rr]^{\mb{S}_{W}} && \mathcal{W}^{n, \stb}_{P\sqcup S_{\mb{t}}, \emptyset}
}
$$
where $W: P\sqcup S_{\mb{s}} \rightsquigarrow P\sqcup S_{\mb{t}}$ is some morphism in $\Cob_{\theta, 2n}^{n-1}$, and the vertical maps are the homotopy equivalences from Proposition \ref{proposition: cut out the embedding homeo}.
By Theorem \ref{theorem: stable stability}, $\mb{S}_{W}$ is an Abelian homological equivalence and so it follows by commutativity that $(j, \varepsilon)^{*}$ is as well.
\end{proof}

Corollary \ref{corollary: homological corollary for k = n} from the introduction concerned an acyclic map. 
We will need to use a homological version of Theorem \ref{proposition: homotopy equivalence of transition maps} for the case $d = 2n$ and $k = n$.
We need to introduce a new definition and a preliminary result. 
\begin{defn} \label{defn: abelian homology fibration}
A map $f: X \longrightarrow Y$ between topological spaces is said to be an \textit{acyclic homology fibration} if the for all $y \in Y$, the inclusion, 
$f^{-1}(y) \hookrightarrow \textstyle{\hofibre_{y}}(f)$,
is an acyclic map. 
\end{defn}

The proposition below is a generalization of Proposition \ref{proposition: homotopy colimit fibre sequence}. 
Its proof follows from \cite[Proposition 4.4]{MP 15}.
\begin{proposition} \label{proposition: homotopy colimit homology fibre sequence}
Let $\mathcal{C}$ be a small category and let $u: \mathcal{G}_{1} \longrightarrow \mathcal{G}_{2}$ be a natural transformation between functors 
from $\mathcal{C}$ to the category of topological spaces. 
Suppose that for each morphism $f: a \longrightarrow b$ in $\mathcal{C}$, the map 
$$
f_{*}: \hofibre\left(u_{a}: \mathcal{G}_{1}(a) \rightarrow \mathcal{G}_{2}(a)\right) \; \longrightarrow \; \hofibre\left(u_{b}: \mathcal{G}_{1}(b) \rightarrow \mathcal{G}_{2}(b)\right)
$$
is an Abelian homology equivalence. 
Then for any object $c \in  \Ob\mathcal{C}$, the inclusion 
$$
\hofibre\left(u_{c}: \mathcal{G}_{1}(c) \rightarrow \mathcal{G}_{2}(c)\right) \; \hookrightarrow \; \hofibre\left(\hocolim\mathcal{G}_{1} \rightarrow \hocolim\mathcal{G}_{2}\right)
$$
is an Abelian homology equivalence.
\end{proposition}
\begin{proof}
Let the simplicial space $\mb{X}_{\bullet}$ denote the nerve of the \textit{transport category} $\mathcal{G}_{1}\wr\mathcal{C}$, and let $\mb{Y}_{\bullet}$ denote the nerve of $\mathcal{G}_{2}\wr\mathcal{C}$.
The homotopy colimits of $\mathcal{G}_{1}$ and $\mathcal{G}_{2}$ are obtained by the taking the geometric realizations of the these simplicial spaces $\mb{X}_{\bullet}$ and $\mb{Y}_{\bullet}$ respectively. 
The natural transformation induces a simplicial map $u_{\bullet}: \mb{X}_{\bullet} \longrightarrow \mb{Y}_{\bullet}$.
The condition in the statement of the proposition implies that for all $p \in \Z_{\geq 0}$ and $0 \leq i \leq p$, the face map $d_{i}$ induces an Abelian homology equivalence, 
$$
\hofibre(u_{p}: \mb{X}_{p} \rightarrow \mb{Y}_{p}) \stackrel{\simeq} \longrightarrow \hofibre(u_{p-1}: \mb{X}_{p-1} \rightarrow \mb{Y}_{p-1})
$$
(a similar statement holds true for the degeneracy maps).
It then follows from \cite[Proposition 4.4]{MP 15} that the inclusion, 
$\hofibre(u_{0}: \mb{X}_{0} \rightarrow \mb{Y}_{0}) \hookrightarrow \hofibre(|u_{\bullet}|: |\mb{X}_{\bullet}| \rightarrow |\mb{Y}_{\bullet}|),$
is an Abelian homology equivalence. 
This completes the proof of the proposition.  
\end{proof}

The next corollary follows by combining Theorem \ref{theorem: stable stability} with Proposition\ref{proposition: homotopy colimit homology fibre sequence}.
By applying Proposition \ref{proposition: homotopy colimit homology fibre sequence} to the above corollary we obtain the following result. 
\begin{corollary} \label{corollary: homology fibration}
Let $d = 2n$. 
The localization map 
$$
\hocolim_{\mb{s} \in \mathcal{K}^{n-1}}\mathcal{W}^{n-1, c, \stb}_{P, \mb{s}} \longrightarrow \hocolim_{\mb{s} \in \mathcal{K}^{n-1}}\mathcal{W}^{\{n, n+1\}}_{\locc, \mb{s}}
$$
is an acyclic homology fibration with fibre over $\emptyset$ given by the space, 
$\mathcal{W}^{n, \stb}_{P, \emptyset}.$
\end{corollary}
\begin{proof}
By Proposition \ref{proposition: homotopy colimit homology fibre sequence} it follows that the inclusion 
$$
\mathcal{W}^{n, \stb}_{P, \emptyset} \hookrightarrow \hofibre\left(\hocolim_{\mb{s} \in \mathcal{K}^{n-1}}\mathcal{W}^{n-1, c, \stb}_{P, \mb{s}} \longrightarrow \hocolim_{\mb{s} \in \mathcal{K}^{n-1}}\mathcal{W}^{\{n, n+1\}}_{\locc, \mb{s}}\right)
$$
is an abelian homology equivalence. 
Both of the spaces $\displaystyle{\hocolim_{\mb{s} \in \mathcal{K}^{n-1}}}\mathcal{W}^{n-1, c, \stb}_{P, \mb{s}}$ and $\displaystyle{\hocolim_{\mb{s} \in \mathcal{K}^{n-1}}}\mathcal{W}^{\{n, n+1\}}_{\locc, \mb{s}}$ are infinite loopspaces and thus have Abelian fundamental group. 
As a result, the above homotopy fibre has Abelian fundamental group as well. 
It follows from this fact that the above Abelian homology equivalence is an Acyclic map. 
\end{proof}

Combining the above corollary with Theorem \ref{theorem: homotopy colimit decomposition}, we obtain Theorem \ref{theorem: homology fibration after stabilization} stated in the introduction. 

Finally we show how to use the above results to recover the theorem of Galatius and Randal-Williams of \cite{GRW 14}.
Using the commutative diagram
$$
\xymatrix{
\mathcal{W}^{n, \stb}_{P, \emptyset} \ar@{^{(}->}[r] \ar[d] & \displaystyle{\hocolim_{\mb{t} \in \mathcal{K}^{n-1}}}\mathcal{W}^{n-1, c, \stb}_{P, \mb{t}} \ar[rr] \ar[d]_{\simeq} && \displaystyle{\hocolim_{\mb{s} \in \mathcal{K}^{n-1}}}\mathcal{W}^{\{n, n+1\}}_{\locc, \mb{s}} \ar[d]_{\simeq} \\
\Omega^{\infty-1}\mb{hW}^{n}_{\theta} \ar[r] & \Omega^{\infty-1}\mb{hW}^{n-1}_{\theta} \ar[rr] && \Omega^{\infty-1}\mb{hW}^{\{n, d-n+1\}}_{\theta, \loc},
}
$$
By Corollary \ref{corollary: homology fibration} it follows that the right-vertical map 
\begin{equation} \label{equation: homological equivalence infinite loop}
\mathcal{W}^{n, \stb}_{P, \emptyset} \longrightarrow \Omega^{\infty-1}\mb{hW}^{n}_{\theta}.
\end{equation}
is an acyclic map.
Let $\theta_{2n}$ denote the tangential structure obtained by restricting $\theta$ to $BO(2n)$. 
Let $\MT\theta_{2n}$ denote the \textit{Madsen-Tillmann} spectrum associated to $\theta_{2n}$.
The theorem of Galatius and Randal-Williams is then recovered by combining this with the homotopy equivalence of spectra 
\begin{equation} \label{equation: identification with MT spectrum}
\Sigma^{-1}\MT\theta_{2n} \simeq \mb{hW}^{n}_{\theta}.
\end{equation}
Below, we sketch how this homotopy equivalence of spectra is obtained.
This is basically the same as what was done in \cite[Section 3.1]{MW 07}.
\begin{Construction}
Let $G_{\theta_{2n}}(\R^{\infty})$ be the space of pairs $(V, \hat{\ell}_{V})$ where $V \leq \R^{\infty}$ is a $2n$-dimensional vector subspace and $\hat{\ell}_{V}$ is a $\theta_{2n}$-orientation on $V$. 
We let $U_{2n, \infty} \longrightarrow G_{\theta_{2n}}(\R^{\infty})$ denote the canonical vector bundle. 
The spectrum $\MT\theta_{2n}$ is defined to be the Thom spectrum associated to the virtual vector bundle $-U_{2n, \infty} \longrightarrow G_{\theta_{2n}}(\R^{\infty})$.
There is a map 
\begin{equation} \label{equation: MT comparrison map}
G_{\theta_{2n}}(\R^{\infty}) \longrightarrow G_{\theta}^{\mf}(\R^{\infty})^{n}
\end{equation}
defined by sending $(V, \hat{\ell}_{V})$ to the element $(\R\times V, \hat{\ell}_{\R\times V}, l_{V}, \sigma_{0})$, where: 
\begin{itemize} \itemsep.2cm
\item $\R\times V \; \; \leq \; \; \R\times\R^{\infty}  \cong \R^{\infty}$
is the product vector space,
\item $\hat{\ell}_{\R\times V}$ is the product $\theta$-orientation induced by $\hat{\ell}_{V}$,
\item $l: \R\times V \longrightarrow \R$ is the product projection, 
\item $\sigma_{0}: (\R\times V)\otimes(\R\times V) \longrightarrow \R$ is the trivial bilinear form.
\end{itemize}
By Definition \ref{defn: morse Grassmannian and spectrum}, it follows that $G^{\mf}_{\theta}(\R^{\infty})^{n}$ is precisely the subspace of $G^{\mf}_{\theta}(\R^{\infty})$ consisting of those $(V, \hat{\ell}, l , \sigma)$ such that $l \neq 0$, and so the above map (\ref{equation: MT comparrison map})  is well defined. 
The virtual bundle $-U_{2n+1, \infty} \longrightarrow G^{\mf}_{\theta}(\R^{\infty})^{n}$ pulls back to $-U_{2n, \infty} \longrightarrow G_{\theta}(\R^{\infty})$ under the map (\ref{equation: MT comparrison map}) and so it follows that this map induces a map of spectra 
$\Sigma^{-1}\MT\theta_{2n} \longrightarrow \mb{hW}^{n}_{\theta}$.
A homotopy inverse of (\ref{equation: MT comparrison map}) is defined by sending $(V, \hat{\ell}, l , \sigma)$ to the element $(\Ker(l), \hat{\ell}|_{\Ker(l)})$.
It follows that (\ref{equation: MT comparrison map}) induces the homotopy equivalence of spectra 
$\Sigma^{-1}\MT\theta_{2n} \simeq \mb{hW}^{n}_{\theta}$.
\end{Construction}

\appendix

\section{Proof of Theorem \ref{theorem: homotopy colimit decomposition}} \label{section: proof of theorem homotopy colimit theorem}
In this section we prove Theorem \ref{theorem: homotopy colimit decomposition} which asserts that for all choices of $k$ and $P$, there is a commutative diagram
\begin{equation} \label{equation: main zig zag diagram appendix}
\xymatrix{
\displaystyle{\hocolim_{\mb{t} \in \mathcal{K}^{k-1}}}\mathcal{W}^{k-1}_{P, \mb{t}} \ar[d]  && \mathcal{L}^{k-1} \ar[ll]_{\simeq} \ar[rr]^{\simeq} \ar[d] && \mathcal{D}^{\mf, k-1}_{\theta, P}  \ar[d] \\
\displaystyle{\hocolim_{\mb{t} \in \mathcal{K}^{k-1}}}\mathcal{W}^{\{k, d-k+1\}}_{\locc, \mb{t}}  && \mathcal{L}^{\{k, d-k+1\}}_{\locc} \ar[ll]_{\simeq} \ar[rr]^{\simeq} && \mathcal{D}_{\theta, \locc}^{\mf, \{k, d-k+1\}}
}
\end{equation}
such that the horizontal maps are weak homotopy equivalences.
The proof uses several constructions, all of which are directly analogous to constructions from \cite[Section 5]{MW 07}, where Madsen and Weiss prove the special case of this theorem for $k = -1$. 
We provide only a sketch of the proof, mainly for the purpose of showing that the constructions are well defined for all $k$. 
We refer the reader to the relevant sections of \cite{MW 07} for technical details. 
\subsection{The bottom row} \label{subsection: the bottom row}
In this section we construct the bottom row of (\ref{equation: main zig zag diagram appendix}).
Our first step is to construct an intermediate space sitting in-between $\displaystyle{\hocolim_{\mb{t} \in \mathcal{K}^{k-1}}}\mathcal{W}^{\{k, d-k+1\}}_{\locc, \mb{t}}$ and $\mathcal{D}_{\theta, \locc}^{\mf, \{k, d-k+1\}}$.
\begin{defn} \label{defn: intermediate space 1}
For $\mb{t}$, the space $\mathcal{L}_{\loc, \mb{t}}$ consists of tuples $((\bar{x}; V, \sigma), \delta, \phi)$  where 
\begin{enumerate} \itemsep.2cm
\item[(i)] $(\bar{x}; V, \sigma) \in \mathcal{D}^{\mf}_{\theta, \loc}$; 
\item[(ii)] $\delta: \bar{x} \longrightarrow \{-1, 0, +1\}$ is a function; 
\item[(iii)] $\phi: \mb{t}  \stackrel{\cong} \longrightarrow \delta^{-1}(0)$ is a bijection;
\end{enumerate}
subject to the following condition: the height function $h_{\bar{x}}: \bar{x} \longrightarrow \R$ is bounded below on $\delta^{-1}(+1)$ and is bounded above on $\delta^{-1}(-1)$.
\end{defn}

We need to describe how the correspondence $\mb{t} \mapsto \mathcal{L}_{\loc, \mb{t}}$ defines a contravariant functor on $\mathcal{K}$. 
Let $(j, \varepsilon): \mb{s} \longrightarrow \mb{t}$ be a morphism in $\mathcal{K}$.
The induced map 
$
(j, \varepsilon)^{*}: \mathcal{L}_{\loc, \mb{t}} \longrightarrow \mathcal{L}_{\loc, \mb{s}}
$
sends an element $((\bar{x}; V, \sigma), \delta, \phi) \in \mathcal{L}_{\loc, \mb{t}}$ to $ ((\bar{x}; V, \sigma), \delta', \phi')$ where $\phi' = \phi\circ j$ and 
\begin{equation} \label{equation: functoriality for delta}
\delta'(y) \; = \; 
\begin{cases}
\varepsilon(s) \quad \text{if $y = \phi(s)$ \; where $s \in \mb{s}\setminus j(\mb{t})$,}\\
\delta(y) \quad \text{otherwise.}
\end{cases}
\end{equation}
For each integer $k$ we define $\mathcal{L}^{\{k, d-k+1\}}_{\loc, \mb{t}} \subset \mathcal{L}_{\loc, \mb{t}}$ in the usual way. 
For each $\mb{t} \in \mathcal{K}^{\{k, d-k+1\}}$ there is a forgetful map $\mathcal{L}^{\{k, d-k+1\}}_{\loc, \mb{t}} \longrightarrow \mathcal{D}^{\mf, \{k, d-k+1\}}_{\loc}$. 
Furthermore, for every morphism $(j, \varepsilon): \mb{s} \longrightarrow \mb{t}$ these forgetful maps fit into a commutative diagram
$$
\xymatrix{
\mathcal{L}^{\{k, d-k+1\}}_{\loc, \mb{t}} \ar[dr] \ar[rr] && \mathcal{L}^{\{k, d-k+1\}}_{\loc, \mb{s}} \ar[dl] \\
& \mathcal{D}^{\mf, \{k, d-k+1\}}_{\loc}, &
}
$$
and thus the forgetful maps induce a map, 
$\displaystyle{\hocolim_{\mb{t} \in \mathcal{K}^{\{k, d-k+1\}}}}\mathcal{L}^{\{k, d-k+1\}}_{\loc, \mb{t}} \longrightarrow \mathcal{D}^{\mf, \{k, d-k+1\}}_{\loc}.$
The following proposition is proven in the same way as \cite[Proposition 5.16]{MW 07}.
\begin{proposition} \label{proposition: forgetful map equiv}
For all $k$, the map $\displaystyle{\hocolim_{\mb{t} \in \mathcal{K}^{\{k, d-k+1\}}}}\mathcal{L}^{\{k, d-k+1\}}_{\loc, \mb{t}} \longrightarrow \mathcal{D}^{\mf, \{k, d-k+1\}}_{\loc}$ is a weak homotopy equivalence. 
\end{proposition}

We now need to compare $\displaystyle{\hocolim_{\mb{t} \in \mathcal{K}^{\{k, d-k+1\}}}}\mathcal{L}^{\{k, d-k+1\}}_{\loc, \mb{t}}$ to $\displaystyle{\hocolim_{\mb{t} \in \mathcal{K}^{\{k, d-k+1\}}}}\mathcal{W}^{\{k, d-k+1\}}_{\loc, \mb{t}}$.
\begin{defn} \label{defn: altered version of w-loc}
For $\mb{t}$, define $\widehat{\mathcal{W}}_{\loc, \mb{t}}$ to be the subspace of $(G^{\mf}_{\theta}(\R^{\infty})_{\loc})^{\mb{t}}$, consisting of those $(V, \sigma)$ such that $\bar{d}(i) = \text{index}(\sigma(i))$ for all $i \in \mb{t}$.
As with $\mathcal{W}_{\loc, \mb{t}}$, the correspondence $\mb{t} \mapsto \widehat{\mathcal{W}}_{\loc, \mb{t}}$ defines a contravariant functor on $\mathcal{K}$.
The subspace $\widehat{\mathcal{W}}^{\{k, d-k+1\}}_{\loc, \mb{t}} \subset \widehat{\mathcal{W}}_{\loc, \mb{t}}$ is defined in the same way as before. 
\end{defn}

\begin{proposition} \label{proposition: forget the embedding}
For all $\mb{t} \in \mathcal{K}^{\{k, d-k+1\}}$, the forgetful map 
$$
\mathcal{W}^{\{k, d-k+1\}}_{\loc, \mb{t}} \longrightarrow \widehat{\mathcal{W}}^{\{k, d-k+1\}}_{\loc, \mb{t}}, \quad ((V, \sigma), e) \mapsto (V, \sigma) 
$$
is a weak homotopy equivalence. 
\end{proposition}
\begin{proof}
We first observe that the map $\mathcal{W}^{\{k, d-k+1\}}_{\loc, \mb{t}} \longrightarrow \widehat{\mathcal{W}}^{\{k, d-k+1\}}_{\loc, \mb{t}}$ is a Serre fibration. 
The lemma then follows from the fact the that the fibre over a point $(V, \sigma)$ is given by the space of embeddings 
$D(V^{-})\times_{\mb{t}}D(V^{+}) \longrightarrow \R^{\infty},$
which is a weakly contractible space. 
\end{proof}

We now consider the map 
\begin{equation} \label{equation: natural transformation in t}
\mathcal{L}^{\{k, d-k+1\}}_{\loc, \mb{t}} \longrightarrow \widehat{\mathcal{W}}^{\{k, d-k+1\}}_{\loc, \mb{t}}, \quad \left((\bar{x}; V, \sigma), \delta, \phi\right) \; \mapsto \; \left((V, \sigma)|_{\delta^{-1}(0)}\right)\circ\phi,
\end{equation}
 viewing $\left((V, \sigma)|_{\delta^{-1}(0)}\right)\circ\phi$ as a function from $\mb{t}$ to $G^{\mf}_{\theta}(\R^{\infty})_{\loc}$. 
 This clearly defines a natural transformation of functors on $\mathcal{K}^{\{k, d-k+1\}}$ and thus induces a map 
 $$
 \hocolim_{\mb{t} \in \mathcal{K}^{\{k, d-k+1\}}}\mathcal{L}^{\{k, d-k+1\}}_{\loc, \mb{t}} \longrightarrow  \hocolim_{\mb{t} \in \mathcal{K}^{\{k, d-k+1\}}}\widehat{\mathcal{W}}^{\{k, d-k+1\}}_{\loc, \mb{t}}.
 $$
The following proposition is proven in the same way as \cite[Proposition 5.20]{MW 07}.
\begin{proposition} \label{proposition: level wise weak equivalence}
For all $\mb{t} \in \mathcal{K}^{\{k, d-k+1\}}$, the map 
$\mathcal{L}^{\{k, d-k+1\}}_{\loc, \mb{t}} \longrightarrow \widehat{\mathcal{W}}^{\{k, d-k+1\}}_{\loc, \mb{t}}$
from (\ref{equation: natural transformation in t}) is a weak homotopy equivalence.
\end{proposition}

It follows from the above proposition that the induced map 
 $$
 \hocolim_{\mb{t} \in \mathcal{K}^{\{k, d-k+1\}}}\mathcal{L}^{\{k, d-k+1\}}_{\loc, \mb{t}} \longrightarrow  \hocolim_{\mb{t} \in \mathcal{K}^{\{k, d-k+1\}}}\widehat{\mathcal{W}}^{\{k, d-k+1\}}_{\loc, \mb{t}}
 $$
 is a weak homotopy equivalence. 
 Combining this with Propositions \ref{proposition: forgetful map equiv} and \ref{proposition: forget the embedding} yields the zig-zag of weak homotopy equivalences
 $$
 \xymatrix{
 \mathcal{D}^{\mf, \{k, d-k+1\}}_{\loc} & \displaystyle{\hocolim_{\mb{t} \in \mathcal{K}^{\{k, d-k+1\}}}}\mathcal{L}^{\{k, d-k+1\}}_{\loc, \mb{t}} \ar[l]_{\simeq \ \ \ \ } \ar[r]^{\simeq} & \displaystyle{\hocolim_{\mb{t} \in \mathcal{K}^{\{k, d-k+1\}}}}\widehat{\mathcal{W}}^{\{k, d-k+1\}}_{\loc, \mb{t}}  & \displaystyle{\hocolim_{\mb{t} \in \mathcal{K}^{\{k, d-k+1\}}}}\mathcal{W}^{\{k, d-k+1\}}_{\loc, \mb{t}} \ar[l]_{\simeq}
 }
 $$
 This establishes the bottom row of weak equivalences from the statement of Theorem \ref{theorem: homotopy colimit decomposition}.

\subsection{The long trace} \label{subsection: the long trace}
We now proceed to establish the top row of the diagram (\ref{equation: main zig zag diagram appendix}).  
This will require the use of some new constructions to be carried out over the course of the next three subsections. 
We begin by recalling a construction from \cite[Section 5.2]{MW 07}.
Let $(V, \sigma) \in G^{\mf}_{\theta}(\R^{\infty})_{\loc}$.  
Observe that the function 
\begin{equation} \label{equation: morse vector space}
f_{V}: V \longrightarrow \R, \quad f_{V}(v) := \sigma(v, v)
\end{equation}
is a Morse function on $V$ with exactly one critical point at the origin with index equal to $\text{index}(\sigma)$.
\begin{defn} \label{defn: sadle set}
Let $(V, \sigma) \in G^{\mf}_{\theta}(\R^{\infty})_{\loc}$. 
The subspace $\sdl(V, \sigma) \subset V$ is defined by 
$$
\sdl(V, \varrho) = \{v \in V \; | \; ||v_{+}||^{2} ||v_{-}||^{2} \leq 1 \;\},
$$
where $v = (v_{-}, v_{+})$ is the coordinate representation of $v$ using the splitting $V = V^{-}\oplus V^{+}$ given by the negative and positive eigenspaces of $\sigma$.
The norm on $V$ is the one induced by the Euclidean inner product on the ambient space in which the subspace $V$ is contained.
\end{defn}

Given $(V, \sigma) \in G^{\mf}_{\theta}(\R^{\infty})_{\loc}$, the formula
$$
v \mapsto \left(f_{V}(v), \; ||v_{-}||v_{+}, \;  ||v_{-}||^{-1}v_{-} \right)
$$
defines a smooth embedding
\begin{equation} \label{equation: embedding into Disk sphere}
\phi: \sdl(V, \sigma)\setminus V^{+} \longrightarrow \R\times D(V^{+})\times S(V^{-}),
\end{equation}
with complement 
$[0, \infty)\times\{0\}\times S(V^{-}).$
It respects boundaries and is a map over $\R$, where we use the restriction of $f_{V}$ on the source and the function $(t, x, y) \mapsto t$ on the target. 

For $\mb{t} \in \mathcal{K}$, let $(V, \sigma) \in \widehat{\mathcal{W}}_{\loc, \mb{t}}$. 
We may form the space 
$$
\sdl(V, \sigma; \mb{t}) \; = \; \bigsqcup_{i \in \mb{t}}\sdl(V(i), \sigma(i)). 
$$
We let 
$$V^{\pm}(\mb{t}) \subset \sdl(V, \sigma; \mb{t})$$ 
denote the subspace given by the disjoint union $\coprod_{i \in \mb{t}}V^{\pm}(i)$. 
By applying the embedding $\phi$ with $i \in \mb{t}$
over each component 
$$\sdl(V(i), \sigma(i)) \subset \sdl(V, \sigma, \mb{t}),$$ 
we obtain an embedding
$$
\phi_{\mb{t}}: \sdl(V, \sigma; \mb{t})\setminus V^{+}(\mb{t}) \longrightarrow \R\times D(V^{+})\times_{\mb{t}}S(V^{-}),
$$
which has complement $[0, \infty)\times\left(\{0\}\times_{\mb{t}}S(V^{-})\right)$.

\begin{Construction} \label{Construction: long trace}
Let $(M, (V, \sigma), e) \in \mathcal{W}_{P, \mb{t}}$. 
The submanifold 
$$\textstyle{\Trc}(e) \subset \R\times\R^{\infty}$$ 
is defined to be the pushout of the diagram
\begin{equation} \label{equation: long trace}
\xymatrix{
\sdl(V, \sigma; \mb{t})\setminus V^{+}(\mb{t}) \ar[d] \ar[rr] && \R\times M\setminus\left([0, \infty)\times e(\{0\}\times_{\mb{t}}S(V^{-}))\right) \\
\sdl(V, \sigma; \mb{t}). && 
}
\end{equation}
The $\theta$-orientation $\hat{\ell}_{V}$ on $V$ together with the $\theta$-structure $\hat{\ell}_{M}$ on $M$ determines a $\theta$-structure $\hat{\ell}_{\Trc(e)}$ on $\Trc(e)$. 
The height function 
$$\Trc(e) \hookrightarrow \R\times\R^{\infty} \longrightarrow \R$$ 
is Morse. 
It has one critical point for each $i \in \mb{t}$, with index equal to $\dim(V^{-}(i))$, all of which have value $0$.
It follows that $\Trc(e)$ equipped with its $\theta$-structure $\hat{\ell}_{\Trc(e)}$ determines an element of the space $\mathcal{W}_{\theta, P}$. 
The correspondence $(M, (V, \sigma), e) \mapsto \Trc(e)$ defines a map 
\begin{equation} \label{equation: long trace map}
\textstyle{\Trc}: \mathcal{W}_{P, \mb{t}} \longrightarrow \mathcal{D}^{\mf}_{\theta, P}.
\end{equation}
It is easily verified that for each integer $k$, the map $\Trc$ restricts to a map, 
$
\textstyle{\Trc_{k}}: \mathcal{W}^{k}_{P, \mb{t}} \longrightarrow \mathcal{D}^{\mf, k}_{\theta, P}.
$
\end{Construction}

\subsection{Couplings} \label{subsection: couplings}
The following definition is analogous to \cite[Definition 5.24]{MW 07}.
First we must fix some notation.
If $(\bar{x}; (V, \sigma)) \in \mathcal{D}^{\mf}_{\theta, \loc}$, we may form the space 
$$
\sdl(V, \sigma; \bar{x}) \; = \; \bigsqcup_{x \in \bar{x}}\sdl(V(x), \sigma(x)).
$$
Let $h_{\bar{x}}: \bar{x} \longrightarrow \R$ denote the height function, i.e.\ the projection $\bar{x} \hookrightarrow \R\times\R^{\infty} \stackrel{\text{proj}} \longrightarrow \R$. 
The function 
$$f_{V, \bar{x}}: \sdl(V, \sigma; \bar{x}) \longrightarrow \R$$
is defined by the formula
$$f_{V, \bar{x}}(z) \; = \; f_{V(x)}(z) + h_{\bar{x}}(x),$$
if $z$ is contained in the component $\sdl(V, \sigma; \bar{x})|_{x}$, where $f_{V(x)}: \sdl(V(x), \sigma(x)) \longrightarrow \R$ is the function from (\ref{equation: morse vector space}).

\begin{defn} \label{defn: couplings}
Let $W \in \mathcal{D}^{\mf}_{\theta, P}$ and $(\bar{x}; V, \sigma) \in \mathcal{D}^{\mf}_{\theta, \loc}$ be elements with $\mb{L}(W) = (\bar{x}; V, \sigma)$, where recall that $\mb{L}$ is the localization map (\ref{equation: localization map}). 
A \textit{coupling} between $W$ and $(\bar{x}; V, \sigma)$ is an embedding 
$$\lambda: \sdl(V, \sigma; \bar{x}) \longrightarrow \R\times\R^{\infty}$$
that satisfies the following conditions:
\begin{enumerate} \itemsep.2cm
\item[(i)] $\lambda(\sdl(V, \sigma; \bar{x})) \subset W$;
\item[(ii)] $f_{V, \bar{x}} = h_{W}\circ\lambda$, where recall that $h_{W}: W \longrightarrow \R$ is the height function on $W$;
\item[(iii)] $\hat{\ell}_{\sdl(V, \sigma)} = \lambda^{*}\hat{\ell}_{W}$. 
\end{enumerate}
\end{defn}

Using the definition of a coupling we define a new space as follows:
\begin{defn} \label{defn: the space L}
The space $\mathcal{L}_{\theta, P}$ consists of tuples $(W, (\bar{x}; V, \sigma), \lambda)$ where $W \in \mathcal{D}^{\mf}_{\theta, P}$ and $(\bar{x}; V, \sigma) \in \mathcal{D}^{\mf}_{\loc}$ are elements with $\mb{L}(W) = (\bar{x}; V, \sigma)$, and $\lambda$ is a coupling between these two elements.
For each integer $k$, we define $\mathcal{L}^{k}_{\theta, P} \subset \mathcal{L}_{\theta, P}$ to be the subspace consisting of those $(W, (\bar{x}; V, \sigma), \lambda)$ for which $W \in \mathcal{D}^{\mf, k}_{\theta, P}$ and $(\bar{x}; V, \sigma) \in \mathcal{D}^{\mf, k}_{\theta, \loc}$.
\end{defn}

The following proposition is proven in the same way as \cite[Proposition 5.28]{MW 07}.
\begin{proposition} \label{proposition: forget the coupling}
For all $k$ the forgetful map 
$$
\mathcal{L}^{k}_{\theta, P} \longrightarrow \mathcal{D}^{\mf, k}_{\theta, P}, \quad (W, (\bar{x}; V, \sigma), \lambda) \mapsto W
$$
is a weak homotopy equivalence. 
\end{proposition}

We need to define one more space sitting in between the spaces $\mathcal{L}^{k}_{\theta, P}$ and $\displaystyle{\hocolim_{\mb{t} \in \mathcal{K}^{k}}}\mathcal{W}^{k}_{P, \mb{t}}$. 
The following definition is essentially the same as \cite[Definition 5.30]{MW 07}. 
\begin{defn} \label{defn: homotopy colim decomp of L}
Fix $\mb{t} \in \mathcal{K}$. 
The space $\mathcal{L}_{P, \mb{t}}$ consists of tuples 
$\left(W, (\bar{x}; V, \sigma), \lambda, \delta, \phi\right)$
where:
\begin{itemize} \itemsep.2cm
\item $(W, (\bar{x}; V, \sigma), \lambda) \in \mathcal{L}_{\theta, P}$;
\item $\delta: \bar{x} \longrightarrow \{-1, 0, +1\}$ is a function;
\item $\phi: \mb{t} \stackrel{\cong} \longrightarrow \delta^{-1}(0)$ is a function over the set $\{0, 1, \dots, d+1\}$;
\end{itemize}
subject to the following condition:
the height function,
$h_{\bar{x}}: \bar{x} \longrightarrow \R,$
admits a lower bound on $\delta^{-1}(+)$ and an upper bound on $\delta^{-1}(-1)$. 
\end{defn}
We need to describe how the correspondence $\mb{t} \mapsto \mathcal{L}_{P, \mb{t}}$ defines a contravariant functor on $\mathcal{K}$. 
Let $(j, \varepsilon): \mb{s} \longrightarrow \mb{t}$ be a morphism in $\mathcal{K}$.
The induced map 
$
(j, \varepsilon)^{*}: \mathcal{L}_{P, \mb{t}} \longrightarrow \mathcal{L}_{P, \mb{s}}
$
sends an element $(W, (\bar{x}; V, \sigma), \lambda, \delta, \phi) \in \mathcal{L}_{P, \mb{t}}$ to $(W, (\bar{x}; V, \sigma), \lambda, \delta', \phi')$ where $\phi' = \phi\circ j$ and 
\begin{equation} \label{equation: functoriality for delta}
\delta'(y) \; = \; 
\begin{cases}
\varepsilon(s) \quad \text{if $y = \phi(s)$ \; where $s \in \mb{s}\setminus j(\mb{t})$,}\\
\delta(y) \quad \text{otherwise.}
\end{cases}
\end{equation}
For each integer $k$ we define $\mathcal{L}^{k}_{P, \mb{t}} \subset \mathcal{L}_{P, \mb{t}}$ in the usual way. 
For each $\mb{t} \in \mathcal{K}^{k}$ there is a forgetful map $\mathcal{L}^{k}_{P, \mb{t}} \longrightarrow \mathcal{L}^{k}_{\theta, P}$. 
Furthermore, for every the morphism $(j, \varepsilon): \mb{s} \longrightarrow \mb{t}$ these forgetful maps fit into a commutative diagram
$$
\xymatrix{
\mathcal{L}^{k}_{P, \mb{t}} \ar[dr] \ar[rr] && \mathcal{L}^{k}_{P, \mb{s}} \ar[dl] \\
& \mathcal{L}^{k}_{\theta, P}, &
}
$$
and thus the forgetful maps induce a map, 
$\displaystyle{\hocolim_{\mb{t} \in \mathcal{K}}}\mathcal{L}^{k}_{P, \mb{t}} \longrightarrow \mathcal{L}^{k}_{\theta, P}.$
The following proposition is proven by running the same argument from the proof of \cite[Proposition 5.16]{MW 07}.
\begin{proposition} \label{proposition: equiv of the l's}
For all $k$, the map, 
$\displaystyle{\hocolim_{\mb{t} \in \mathcal{K}^{k}}}\mathcal{L}^{k}_{P, \mb{t}} \longrightarrow \mathcal{L}^{k}_{\theta, P},$
is a weak homotopy equivalence. 
\end{proposition}

We have weak homotopy equivalences
$$\xymatrix{
\mathcal{D}^{\mf, k}_{\theta, P} & \mathcal{L}^{k}_{\theta, P} \ar[l]_{ \ \ \ \ \simeq} & \displaystyle{\hocolim_{\mb{t} \in \mathcal{K}^{k}}}\mathcal{L}^{k}_{P, \mb{t}}. \ar[l]_{\simeq \ \ \ \ \ }
}
$$
It remains to construct a weak homotopy equivalence between $\displaystyle{\hocolim_{\mb{t} \in \mathcal{K}^{k}}}\mathcal{L}^{k}_{P, \mb{t}}$ and $\displaystyle{\hocolim_{\mb{t} \in \mathcal{K}^{k}}}\mathcal{W}^{k}_{P, \mb{t}}$.
Defining the required map will require a new construction which we carry out in the following subsection. 
This construction is essentially the same as what was done in \cite[Pages 893-894]{MW 07}.

\subsection{Regularization} \label{subsection: regularization}
Choose once and for all a diffeomorphism $\psi: \R \longrightarrow (-\infty, 0)$ such that $\psi(t) = t$ for $t < -\tfrac{1}{2}$, and a smooth nondecreasing function $\varphi: [0, 1] \longrightarrow [0,1]$ such that $\varphi(x) = x$ for $x$ close to $0$, and $\varphi(x) = 1$ for $x$ close to $1$. 
Let 
$$
\psi_{x}(t) = \varphi(x)t + (1 - \varphi(x))\psi(t)
$$
for $x \in [0,1]$. 
Then $\psi_{0} = \psi$ embeds $\R$ in $\R$ with image $(-\infty, 0)$, whereas each $\psi_{x}$ for $x > 0$ is a diffeomorphism $\R \longrightarrow \R$. 
Let $(V, \sigma) \in G_{\theta}^{\mf}(\R^{\infty})_{\loc}$. 
Recall form (\ref{equation: morse vector space}) the Morse function $f_{V}: V \longrightarrow \R$.
We define functions 
$$\begin{aligned}
f^{+}_{V}: \sdl(V, \varrho)\setminus V^{+} &\longrightarrow \R, \\
 f^{-}_{V}: \sdl(V, \varrho)\setminus V^{-} &\longrightarrow \R,
 \end{aligned}$$
by the formulae
\begin{equation} \label{equation: regularization functions}
\begin{aligned}
f^{+}_{V}(v) &= \psi^{-1}_{x}(t), \\
f^{-}_{V}(v) &= (-\psi_{x})^{-1}(-t),
\end{aligned}
\end{equation}
where $t = f_{V}(v)$ and $x = ||v_{-}||^{2}||v_{+}||^{2}$. 
These functions agree with $f_{V}$ on open subsets that contain the entire boundary and the sets 
$$\begin{aligned}
\{\omega \in \sdl(V, \varrho) | f_{V}(\omega) \leq -1\}, \\ 
\{\omega \in \sdl(V, \varrho) | f_{V}(\omega) \geq +1\},
\end{aligned}$$
respectively. 
The following proposition can be verified by hand.
\begin{proposition} \label{proposition: regularized functions}
The functions $f^{\pm}_{V}$ defined in (\ref{equation: regularization functions}) are proper submersions. 
\end{proposition}

Let $(\bar{x}; (V, \sigma)) \in \mathcal{D}^{\mf}_{\theta, \loc}$.
The functions (\ref{equation: regularization functions}) determine functions 
\begin{equation} \label{equation: regularized functions 2}
\begin{aligned}
f^{+}_{V, \bar{x}}: \sdl(V, \sigma; \bar{x})\setminus V^{+} &\longrightarrow \R, \\
 f^{-}_{V, \bar{x}}: \sdl(V, \sigma; \bar{x})\setminus V^{-} &\longrightarrow \R,
 \end{aligned}
\end{equation}
defined by 
$$
f^{\pm}_{V, \bar{x}}(z) \; = \; f^{\pm}_{V}(z) + h(x)
$$
for $z$ on $\sdl(V, \varrho)|_{x}$ for each $x \in \bar{x}$.
By Proposition \ref{proposition: regularized functions} it follows that the functions $f^{\pm}_{V, \bar{x}}$ are proper submersions as well. 
In the construction below, we use these functions to define a map $\mathcal{L}_{P, \mb{t}} \longrightarrow \mathcal{W}_{P, \mb{t}}$.

\begin{Construction} \label{construction: regularization}
Let $(W, (\bar{x}; V, \sigma), \lambda, \delta, \phi)$ of $\mathcal{L}_{P, \mb{t}}$.
Let $(V_{+1}, \sigma_{+1}), (V_{-1}, \sigma_{-1})$, and $(V_{0}, \sigma_{0})$ denote the restrictions of $(V, \sigma)$ to the subsets 
$$\delta^{-1}(+1), \; \delta^{-1}(-1), \; \delta^{-1}(0) \; \subset \; \bar{x}.$$
Let us denote the submanifold
$$W^{\rg} = W\setminus\lambda(V^{+}_{+1}\cup V^{+}_{0}\cup V^{-}_{-1}).$$
We define a function 
$f^{\rg}: W^{\rg} \longrightarrow \R$ by
\begin{equation}
f^{\rg}(z) \; = \; 
\begin{cases}
f(z) &\quad \text{if $z \in \Image(\lambda)$,}\\
f^{+}_{V, \bar{x}}(v) &\quad \text{if $z = \lambda(v)$ and $v \in V_{+}\cup V_{0}$,}\\
f^{-}_{V, \bar{x}}(v) &\quad \text{if $z = \lambda(v)$ and $v \in V_{-}$.}
\end{cases}
\end{equation}
Let $j_{W^{\rg}}: W \longrightarrow (-\infty, 0]\times\R^{\infty-1}$ denote the composite 
$$W \hookrightarrow \R\times(-\infty, 0]\times\R^{\infty-1} \longrightarrow (-\infty, 0]\times\R^{\infty-1}.$$
We re-embed $W^{\rg}$ into $\R\times(-1, 0]\times(-1, 1)^{\infty-1}$ using the product map
$$
f^{\rg}\times j_{W^{\rg}}: W^{\rg} \longrightarrow \R\times(-\infty, 0]\times\R^{\infty-1},
$$
and we let $\widetilde{W} \subset \R\times(-\infty, 0]\times\R^{\infty-1}$ denote its image. 
By construction, the height function 
$$
h_{\widetilde{W}}: \widetilde{W} \longrightarrow \R
$$
is a proper submersion, and thus
$$M := h^{-1}_{\widetilde{W}}(0)$$
is a compact $d$-dimensional submanifold of $\{0\}\times(-\infty, 0]\times\R^{\infty-1}$ with $\partial M = P$.  
The restriction of $\lambda$ gives an embedding 
$$
e: D(V^{+}_{0})\times_{\mb{t}}S(V^{-}_{0}) \longrightarrow M.
$$
This embedding together with $M$ and $(V_{0}, \sigma_{0})$ determines an element 
$$(M, (V_{0}, \sigma_{0}), e) \in \mathcal{W}_{P, \mb{t}}.$$ 
The correspondence, 
$\left(W, (\bar{x}; V, \sigma), \lambda, \delta, \phi\right) \; \mapsto \; \left(M, (V_{0}, \sigma_{0}), e\right),$
defines a map 
\begin{equation} \label{equation: regularization map}
\mathcal{L}_{P, \mb{t}} \longrightarrow \mathcal{W}_{P, \mb{t}}.
\end{equation}
It is easily verified that for each $k \in \Z_{\geq 0}$ this map restricts to a map,
$
\mathcal{L}^{k}_{P, \mb{t}} \longrightarrow \mathcal{W}^{k}_{P, \mb{t}}.
$
By what was observed in \cite{MW 07}, this map defines a natural transformation of functors on $\mathcal{K}$. 
\end{Construction}

The following proposition is the same as \cite[Proposition 5.36]{MW 07}.
\begin{proposition} \label{proposition: levelwise equivalence t-W}
For all $k$ and $\mb{t}$, the map 
$\mathcal{L}^{k}_{P, \mb{t}} \longrightarrow \mathcal{W}^{k}_{P, \mb{t}}$ is a weak homotopy equivalence. 
\end{proposition}
\begin{proof}
We use Construction \ref{Construction: long trace} to define a homotopy inverse. 
Indeed, we define a map 
\begin{equation} \label{equation: homotopy inverse map}
\mathcal{W}^{k}_{P, \mb{t}} \longrightarrow \mathcal{L}^{k}_{P, \mb{t}}
\end{equation}
by sending $(M, (V, \sigma), e)$ to the tuple $\left(\Trc(M, (V, \sigma), e), (\bar{x}; V', \sigma'), \lambda, \delta, \phi\right)$,
where:
\begin{enumerate} \itemsep.3cm
\item[(i)] $\Trc(M, (V, \sigma), e)$ is defined as in Construction \ref{Construction: long trace}.
\item[(ii)] $\bar{x} \subset  \Trc(M, (V, \sigma), e)$ is the set of critical points of the height function 
$$\Trc(M, (V, \sigma), e) \longrightarrow \R.$$ 
Induced by the construction of $\Trc(M, (V, \sigma), e)$ there is natural bijection, 
$\bar{x} \stackrel{\cong} \longrightarrow \mb{t}.$

\item[(iii)]  $(V', \sigma')$ is given by the composite 
$\bar{x} \stackrel{\cong} \longrightarrow \mb{t} \stackrel{(V, \sigma)} \longrightarrow G^{\mf}_{\theta}(\R^{\infty})_{\loc}$, where the first map is the bijection from (ii). 

\item[(iv)] $\lambda$ is the coupling coming from the embedding $\sdl(V, \sigma; \mb{t}) \hookrightarrow \Trc(M, (V, \sigma), e)$ used in the construction of 
$\Trc(M, (V, \sigma), e)$;

\item[(v)] $\delta: \bar{x} \longrightarrow \{0\}$ is the constant function. 

\item[(vi)] $\phi: \delta^{-1}(0) = \bar{x} \stackrel{\cong} \longrightarrow \mb{t}$ is the bijection from (iii). 
\end{enumerate}
By what was observed above, it follows that the map (\ref{equation: homotopy inverse map}) defines a homotopy inverse to (\ref{equation: homotopy inverse map}) (see also \cite[Proposition 5.36]{MW 07}).
\end{proof}

The above proposition implies that the natural transformation of (\ref{equation: regularization map}) induces a weak homotopy equivalence 
$$
\hocolim_{\mb{t} \in \mathcal{K}^{k}}\mathcal{L}^{k}_{P, \mb{t}} \; \stackrel{\simeq} \longrightarrow \;\hocolim_{\mb{t} \in \mathcal{K}^{k}}\mathcal{W}^{k}_{P, \mb{t}}.
$$
Putting this together with the weak homotopy equivalences from Propositions \ref{proposition: forget the coupling} and \ref{proposition: equiv of the l's} we obtain the zig-zag of weak homotopy equivalencies
$$
\xymatrix{
\mathcal{D}^{\mf, k}_{\theta, P} && \mathcal{L}_{\theta, P}^{k} \ar[ll]_{\simeq} && \displaystyle{\hocolim_{\mb{t} \in \mathcal{K}^{k}}}\mathcal{L}^{k}_{P, \mb{t}} \ar[ll]_{\simeq} \ar[rr]^{\simeq} && \displaystyle{\hocolim_{\mb{t} \in \mathcal{K}^{k}}}\mathcal{W}^{k}_{P, \mb{t}}.
}
$$
This establishes the top row of the diagram (\ref{equation: main zig zag diagram appendix}). 
Proving that it is commutative is a straight-forward verification done by tracing through the constructions. 
This completes the proof of Theorem \ref{theorem: homotopy colimit decomposition} which was the objective of this appendix.

\end{document}